\documentclass[12pt]{amsart}
\usepackage{graphicx,amssymb}
\usepackage[usenames]{color}
\AtEndDocument{\vfill\eject\batchmode}
\newtheorem{result}{Theorem}
\newtheorem*{problem}{Problem}

\newtheorem{theorem}{Theorem}[section]
\newtheorem{lemma}[theorem]{Lemma}
\newtheorem{proposition}[theorem]{Proposition}
\newtheorem{corollary}[theorem]{Corollary}

\newtheorem{def/prop}[theorem]{Definition/Proposition}

\newtheorem{prop}[theorem]{Proposition}

\theoremstyle{definition}
\newtheorem{definition}[theorem]{Definition}
\newtheorem{remark}[theorem]{Remark}

\newtheorem{example}{Example}[section]
\newtheorem{subexample}{Example}[example]

\DeclareSymbolFont{bbold}{U}{bbold}{m}{n}
\DeclareSymbolFontAlphabet{\mathbbold}{bbold}

\def\BZero{\mathchoice{\scalebox{1.1}{$\displaystyle\mathbbold 0$}}{\scalebox{1.1}{$\textstyle\mathbbold 0$}}{\scalebox{1.1}{$\scriptstyle\mathbbold 0$}}{\scalebox{1.1}{$\scriptscriptstyle\mathbbold 0$}}}
\def\BOne{\mathchoice{\scalebox{1.16}{$\displaystyle\mathbbold 1$}}{\scalebox{1.16}{$\textstyle\mathbbold 1$}}{\scalebox{1.16}{$\scriptstyle\mathbbold 1$}}{\scalebox{1.16}{$\scriptscriptstyle\mathbbold 1$}}}

\def\bbc{{\mathbb C}}

\def\bbn{{\mathbb N}}

\def\bbp{{\mathbb P}}
\def\bbq{{\mathbb Q}}
\def\bbr{{\mathbb R}}

\def\bbt{{\mathbb T}}

\def\bbz{{\mathbb Z}}

\def\gra{\alpha}

\def\grd{\delta}
\def\gre{\epsilon}

\def\grg{\gamma}

\def\grl{\lambda}

\def\gro{\omega}

\def\grr{\rho}
\def\grs{\sigma}
\def\grt{\tau}

\def\grL{\Lambda}

\def\grS{\Sigma}

\def\bfk{{\bf k}}
\def\bfl{{\bf l}}

\def\bfr{{\bf r}}

\def\bfu{{\bf u}}

\def\bfx{{\bf x}}
\def\bfy{{\bf y}}

\def\cala{{\mathcal A}}
\def\calb{{\mathcal B}}
\def\calc{{\mathcal C}}

\def\cale{{\mathcal E}}

\def\calh{{\mathcal H}}
\def\cali{{\mathcal I}}

\def\calk{{\mathcal K}}
\def\call{{\mathcal L}}

\def\calo{{\mathcal O}}

\def\cals{{\mathcal S}}
\def\calt{{\mathcal T}}

\def\gg{{\mathfrak g}}

\def\gt{{\mathfrak t}}

\def\<{\langle}
\def\>{\rangle}
\def\ra#1{\to}
\def\a{\alpha}

\def\eps{\varepsilon}

\def\Scal{\mathit{Scal}}

\newcommand{\x}{r}
\newcommand{\Aut}{\mathrm{Aut}}
\newcommand{\BT}{\calb\calt^n}
\newcommand{\BM}{\calb_n}
\newcommand{\Symp}{\mathrm{Symp}}
\DeclareMathOperator{\grad}{grad}
\DeclareMathOperator{\tr}{tr}

\begin{document}

\title{The K\"ahler geometry of Bott manifolds}

\author[Charles Boyer]{Charles P. Boyer}
\address{Charles P. Boyer, Department of Mathematics and Statistics,
University of New Mexico, Albuquerque, NM 87131.}
\email{cboyer@math.unm.edu} 
\author[David Calderbank]{David M. J. Calderbank}
\address{David M. J. Calderbank, Department of Mathematical Sciences,
University of Bath, Bath BA2 7AY, UK}
\email{D.M.J.Calderbank@bath.ac.uk}
\author[Christina T{\o}nnesen-Friedman]{Christina W. T{\o}nnesen-Friedman}
\address{Christina W. T{\o}nnesen-Friedman, Department of Mathematics, Union
College, Schenectady, New York 12308, USA }
\email{tonnesec@union.edu}
\thanks{The first and third author were partially supported by grants from the
Simons Foundation, CPB by (\#245002) and (\#519432), and CWT-F by (\#208799) and (\#422410)}
\date{\today}
\begin{abstract}
We study the K\"ahler geometry of stage $n$ Bott manifolds, which can be
viewed as $n$-dimensional generalizations of Hirzebruch surfaces. We show,
using a simple induction argument and the generalized Calabi construction from
\cite{ACGT04,ACGT11}, that any stage $n$ Bott manifold $M_n$ admits an extremal
K\"ahler metric. We also give necessary conditions for $M_n$ to admit a
constant scalar curvature K\"ahler metric. We obtain more precise results for
stage $3$ Bott manifolds, including in particular some interesting relations
with $c$-projective geometry and some explicit examples of almost K\"ahler
structures.

To place these results in context, we review and develop the topology, complex
geometry and symplectic geometry of Bott manifolds. In particular, we study
the K\"ahler cone, the automorphism group and the Fano condition.  We also
relate the number of conjugacy classes of maximal tori in the
symplectomorphism group to the number of biholomorphism classes compatible
with the symplectic structure.
\end{abstract}
\maketitle
\vspace{-7mm}

\addtocontents{toc}{\protect\setcounter{tocdepth}{0}}

\tableofcontents

\section*{Introduction}\label{intro}
\addtocontents{toc}{\protect\setcounter{tocdepth}{2}}

The purpose of this paper is to present and develop the K\"ahler geometry of a
class of toric complex manifolds known as {\it Bott manifolds}, with one
ultimate goal being to understand their extremal K\"ahler metrics.

The introduction to Grossberg's PhD thesis~\cite{Gro91} describes a conjecture
by Bott, in a 1989 letter to Atiyah, that the well-studied Bott--Samelson
manifolds~\cite{BoSa58} ``should be realizible as some kind of tower of
projectivized vector bundles''. The thesis then proved the conjecture, showing
that each bundle in the tower is a $\bbc\bbp^1$-bundle with a circle action,
and Grossberg named such iterated $\bbc\bbp^1$-bundles {\it Bott towers}.
Their study was taken up by Grossberg and Karshon~\cite{GrKa94}, who proved
that isomorphism classes of Bott towers of complex dimension $n$ are in
bijection with $\bbz^{n(n-1)/2}$. More precisely, given an integer-valued
$n\times n$ lower triangular unipotent matrix $A=(A^j_i)$, they constructed a
Bott tower $M_n(A)$ as quotient of $(\bbc^2_*)^n$ (with $\bbc^2_*=
\bbc^2\setminus\{0\}$) by the action of a complex $n$-torus determined by
$A$. Then they proved that there is a unique such Bott tower $M_n(A)$ in each
isomorphism class.

A (stage $n$) Bott manifold $M_n$ is a complex $n$-manifold biholomorphic to
the total space of a Bott tower. Bott manifolds are natural generalizations of
Hirzebruch surfaces: as shown by Masuda and Panov~\cite{MaPa08} they are
precisely the toric complex manifolds whose fan $\grS$ is the cone over an
$n$-cross (i.e., combinatorially dual to an $n$-cube). Thus the primitive
generators of the rays of $\grS$ in the toric real Lie algebra $\grt$ come in
opposite pairs $u_j,v_j:j\in\{1,\ldots, n\}$, and a subset of generators spans
a cone of $\grS$ iff it contains no opposite pairs.

The close relation to Hirzebruch surfaces suggests that one may be able to
systematically explore the existence of extremal K\"ahler metrics, elegantly
constructed by Calabi~\cite{Cal82} when $n=2$, for arbitrary Bott manifolds.

Generally, let $\calk(M_n)$ be the K\"ahler cone of a complex manifold $M_n$,
and let $\cale(M_n)$ be the subset of K\"ahler classes which contain an
extremal K\"ahler metric.
\begin{problem}
Describe the extremal K\"ahler cone $\cale(M_n)$. In particular, \textup{(1)}
is $\cale(M_n)$ nonempty, and if so, \textup{(2)} is
$\cale(M_n)=\calk(M_n)\,$\textup?
\end{problem}

By a well-known result of LeBrun and Simanca \cite{LeSi93b}, $\cale(M_n)$ is
open in $\calk(M_n)$, a result which fails for the subclass of constant scalar
curvature (CSC) K\"ahler metrics.

Question (2) is known in the affirmative for Hirzebruch surfaces~\cite{Cal82}
(as already noted), and question (1) is true for toric surfaces in
general~\cite{WaZh11}. In another direction, Zhou and Zhu \cite{ZhZh08a}
proved that the existence of an extremal K\"ahler metric in the class $c_1(L)$
implies the K-polystability of the polarized toric complex manifold $(M_n,L)$,
and this has been generalized to arbitrary polarized complex manifolds by
Stoppa and Sz\'ekelyhidi \cite{StSz11} and Mabuchi \cite{Mab14}. Now Donaldson
\cite{Don02} and Wang and Zhou \cite{WaZh11} have given examples of K-unstable
K\"ahler classes on certain smooth toric surfaces, so they admit no extremal
metrics, and hence question (2) is false for toric surfaces in general. We
refer to \cite{WaZh14} for a recent survey.

For Bott manifolds, our main general result is the following.

\begin{result}\label{existence}
Let $M_n(A)$ be the Bott tower corresponding to the matrix $A$.
\begin{enumerate}
\item An invariant $\bbr$-divisor $D$ is ample if and only if its support
  function $\psi_D$ satisfies
\[
\psi_D(u_j)+\psi_D(v_j)>\psi_D\biggl(-\sum_{i=j+1}^nA^j_iv_i\biggr)
\]
for all opposite pairs $v_j,u_j$ of generators of the fan. In particular, if
$A^j_i\leq 0$ for all $i>j\geq 1$ then the ample cone $\cala(M_n)$, and hence
the K\"ahler cone $\calk(M_n)$, is the entire first orthant with respect to the natural basis of toric semi-ample divisors.
\item The extremal K\"ahler cone $\cale(M_n(A))$ is a nonempty open cone in
  $\calk(M_n(A))$, i.e., $M_n(A)$ admits extremal K\"ahler metrics.
\item If the elements below the diagonal in any one row of the matrix $A$ all have the same
  sign and are not all zero \textup(in particular, if $A^1_2\neq 0$\textup),
  then $M_n(A)$ does not admit a CSC K\"ahler metric.
\end{enumerate}
\end{result}
The first part of this result follows from Batyrev's theorem for toric
varieties \cite{Bat91,CoRe09} and the structure of the fan of $M_n(A)$.  This
result has been well known since Calabi \cite{Cal82} for $\bbc\bbp^1$ and
Hirzebruch surfaces. In particular, with respect to a natural basis, the
K\"ahler cone of any Hirzebruch surface is the entire first quadrant; however,
this is not true generally for every Bott manifold. In fact, it breaks down at
stage $3$ as we shall show.  We also mention recent related papers
\cite{Cha17a,Cha17b} which study the Mori cone, and toric degenerations,
respectively.

For the second part, we apply a theorem of \cite{ACGT11} obtained using the
generalized Calabi construction of \cite{ACGT04}.

The third part uses the generalization of Matsushima's theorem by Lichnerowicz
\cite{Lic58} that the existence of a constant scalar curvature K\"ahler metric
implies that $\Aut(M_n(A))$ is reductive. The criterion then follows from a
balancing condition due to Demazure \cite{Dem70} on the roots of the Lie
algebra $\mathfrak{aut}(M_n(A))$ of $\Aut(M_n(A))$.

Theorem \ref{existence} does not produce explicit examples of extremal
K\"ahler metrics on Bott manifolds; however, there are a few explicit examples
known. These include:
\begin{enumerate}
\item The product metrics on $\bbc\bbp^1 \times \cdots \times \bbc\bbp^1$;
\item Calabi's extremal K\"ahler metrics \cite{Cal82} on Hirzebruch surfaces
  (stage $2$ Bott manifolds);
\item Koiso and Sakane's K\"ahler--Einstein metric \cite{KoSa86} on
  $\bbp(\BOne \oplus \calo(-1,1)) \rightarrow \bbc\bbp^1\times \bbc\bbp^1$;
  and, more generally,
\item admissible extremal K\"ahler metrics on $\bbp(\BOne\oplus
  \calo(k_1,\ldots,k_{n-1}))\to(\bbc\bbp^1)^{n-1}$ whose existence was proven
  by Hwang \cite{Hwa94} and Guan \cite{Gua95} (see also \cite[Section
    3]{ACGT08} for a treatment using the admissible convention).
\end{enumerate}
It is shown in \cite{Hwa94} that for $n=3$, if $k_1,k_2$ in (4) above have
opposite signs then the corresponding twist $\leq 1$ Bott manifold $M_3$
admits a CSC K\"ahler metric; whereas, if they have the same sign $M_3$ cannot
have a CSC K\"ahler metric \cite{ACGT08}.

To organize and generalize such examples, we note that the number
$t\in\{0,\ldots,n-1\}$ of holomorphically nontrivial $\bbc\bbp^1$ bundles in a
Bott tower $M_n(A)$ is a biholomorphism invariant of the total space which we
call the (holomorphic) \emph{twist}. An analogous topological invariant (the
number of topologically nontrivial $\bbc\bbp^1$ bundles in the tower), which
we call the \emph{topological twist}, was introduced by Choi and Suh
in~\cite{ChSu11}. As the above examples are all $\bbc\bbp^1$ bundles over a
product $(\bbc\bbp^1)^{n-1}$, they all have twist $\leq 1$.

There is a dual notion of (holomorphic or topological) \emph{cotwist}
$t'\in\{0,\ldots,n-1\}$, which is the number of bundles in the tower such that
the inverse image of any $\bbc\bbp^1$ fiber in $M_n(A)$ is (holomorphically or
topologically) trivial over the fiber.  For example, a stage $3$ Bott manifold
$M_3$ may be considered as a bundle of Hirzebruch surfaces over $\bbc\bbp^1$,
or as a $\bbc\bbp^1$ bundle over a Hirzebruch surface. If $M_3$ has twist $1$,
it is a $\bbc\bbp^1$ bundle over $\bbc\bbp^1\times\bbc\bbp^1$, while if it has
cotwist $1$, it is a $\bbc\bbp^1\times\bbc\bbp^1$ bundle over $\bbc\bbp^1$.

In Section~\ref{bm}, after reviewing Bott towers and the quotient
construction, we present the Bott manifolds of twist $t=0,1,2$, which we use
as running examples throughout the article. In particular, these values of $t$
cover all stage $3$ Bott manifolds, on which we place much emphasis, $n=3$
being the next dimension after Hirzebruch surfaces. We also discuss when two
stage $n$ Bott manifolds $M_n(A)$ and $M_n(A')$ are biholomoprhic.

In addition to proving Theorem~\ref{existence} and giving examples, we place
them in context in order to motivate further study. Bott towers can be studied
as smooth manifolds, complex manifolds, symplectic manifolds or K\"ahler
manifolds, and the rest of the article explores these approaches. Along the
journey, we obtain several results of independent interest, as we now explain.

In Section~\ref{top} we consider the topology of Bott manifolds.  In recent
years, research on Bott manifolds has centered around the {\it cohomological
  rigidity problem}~\cite{MaPa08,ChMaSu11,ChMa12,Choi15} which asks if the
integral cohomology ring of a toric complex manifold determines its
diffeomorphism (or homeomorphism) type.  This problem is still open even for
Bott manifolds, but has an affirmative answer in important special cases: in
particular for $n\leq 4$ and for topological twist $t\leq 1$, the cohomology
ring determines the diffeomorphism type~\cite{ChMaSu10,ChSu11,ChMa12,Choi15}.
Hence, in these cases the homeomorphism classification coincides with the
diffeomorphism classification. We review these ideas and obtain a
diffeomorphism classification of stage $3$ Bott manifolds.

\begin{result}\label{s3thm} The diffeomorphism type of a stage $3$ Bott
manifold $M_3$ is determined by its second Stiefel--Whitney class $w_2(M_3)$
and its first Pontrjagin class $p_1(M_3)$, which can be any integral multiple
$p$ of the primitive class $\tfrac12 c_1(\calo(1,1))^2$. Moreover\textup:
\begin{enumerate}
\item Each diffeomorphism type contains finitely many Bott manifolds with
  twist $\leq 1$, determined by the prime decomposition of $p$.
\item There are precisely three diffeomorphism types of stage $3$ Bott
  manifolds $M_3$ of cotwist $\leq 1$ and each diffeomorphism type has an
  infinite number of inequivalent toric Bott manifolds.
\end{enumerate}
\end{result}

In Section~\ref{complex}, we study the automorphism group, the K\"ahler cone,
and the Fano condition. For stage $3$ Bott manifolds we can give specific
information about $\cale(M_n)$ and CSC K\"ahler metrics according to the
diffeomorphism type.  For example, using the results in \cite{Hwa94,Gua95} we
obtain:

\begin{result}
A stage $3$ Bott tower $M_3(A)$ admits a constant scalar curvature K\"ahler
metric if and only if $\Aut(M_3(A))$ is reductive, which holds if and only if
$A^1_2=0$ and $A^1_3A^2_3<0$ or $A=I$.
\end{result}

We turn to the symplectic viewpoint in Section~\ref{symp}. In
Proposition~\ref{finite}, we observe that a result of McDuff \cite{McD11}
implies that there are only finitely many biholomorphism classes of complex
structure compatible with a fixed compact toric symplectic manifold.

\begin{result}
Let $(M,\gro)$ be a symplectic $2n$-manifold. Then there are finitely many
biholomorphism classes of Bott manifolds compatible with $(M,\gro)$, and their
number is bounded above by the number of conjugacy classes of $n$-tori in the
symplectomorphism group $\Symp(M,\gro)$.
\end{result}

In Section~\ref{calabiconstruction}, we use the generalized Calabi
construction to prove Theorem~\ref{existence} (2), and study the admissible
construction. Generally, Bott manifolds do not fit so well with the admissible
construction (cf.~\cite{ACGT08} and references therein) that has been so
successful in producing explicit examples of extremal K\"ahler metrics. 
Nevertheless, in Section \ref{extradmBott} we describe such examples in
arbitrary dimension, and these suffice to show that for stage $3$ Bott
manifolds of twist $\leq 1$ there is an extremal K\"ahler metric in every
K\"ahler class, i.e., $\cale(M_3)=\calk(M_3)$ in this case. We can also construct
extremal K\"ahler metrics for certain twist 2 examples in arbitrary dimension,
cf.~Proposition~\ref{2textremal}.

In Section \ref{cprojconnection} we touch briefly on how admissible examples
are related to c-projective equivalence \cite{CEMN16}. In particular we prove:

\begin{result}
On the total space of $\bbp(\BOne \oplus \calo(-1,1)) \rightarrow
\bbc\bbp^1\times \bbc\bbp^1 $ there exists an infinite number of pairs of
c-projectively equivalent constant scalar curvature (CSC) K\"ahler metrics
which are \emph{not} affinely equivalent.
\end{result}

We end with some almost K\"ahler examples, in the spirit
of~\cite{Don02,ACGT11,Lej10}.

\section{Bott manifolds}\label{bm}

We begin by following Grossberg and Karshon's description of Bott towers
\cite{GrKa94} as well as \cite{ChMa12}. We refer to the recent book
\cite{BuPa15} and references therein for further developments.

\subsection{Bott towers and their cohomology rings}

\begin{definition}\label{funddef} 
Given $n\in \bbn$, we construct complex manifolds $M_k$ for
$k\in\{0,1,\ldots,n\}$ inductively as follows. Let $M_0$ be a point
$\mathit{pt}$, and for $k\geq 1$, assume $M_{k-1}$ is already defined and
choose a holomorphic line bundle $\call_k$ on $M_{k-1}$.  Then $M_k$ is the
compact complex manifold arising as the total space of the $\bbc\bbp^1$ bundle
$\pi_k\colon \bbp(\BOne \oplus \call_k) \to M_{k-1}$.

We call $M_k$ the \emph{stage $k$ Bott manifold} of the \emph{Bott tower of
  height $n$}:
\[
M_n \xrightarrow{\pi_n}M_{n-1} \xrightarrow{\pi_{n-1}} \cdots M_2
\xrightarrow{\pi_2} M_1 =\bbc\bbp^1 \xrightarrow{\pi_1} \mathit{pt}.
\]
At each stage we have \emph{zero} and \emph{infinity sections} $\grs_k^0\colon
M_{k-1}\to M_k$ and $\grs_k^\infty\colon M_{k-1}\to M_k$ which respectively
identify $M_{k-1}$ with $\bbp(\BOne \oplus 0)$ and $\bbp(0\oplus \call_k)$.
We consider these to be part of the structure of the Bott tower
$(M_k,\pi_k,\grs_k^0,\grs_k^\infty)_{k=1}^n$.
\end{definition}

Notice that the stage $2$ Bott manifolds are nothing but the Hirzebruch
surfaces $\calh_a:=\bbp(\BOne \oplus \calo(a))\to \bbc\bbp^1$.

\begin{definition}  A \emph{Bott tower isomorphism} between Bott
towers $(M_k,\pi_k,\grs_k^0,\grs_k^\infty)_{k=1}^n$ and
$(\tilde{M}_k,\tilde{\pi}_k,\tilde{\grs}_k^0,\tilde{\grs}_k^\infty)_{k=1}^n$
is a sequence of biholomorphisms $F_k\colon M_k\to \tilde{M}_k$
(for $k\in\{0,1,\ldots,n\}$) which intertwine the maps
$\pi_k,\grs_k^0,\grs_k^\infty$ with the maps
$\tilde{\pi}_k,\tilde{\grs}_k^0,\tilde{\grs}_k^\infty$.
\end{definition}
Isomorphism of Bott towers of height $n$ is thus a stronger notion than
biholomorphism between the corresponding stage $n$ Bott manifolds.

\begin{remark}
We can consider the more general case where we have a projectivization
\[
\bbp(E_n)\to\bbp(E_{n-1})\to\cdots\to\bbp(E_2) \to\bbp(E_1)\to\mathit{pt}
\]
where the rank $2$ bundles $E_j$ do not necesarily split as they do in a Bott
tower. Of course, $E_2$ splits by a famous theorem of Grothendieck. But this
fails at the next stage since a rank $2$ bundle over a Hirzebruch surface need
not split as the first Hirzebruch surface $\calh_1$ shows. It would be
interesting to investigate this further. However the methods of the present
paper rely in an essential way on the toric geometry of the split case, which
we introduce in the next section.
\end{remark}

As the total space of a $\bbc\bbp^1$ bundle $\pi_k\colon
\bbp(\BOne\oplus\call_k)\to M_{k-1}$, $M_k$ has a fiberwise dual tautological
bundle $\calo(1)_{\BOne\oplus\call_k}$ whose first Chern class
$c_1(\calo(1)_{\BOne\oplus\call_k})$ is the Poincar\'e dual $PD(D_k^\infty)$
to the infinity section $D_k^\infty:=\grs_k^\infty(M_{k-1})$.

The vertical bundle $VM_k$ of $M_k\to M_{k-1}$ has first Chern class
$PD(D_k^0+D_k^\infty)$, as the generator of the fiberwise $\bbc^\times$ action
vanishes there. Since $VM_k=\calo(2)_{\BOne\oplus\call_k}\otimes
\pi_k^*\call_k$, it follows that $\pi_k^*\call_k$ has first Chern class
$PD(D_k^0 - D_k^\infty)$, and we let $\gra_k$ be the pullback of this class to
$M_n$. We let $x_k$ be the pullback of $PD(D_k^\infty)$ to $M_n$; thus
$y_k:=x_k+\gra_k$ is the pullback of $PD(D_k^0)$ to $M_n$.

\begin{prop}[\cite{ChMa12}] The map sending $x_j$
to $X_j+\cali$ induces a ring isomorphism
\begin{equation}\label{Bcohring}
H^*(M_n,\bbz)\cong\bbz[X_1,X_2,\ldots,X_n]/\cali
\end{equation}
where $\cali$ is the ideal generated by
\[
X_k^2+\tilde \gra_k(X_1,\ldots, X_{k-1}) X_k \quad\text{for}\quad
k\in\{1,\ldots,n\}
\]
and $\tilde\gra_k$ is the \textup(unique\textup) linear polynomial with $\tilde
\gra_k(x_1,\ldots, x_{k-1})=\gra_k$.
\end{prop}
\begin{proof} Since $1$ and $PD(D_k^\infty)$ restrict to a basis for the
cohomology of any $\bbc\bbp^1$ fiber of $M_k$ over $M_{k-1}$, the
Leray--Hirsch Theorem implies that the cohomology ring $H^*(M_k,\bbz)$ is a
free module over $H^*(M_{k-1},\bbz)$ with generators $1$ and $PD(D_k^\infty)$.
The result follows by induction, using the fact that for all $k\in\{1,\ldots,
n\}$, $x_k y_k=0$ in $H^*(M_n,\bbz)$.
\end{proof}
Note that the cohomology ring of $M_n$ is filtered by pullbacks of the
cohomology rings of $M_k$ for $0\leq k\leq n$.

\subsection{The quotient construction and examples}
A stage $n$ Bott manifold $M_n$ can be written as a quotient of $n$ copies of
$\bbc^2_*:=\bbc^2\setminus\{0\}$ by a complex $n$-torus $(\bbc^\times)^n$. To
see this, consider the action of $(t_i)_{i=1}^n\in (\bbc^\times)^n$ on
$(z_j,w_j)_{j=1}^n\in (\bbc^2_*)^n$ by
\begin{equation}\label{Tnact}
(t_i)_{i=1}^n\colon (z_j,w_j)_{j=1}^n\mapsto
\Bigl(t_jz_j,\Bigl(\prod_{i=1}^{n}t_i^{A^i_j}\Bigr) w_j\Bigr)\,\strut_{j=1}^n,
\end{equation}
where $A$ is a lower triangular unipotent integer-valued matrix
\begin{equation}\label{Amatrix}
A=\begin{pmatrix}
1   & 0 &\cdots  & 0 &0 \\
A^1_2& 1    &\cdots &0  &0\\
\vdots&\vdots&\ddots&\vdots& \vdots\\
A^1_{n-1}& A^2_{n-1}& \cdots& 1 &0\\
A^1_n   & A^2_n & \cdots & A^{n-1}_n &  1
\end{pmatrix}, \qquad A^i_j\in\bbz.
\end{equation}
Since the induced action of $(\bbr^+)^n$ is transverse to $(S^3)^n$, where
$S^3$ is the unit sphere in $\bbc^2_*$, the orbits of this action are in
bijection with orbits of the induced free $(S^1)^n$ action on $(S^3)^n$, and
the geometric quotient is a compact complex $n$-manifold $M_n(A)$.

\begin{prop}[\cite{GrKa94,Ish12}] There is a bijection between
isomorphism classes of height $n$ Bott towers and $\bbz^{n(n-1)/2}$, i.e.,
matrices $A$ as in~\eqref{Amatrix}: $M_n(A)$ is the unique Bott tower for
which the pullback of $c_1(\call_k)$ to the total space is
$\gra_k=\sum_{j=1}^{k-1} A^j_k x_j$ for $k\in\{1,\ldots,n\}$. In particular
the isomorphism class of a height $n$ Bott tower is determined by its filtered
cohomology ring.
\end{prop}
\begin{proof} Observe that $M_k(A)$ be the quotient of $(\bbc^2_*)^k$ by the
action~\eqref{Tnact} with $n$ replaced by $k$ for $0\leq k\leq n-1$; then
$M_k(A)=\bbp(\BOne\oplus\call_k)\to M_{k-1}(A)$ for some line bundle
$\call_k$. As shown in~\cite{GrKa94}, a holomorphic line bundle on $M_k(A)$ is
determined by its first Chern class, and it follows inductively that $M_n(A)$
is a Bott tower with the given Chern classes. Now the matrix $A$ is determined
by the filtered cohomology ring of the Bott tower and the result
follows~\cite{Ish12}.
\end{proof}
In this correspondence, it follows~\cite[Prop.~3.5]{CiRa05} that the height
$k$ Bott tower associated to the principal $k\times k$ submatrix of $A$ given
by removing the first $j$ rows and columns, and the last $n-j-k$ rows and
columns is the fiber of $M_{j+k}(A)$ over $M_j(A)$.

We turn now to examples; it is useful to organize these using the following
notions.
\begin{definition} The (\emph{holomorphic}) \emph{twist} of a Bott
tower $M_n(A)$ is the number $t\in\{0,\ldots,n-1\}$ of holomorphically
nontrivial $\bbc\bbp^1$ bundles in the tower, or equivalently the number of
nonzero rows in $A-\BOne_n$. Dually, we refer to the number
$t'\in\{0,\ldots,n-1\}$ of nonzero columns in $A-\BOne_n$ as the
(\emph{holomorphic}) \emph{cotwist}. (The $k$th column of $A-\BOne_n$ is zero
if and only if $M_n(A)\to M_k(A)$ is a pullback from $M_{k-1}(A)$.)
\end{definition}

The only Bott manifold $M_n$ of twist (or cotwist) $0$ is $(\bbc\bbp^1)^n$
with cohomology ring
\begin{equation}\label{cohring0twist}
H^*(M_n,\bbz)=\bbz[X_1,\ldots,X_n]/\bigl(X_1^2,\ldots,X_{n}^2\bigr).
\end{equation}
We now consider some examples of twist $1$ and $2$.
\begin{example}\label{twist1-examples} 
For $\bfk=(k_1,\ldots,k_N)\in \bbz^N$, let $M_{N+1}(\bfk)$ denote the
Bott tower $M_{N+1}(A)$ with 
\begin{equation}\label{genkexample}
A =\begin{pmatrix}
1   & 0 &\cdots & 0 & 0 \\
0   & 1 &\cdots & 0 & 0\\
\vdots&\vdots&\ddots&\vdots&\vdots\\
0   & 0 &\cdots & 1 & 0\\
k_1 & k_2&\cdots& k_N & 1
\end{pmatrix},
\end{equation}
which is the $\bbc\bbp^1$ bundle $\bbp(\BOne\oplus \calo(k_1,\ldots,k_N))\to
(\bbc\bbp^1)^N$. Hence it has twist $1$ unless $\bfk=0$, and its cohomology
ring is
\begin{equation}\label{cohring1twist}
\begin{split}
H^*(M_{N+1}(\bfk),\bbz)&=\bbz[X_1,\ldots,X_{N+1}]/\cali\\
\cali&=\bigl(X_1^2,\ldots,X_N^2,X_{N+1}(X_{N+1}+k_1X_1+\cdots +k_NX_N)\bigr).
\end{split}
\end{equation}
The case $N=1$ and $k_1=a$ is the Hirzebruch surface $\calh_a=M_n(A)$ with
$A=\bigl(\begin{smallmatrix} 1 & 0\\ a & 1 \end{smallmatrix}\bigr)$.
\end{example}

\begin{example}\label{twist2-examples}
Similarly for $\bfl=(l_1,\ldots,l_{N-1})\in\bbz^{N-1}$ and
  $\bfk=(k_1,\ldots,k_N)\in \bbz^N$, let $M_{N+1}(\bfl,\bfk)$ denote the
Bott tower $M_{N+1}(A)$ with
\begin{equation}\label{gen2example}
A =\begin{pmatrix}
1   & 0 &\cdots & 0 & 0 \\
0   & 1 &\cdots & 0 & 0\\
\vdots&\vdots&\ddots&\vdots&\vdots\\
l_1   & l_2 &\cdots & 1 & 0\\
k_1 & k_2&\cdots& k_N & 1
\end{pmatrix},
\end{equation}
This is a fiber bundle over $(\bbc\bbp^1)^{N-1}$ whose fiber is the Hirzebruch
surface $\calh_{k_N}$, and has twist $2$ unless $\bfk=0$ or $\bfl=0$. Its
cohomology ring is given by
\begin{equation}\label{cohring2twist}
\begin{split}
H^*(M_{N+1}(\bfk,\bfl),\bbz)&=\bbz[X_1,\ldots,X_{N+1}]/\cali\\
\cali&=\bigl(X_1^2,\ldots,X_{N-1}^2,X_N(X_N+l_1X_1+\cdots +l_{N-1}X_{N-1}),\\
&\qquad \qquad \qquad\qquad \qquad X_{N+1}(X_{N+1}+k_1X_1+\cdots +k_NX_N) \bigr).
\end{split}
\end{equation}
\end{example}

\begin{example}\label{3stageBottex} 
For stage $3$ Bott towers, we follow the notation of \cite{ChMaSu10} and let
$M_3(a,b,c)$ denote the Bott tower $M_3(A)$ with
\begin{equation}\label{A3matrix}
A=\begin{pmatrix}
1 & 0 & 0 \\
a & 1 & 0 \\
b & c & 1
\end{pmatrix}.
\end{equation}
This is the total space of $\bbc\bbp^1$ bundle over the Hirzebruch surface
$\calh_a$, and also a bundle of Hirzebruch surfaces over $\bbc\bbp^1$ with
fiber $\calh_c$.  Thus it has twist $\leq 1$ if $a=0$ and cotwist $\leq 1$ if
$c=0$. Its cohomology ring is
\begin{equation}\label{B3cohring}
H^*(M_3,\bbz)
=\bbz[X_1,X_2,X_3]/\bigl(X_1^2, X_2(aX_1+X_2), X_3(bX_1+cX_2+X_3)\bigr),
\end{equation}
which, as a $\bbz$-module, is freely generated by
$1,x_1,x_2,x_3,x_1x_2,x_2x_3,x_3x_1$ and $x_1x_2x_3$ (where $x_i$ is the image
of $X_i$ in $H^2(M_3,\bbz)$).
\end{example}

\subsection{Toric description of Bott manifolds}\label{s:toric}

The quotient construction shows that any Bott manifold is toric.  For general
descriptions of toric geometry we refer to
\cite{BuPa15,CoLiSc11,Dan78,Del88,LeTo97}, but we summarize the main ideas
here in order to declare our conventions.

\begin{definition} For a finite set $\cals$, let $\bbz_\cals$
and $\bbc_\cals=\bbz_\cals\otimes_Z \bbc$ be respectively the free abelian
group and complex vector space with basis $e_\grr:\grr\in\cals$, and let
the complex torus $\bbc_\cals^\times:=\bbz_\cals\otimes_Z \bbc^\times$ act
diagonally on $\bbc_\cals$ in the natural way, i.e.,
$(t_\grr)_{\grr\in\cals}\colon (z_\grr)_{\grr\in\cals}\mapsto (t_\grr
z_\grr)_{\grr\in\cals}$. A complex $n$-manifold (or orbifold) $M$ is
\emph{toric} if it is a quotient of a $\bbc_\cals^\times$-invariant open
subset of $\bbc_\cals$ by a subgroup $G^c$ of $\bbc_\cals^\times$ for some
$\cals$.
\end{definition}

It follows that the quotient torus $\bbt^c\cong \bbc_\cals^\times/G^c$ acts on
$M$ with an open orbit. To describe such $M$ up to equivariant biholomorphism,
it is convenient to fix $\bbt^c=\grL\otimes_\bbz \bbc^\times$, where $\grL$ is
a free abelian group of rank $n$ (the lattice of circle subgroups of
$\bbt^c$), and determine $G^c$ as the kernel of group homomorphism
$\bbc_\cals^\times\to \bbt^c$.  Any such homomorphism is determined by a
homomorphism $\bfu\colon\bbz_\cals\to \grL$, hence by the images
$u_\grr\in\grL$ of the basis vectors $e_\grr$ for all $\grr\in\cals$.  The
$\bbc_\cals^\times$-orbits in $\bbc_\cals$ may be parametrized by subsets of
$\cals$, where $R\subseteq \cals$ corresponds to the orbit $\bbc_{\cals,R}$ of
$\sum_{\grr\notin R} e_\grr$ (thus $\bbc_{\cals,\varnothing}$ is the open
orbit).

Hence any $\bbc_\cals^\times$-invariant subset of $\bbc_\cals$ has the form
$\bbc_{\cals,\Phi}:=\bigcup_{R\in\Phi} \bbc_{\cals,R}$ for some subset $\Phi$
of the power set $P(\cals)$, and it is open if and only if $\Phi$ is a
\emph{simplicial set}, i.e., $R\in\Phi$ and $R'\subseteq R$ implies that
$R'\in\Phi$. We also assume all singletons $\{\grr\}$ are in $\Phi$.

Any $R\subseteq \cals$ defines a \emph{cone} $\grs_R$ in $\gt:=
\grL\otimes_\bbz\bbr$ as the convex hull of $\{u_\grr\,|\,\grr\in R\}$.  We
require that the cones $\grs_R:R\in\Phi$ are distinct, strictly convex (they
contain no nontrivial linear subspace) and simplicial ($\dim\grs_R$ is the
cardinality of $R$), are closed under intersections, and have union $\gt$.
Then $\grS:=\{\grs_R:R\in\Phi\}$ is a \emph{complete simplicial fan}, and
there is a bijection between the cones in the fan and the orbits of
$\bbc_\cals^\times$ in $\bbc_{\cals,\Phi}$, hence the orbits of $\bbt^c$ in $M$. 
Let $\grS_r$ denote the set of $r$-dimensional cones in $\grS$; it is convenient to
identify $\cals$ with $\grS_1$, the set of $1$-dimensional cones, or rays, in
$\grS$; then $u_\grr\in \grr\cap\grL$, and is a multiple $m_\grr$ of the
unique primitive vector of $\grL$ in $\grr$, where $2\pi/m_\grr$ is the cone
angle of the orbifold singularity along the corresponding $\bbt^c$-orbit in
$M$, whose closure $D_{u_\grr}$ is a $\bbt^c$-invariant divisor in $M$. We refer
to $u_\grr\colon \grr\in\cals$ as the \emph{normals} of the fan: if $M$ is
smooth (or has no orbifold singularities in codimension one), the fan $\grS$
determines $(\cals,\bfu,\Phi)$, hence $M$; in general the normals are extra
data.

\smallbreak

Applying these methods to the Bott tower $M_n(A)$, we now see that the role of
the matrix $A$ in~\eqref{Tnact}--\eqref{Amatrix} is to define an inclusion of
lattices $\bigl(\begin{smallmatrix} \BOne_n\\ A \end{smallmatrix}\bigr)
\colon\bbz^n\hookrightarrow\bbz^{2n}$, where $\BOne_n$ is the $n\times n$ identity
matrix, and hence real and complex subtori $(S^1)^n\hookrightarrow(S^1)^{2n}$
and $(\bbc^\times)^n\hookrightarrow(\bbc^\times)^{2n}$ respectively of the
standard $2n$-torus and its complexification acting on $\bbc^{2n}$. We thus
identify $\bbz_\cals$ with $\bbz^{2n}$, although it is convenient to take
$\cals=\{1,\ldots,n,1',\ldots,n'\}$ to be the disjoint union of two copies of
$\{1,\ldots,n\}$ and define $w_j:=z_{j'}$, so that $\bbc_\cals\cong \bbc^{2n}$
with coordinates $(z_1,\ldots,z_n,w_1,\ldots,w_n)$.

Now the image of the inclusion $\bigl(\begin{smallmatrix} \BOne_n\\ A
\end{smallmatrix}\bigr)$ is the kernel of the map
$(-A{\:\:}\BOne_n)\colon \bbz^{2n}\to \bbz^n$.  The fan $\grS$ therefore has
normals $u_1,\ldots, u_n,v_1,\ldots,v_n\in\grL$ with $u_j=-\sum_{i=1}^n A_i^j
v_i$ and $v_j=u_{j'}$; however $v_1,\ldots, v_n$ is a $\bbz$-basis for $\grL$,
which we may use to identify $\grL$ with $\bbz^n$. In particular the quotient
torus $\bbt^c\cong(\bbc^\times)^n$ acts on $M_n(A)$ and the images of the
coordinate hyperplanes $z_k=0$ or $w_k=0$ are $\bbt^c$-invariant divisors in
$M_n(A)$.

We can describe these invariant divisors by noting that the quotient of
$\{(z_i,w_i)_{i=1}^n\in (S^3)^n: z_j=0\}$ by the $j$th circle in $(S^1)^n$ is
(since $|w_j|=1$) $(S^3)^{n-1}$ with the $j$th $3$-sphere removed, and
similarly for $\{(z_i,w_i)_{i=1}^n\in (S^3)^n: w_j=0\}$. The image of $z_j=0$
or $w_j=0$ in $M_n(A)$ is thus the height $n-1$ Bott tower obtained by
removing the $j$th row and $j$th column of $A$. In terms of the Bott tower
structure, the $\bbt^c$ action on $M_n(A)$ is a lift of the the fiberwise
$\bbc^\times$ actions acting on each stage $M_k(A)\to M_{k-1}(A)$, and the
invariant divisors in $M_n(A)$ given by $z_k=0$ and $w_k=0$ are the inverse
images $\hat D_k^\infty$ and $\hat D_k^0$ in $M_n(A)$ of $D_k^\infty$ and
$D_k^0$, whose homology classes are Poincar\'e dual to $x_k$ and
\begin{equation}\label{xyeq}
y_k = x_k+\gra_k =\sum_{j=1}^n A^j_k x_j.
\end{equation}
It will be convenient later to write this in vector notation $\bfy=A\bfx$.

There are $2n$ such divisors coming in opposite pairs $(\hat D_k^0,\hat
D_k^\infty):k\in\{1,\ldots, n\}$, labelled by the Bott tower structure, and
only opposite pairs have empty intersection. The divisors thus have the
structure of an \emph{$n$-cross} or \emph{cross-polytope} (the dual of an
$n$-cube) and the cones of the fan are therefore cones over the faces of an
$n$-cross in $\gt$ with vertex set $\cals$, the rays of the fan
\cite{Civ05}.

The symmetry group of an $n$-cross (or $n$-cube) is the Coxeter group
$BC_n\cong \mathrm{Sym}_n\ltimes\bbz_2^n$ acting on $\cals$, where the $i$th
generator of the $\bbz_2^n$ normal subgroup interchanges $i$ and $i'$ in
$\cals$, while the quotient group $\mathrm{Sym}_n$ permutes the pairs
$(1,1'),\ldots,(n,n')$.

\begin{remark} In~\cite[Theorem 3.4]{MaPa08}, it is shown that stage $n$ Bott
manifolds are the only toric complex manifolds whose fan is (the cone over) an
$n$-cross. A corresponding result does not hold in the orbifold case, even for
$n=2$.
\end{remark}

\subsection{Equivalences and the Bott tower groupoid}\label{equivsect}

It is natural to ask when two height $n$ Bott towers determine the same stage
$n$ Bott manifold, i.e., their total spaces are biholomorphic. Since a Bott
manifold $M_n$ is a complete toric variety, its biholomorphism group
$\Aut(M_n)$ is a linear algebraic group, and so all maximal tori in
$\Aut(M_n)$ are conjugate. Hence if $M_n$ and $\tilde M_n$ are biholomorphic,
there is a $\phi$-equivariant biholomorphism $f\colon M_n\to\tilde M_n$ for
some isomorphism $\phi\colon \bbt^c\to\tilde\bbt^c$ of the canonical complex
tori acting on $M_n$ and $\tilde M_n$ (i.e., $f(t\cdot z) =\phi(t)\cdot f(z)$
for all $t\in \bbt^c$, $z\in M_n$). In other words, $M_n$ and $\tilde M_n$ are
equivalent as toric complex manifolds.

We also use the fact that any height $n$ Bott tower is isomorphic to $M_n(A)$
for a unique $n\times n$ matrix $A$, and is toric with respect to a fixed
complex $n$-torus $\bbt^c\cong(\bbc^\times)^n$. Thus the set of all
isomorphism classes of Bott towers of dimension $n$ can be identified with the
set $\BT_0$ of unipotent lower triangular $n\times n$ matrices $A$ over $\bbz$
(which is in bijection with $\bbz^{n(n-1)/2}$) by associating to $A$ the
isomorphism class of $M_n(A)$.

\begin{definition}\label{def:BTgroupoid} The $n$-dimensional
\emph{Bott tower groupoid} is the groupoid with object set $\BT_0$, where the
morphisms from $A$ to $A'$ are the biholomorphisms
$M_n(A)\to M_n(A')$. We denote the set of morphisms by $\BT_1$.

The orbit of $\BT_1$ through $A$ is thus the set of all $A'$ such that
$M_n(A')$ that is \emph{equivalent} (i.e., biholomorphic) to $M_n(A)$.  The
quotient space of orbits $\BT_0/\BT_1$ is then in bijection with the set $\BM$
of (biholomorphism classes of) Bott manifolds $M_n$.
\end{definition}

As noted before the definition, we may restrict attention to
\emph{$\phi$-equivariant} equivalences $f\colon M_n(A)\to M_n(A')$ (for
$\phi\colon\bbt^c\to\bbt^c$). Such an equivalence pulls back
$\bbt^c$-invariant divisors to $\bbt^c$-invariant divisors preserving their
intersections, i.e., it induces an element of the combinatorial symmetry group
$BC_n\cong \mathrm{Sym}_n\ltimes\bbz_2^n$. It remains to understand which
elements of $BC_n$ are induced by equivalences. For this, we follow Masuda and
Panov~\cite{MaPa08} (although these authors work in the smooth context).

We first note that if we only know a Bott manifold $M_n$ as a toric complex
manifold, then to write it as a quotient of $(\bbc^2_*)^n\subset\bbc_\cals$ we
have to choose a bijection between $\cals=\{1,\ldots, n,1',\ldots, n'\}$ and
the $\bbt^c$-invariant divisors in $M_n$ (so that all pairs $(j,j')$
correspond to opposite divisors). Now a general $n$-dimensional subtorus
$G^c\leq \bbc^\times_\cals$ is defined by the kernel of maximal rank block
matrix $(B\:\:C)\colon\bbz^{2n}\to \bbz^n$ (where $B,C$ are $n\times n$).
This kernel is the image of $\bigl(\begin{smallmatrix}
  \BOne_n\\ A \end{smallmatrix}\bigr) \colon\bbz^n\hookrightarrow\bbz^{2n}$ if
$B+CA=0$, which forces $C$ to be invertible and $A=-C^{-1}B$. The automorphism
group GL$(n,\bbz)$ acts on $(B\:\:C)\colon\bbz^{2n}\to \bbz^n$ by left matrix
multiplication. Left multiplication by $C^{-1}$ thus provides the normal form
$(-A{\:\:}\BOne_n)\colon \bbz^{2n}\to \bbz^n$ for the quotient map that we have
used.

In $M_n(A)$, the stabilizers of $\hat D_k^0$ and $\hat D_k^\infty$ in $\bbt^c$
are the oriented $\bbc^\times$ subgroups generated by the images of $e_k$ and
$e_{k'}$ under $(-A{\:\:}\BOne_n)$. Hence the action of $BC_n$ on divisors lifts
to the natural action on $\bbz^{2n}$ by permutation matrices of the form
$\bigl(\begin{smallmatrix} P & Q \\ Q & P \end{smallmatrix}\bigr)$. Such a
matrix acts on $(-A{\:\:}\BOne_n)$ by right multiplication to give
$(Q-AP\:\:P-AQ)$. It follows that such an element of $BC_n$ is induced by an
equivalence of $M_n(A)$ with some $M_n(A')$ if and only if
$A':=(P-AQ)^{-1}(AP-Q)$ is lower triangular, and then it is induced by an
equivalence $M_n(A)\to M_n(A')$.

\begin{lemma}\label{0tlemma} Any element of $\bbz_2^n\trianglelefteq BC_n$
is induced by an equivalence $M_n(A)\to M_n(A')$ for some $A'$. In particular,
the Bott towers $M_n(A)$ and $M_n(A^{-1})$ are equivalent.
\end{lemma}
\begin{proof} It suffices to show that for each $1\leq k\leq n$, the
generator $\grt_k$ which interchanges the $k$th column of $-A$ with the $k$th
column of $\BOne_n$ is induced by an equivalence. However, here $Q-AP$ and $P-AQ$
are lower triangular, hence so is $A'=(P-AQ)^{-1}(AP-Q)$.

For the second part, we consider $\grt_1\grt_2\cdots\grt_n$, which sends
$(\hat D_k^0,\hat D_k^\infty)$ to $(\hat D_k^\infty,\hat D_k^0)$ for each
$k\in\{1,\ldots,n\}$; hence,
$\bfx':=\grt_1\grt_2\cdots\grt_n(\bfx)=\bfy$ and
$\bfy':=\grt_1\grt_2\cdots\grt_n(\bfy)=\bfx$. But for the Bott tower
$M_n(A)$ we have $\bfy=A\bfx$, whereas, for $M_n(A^{-1})$ we have
$\bfy'=A^{-1}\bfx'$.
\end{proof}
The equivalence inducing $\grt_k$ is easy to understand as a fiber
inversion at the $k$th stage: if $M_k(A)\cong P(\BOne\oplus\call_k)$, then
$M_k(A')\cong P(\BOne\oplus\call_k^{-1})$ and in homogeneous coordinates in
the fibers, the equivalence sends $(z_1,z_2)$ to $(z_1^{-1},z_2^{-1})$, or
equivalently, in a local trivialization of $\call_k$, to $(z_2,z_1)$, swapping
zero and infinity sections. The equivalences between matrices $A,A'$ induced
by the fiber inversion maps can be worked out in principle in terms of minor
determinants. Here we only do this for certain special examples.

Before turning to examples, we consider when $\sigma\in\mathrm{Sym}_n$ is
induced by an equivalence $M_n(A)\to M_n(A')$. In this case we can take $Q=0$,
so $P$ is the permutation matrix defined by $\sigma$, and hence lifts to an
equivalence $M_n(A)\to M_n(P^{-1}AP)$ if and only if $P^{-1}AP$ is lower
triangular. Since $\mathrm{Sym}_n$ is generated by transpositions, the
following case is of particular interest.

\begin{lemma}\label{trans} A transposition $\sigma=(i~j)$ with $1\leq i<j\leq n$
is induced by an equivalence of $M_n(A)\to M_n(P^{-1}AP)$ for the
corresponding permutation matrix $P$ if and only if $A^i_k=0$ for $i<k\leq j$ and
$A^k_j=0$ for $i\leq k<j$.  In particular, when $i=j-1$, this holds if and
only if $A^{j-1}_j=0$.
\end{lemma}
The proof of this lemma is immediate; we illustrate it in examples below.
It has the following consequence, which is one of the main steps in the proof
of~\cite[Lemma~3.1]{ChSu11}.

\begin{prop}\label{tw-equiv} A Bott tower $M_n(A)$ with twist $\leq t$ is
equivalent to a Bott tower $M_n(A')$ such that the first $n-t$ rows of
$A'-\BOne_n$ are identically zero, i.e., $M_{n-t}(A')=
(\bbc\bbp^1)^{n-t}$. Similarly, a Bott tower $M_n(A)$ with cotwist $\leq t$ is
equivalent to a Bott tower $M_n(A')$ such that the last $n-t$ columns 
of $A'-\BOne_n$ are
identically zero, i.e., $M_n(A')$ is a $(\bbc\bbp^1)^{n-t}$ bundle over
$M_t(A')$.
\end{prop}
\begin{proof} If the $j$th row of $A-\BOne_n$ is identically zero then
iterating Lemma~\ref{trans} shows that the permutation $(1~2\cdots j)
=(1~2)(2~3)\cdots(j-1~j)$ is induced by an equivalence. Now apply this
argument to all such rows. The proof for columns is similar.
\end{proof}

\begin{example}\label{1tequiv} Proposition~\ref{tw-equiv} shows that
any Bott tower $M_{N+1}(A)$ of twist $\leq 1$ is equivalent to $M_{N+1}(\bfk)$
for some $\bfk\in\bbz^N$.  The only fiber inversion that yields a nontrivial
equivalence of $M_{N+1}(\bfk)$ is $\grt_{N+1}$, which interchanges $\hat
D_{N+1}^\infty$ and $\hat D_{N+1}^0$, and is induced by an equivalence between
$M_{N+1}(\bfk)$ and $M_{N+1}(-\bfk)$. Also we have $M_{N+1}(\grs(\bfk))\sim
M_{N+1}(\bfk)$ for any permutation $\grs$ that lies in the subgroup
$\mathrm{Sym}_{N}\subset \mathrm{Sym}_{N+1}$ which permutes the first $N$
opposite pairs of invariant divisors. Since $M_{N+1}(\bfk)$ is a non-trivial
$\bbc\bbp^1$ bundle over the product $(\bbc\bbp^1)^N$, these equivalences are
just permuting the $N$ factors of the base.
\end{example}

\begin{example}\label{2tequiv} Proposition~\ref{tw-equiv} also shows that
any Bott tower $M_{N+1}(A)$ of twist $\leq 2$ is equivalent to
$M_{N+1}(\bfl,\bfk)$ for some $(\bfl,\bfk)\in\bbz^{N-1}\times\bbz^N$.  There
are now two non-trivial fiber inversion maps, namely, those interchanging
$\hat D_{N}^\infty$ and $\hat D_{N}^0$ and those interchanging $\hat
D_{N+1}^\infty$ and $\hat D_{N+1}^0$, giving 
equivalences
\[
M_{N+1}(\bfl,\bfk)\sim M_{N+1}(\bfl,-\bfk)\sim M_{N+1}(-\bfl,\bfk')
\sim M_{N+1}(-\bfl,-\bfk')
\]
where $\bfk'=(k_1-l_1k_N,\ldots, k_{N-1}-l_{N-1}k_N,k_N)$. A twist $2$ Bott
tower of the form $M_{N+1}(\bfl,\bfk)$ can be viewed as a fiber bundle over
the product $(\bbc\bbp^1)^{N-1}$ whose fiber is a Hirzebruch surface. So
permuting the $N-1$ factors of the base induce equivalences
$M_{N+1}(\grs(\bfl,\bfk))\sim M_{N+1}(\bfl,\bfk)$ where $\grs\in
\mathrm{Sym}_{N-1}\subset \mathrm{Sym}_{N+1}$ where $ \mathrm{Sym}_{N-1}$ is
the subgroup that permutes the first $N-1$ opposite pairs of invariant
divisors.
\end{example}

\begin{example}\label{st3equiv}
For the height $3$ Bott tower $M_3(A)=M_3(a,b,c)$, the fiber inversion maps
give rise to equivalences $\grt_1\colon M_3(a,b,c)\to M_3(a,b,c)$,
$\grt_2\colon M_3(a,b,c)\to M_3(-a,b-ac,c)$ and $\grt_3\colon M_3(a,b,c)\to
M_3(a,-b,-c)$. Hence we have equivalences:
\[
M_3(a,b,c)\sim M_3(a,-b,-c)\sim M_3(-a,b-ac,c)\sim M_3(-a,-(b-ac),-c)
=M_3(A^{-1}).
\]
The transposition $(1\:3)$ is order reversing, so is induced by an equivalence
only when $a=b=c=0$. The $3$-cycles $(1\:2\:3)$ and $(1\:3\:2)$ are induced by
equivalences when $b=c=0$ and $a=b=0$ respectively.  The transposition $(1\:
2)$ is induced by an equivalence when $a=0$ and the transposition $(2\: 3)$ is
induced by an equivalence when $c=0$.  We conclude that
\[
M_3(a,b,0)\sim M_3(b,a,0)\quad\text{and}\quad  M_3(0,b,c)\sim M_3(0,c,b),
\]
hence in particular $M_3(a,0,0)\sim M_3(0,a,0)\sim M_3(0,0,a)$. These
equivalences have straightforward interpretations: $M_3(a,b,0)$ has cotwist
$\leq 1$, i.e., is a $\bbc\bbp^1\times\bbc\bbp^1$ bundle over $\bbc\bbp^1$, and the
equivalence interchanges the two factors in the fibers; similarly $M_3(0,b,c)$
has twist $\leq 1$, i.e., is a $\bbc\bbp^1$ bundle over
$\bbc\bbp^1\times\bbc\bbp^1$, and the equivalence interchanges the two factors
in the base.
\end{example}

\begin{example} To illustrate Lemma~\ref{trans}, observe that transposing
the second and fourth rows and columns of a unipotent $4\times 4$ lower
triangular matrix $A$ yields
\[
A' =\begin{pmatrix}
1   & 0 &0 & 0\\
A^1_4& 1 &A^3_4 & A^2_4 \\
A^1_3& 0 &1 & A^2_3 \\
A^1_2& 0 &0 & 1
\end{pmatrix}
\]
and so the transposition $(2\:4)$ is induced by an equivalence of $M_4(A)$
with $M_4(A')$ if and only if $A^2_3 = A^2_4 = A^3_4 = 0$.
\end{example}

\section{The topology of Bott manifolds}\label{top}

\subsection{Topological twist}

There is a close interplay between the topological and biholomorphic theory
of Bott manifolds. One example is the following result.

\begin{prop}[Choi--Suh~\cite{ChSu11}]\label{top-to-hol} In a Bott
tower $M_n(A)$, the $\bbc\bbp^1$ bundle $M_k(A)=\bbp(\BOne\oplus\call_k)\to
M_{k-1}(A)$ is topologically trivial if and only if $\gra_k:=\sum_{j=1}^{k-1}
A^j_k x_j\equiv 0$ mod $2$ and $\gra_k^2=0$.  In this case $M_n(A)$ is
diffeomorphic to a Bott tower $M_n(A')$ such that $M_k(A')=\bbc\bbp^1\times
M_{k-1}(A)$ is holomorphically trivial over $M_{k-1}(A')=M_{k-1}(A)$.
\end{prop}
\begin{proof} $\bbp(\BOne\oplus\call_k)$ is topologically trivial if and only
if $(\BOne\oplus\call_k)\otimes\call$ is topologically trivial for some line
bundle $\call$. Choi and Suh~\cite{ChSu11} show that a sum of line bundles over
a Bott manifold is topologically trivial if and only if its total Chern class
is trivial. If $c_1(\call)=\lambda$ then
\[
c(\call\oplus\call_k\otimes\call)=c(\call)c(\call_k\otimes\call)
=(1+\lambda)(1+\alpha_k+\lambda)=1+\alpha_k+2\lambda+\lambda^2.
\]
This is trivial if $\alpha_k=-2\lambda$ and $\lambda^2=0$, i.e.,
$\alpha_k\equiv 0$ mod $2$ and $\alpha_k^2=0$ (since the cohomology ring has
no torsion). For the last part, we pullback $M_n(A)\to M_k(A)$ by the
diffeomorphism $M_k(A)\to\bbc\bbp^1\times M_{k-1}(A)$ trivializing $M_k(A)$
over $M_{k-1}(A)$.
\end{proof}

The construction in the proof does not affect the topological triviality
of the other fibrations in $M_n(A)$, which prompts the following definition.

\begin{definition} The \emph{topological twist} of a Bott manifold $M_n$
is the minimal (holomorphic) twist among Bott towers $M_n(A)$ diffeomorphic to
$M_n$.
\end{definition}

Proposition~\ref{top-to-hol} shows that the topological twist of $M_n$ is also
the minimal number of topologically nontrivial stages among Bott towers
diffeomorphic to $M_n$. In fact Choi and Suh show~\cite[Theorem 3.2]{ChSu11}
that the topological twist is the number of topologically nontrivial stages in
\emph{any} Bott tower diffeomorphic to $M_n$ (and this is their definition of
``twist''). Furthermore, by Proposition~\ref{tw-equiv} (cf.~\cite[Lemma
  3.1]{ChSu11}), a Bott manifold with topological twist $t$ is diffeomorphic
to a Bott tower $M_n(A)$ where the first $n-t$ rows of $A-\BOne_n$ are
identically zero, i.e., a holomorphic fiber bundle over $(\bbc\bbp^1)^{n-t}$
whose fiber is a stage $t$ Bott manifold.

The topological twist of a Bott manifold has implications for its Pontrjagin
classes.
\begin{lemma}[\cite{ChMaMu15,ChSu11}]\label{Pontt}
Let $M_n$ be a Bott manifold with topological twist $\leq t$.
\begin{enumerate}
\item If $M_n$ is diffeomorphic to $M_n(A)$, its total Pontrjagin class
$p(M_n)$ is given by
\begin{equation}\label{totPont}
p(M_n)=\prod_{j=n+1-t}^n(1+\gra_j^2);
\end{equation}
in particular $p_k(M_n)$ vanishes for $k>t$.
\item $p_k(M_n)=0$ if $k\geq \frac{n}{2}$.
\end{enumerate}
\end{lemma}
\begin{proof}
For (1), the computation of~\eqref{totPont} is straightforward~\cite{ChMaMu15}
and the last part follows by taking $M_n(A)$ to have twist $\leq t$. For (2)
we note that the strict inequality $k>\frac{n}{2}$ follows for dimensional
reasons. For the equality we set $n=2m$ and and observe that the class
$p_m(M_{2m})$ is a multiple of $x_1\cdots x_{2m}$, so when $t=2m-1$ we have
\[
p_m(M_{2m})=\prod_{j=2}^{2m}\gra_j^2
=\prod_{j=2}^{2m}\biggl(\sum_{i=1}^{j-1}A^i_jx_i\biggr)^2
\]
which does not contain $x_{2m}$.
\end{proof}

For later use, we also note that the total Chern class of $M_n(A)$ is
\begin{equation}\label{totChern}
c(M_n(A))=\prod_{j=1}^n(1+x_j+y_j)=\prod_{j=1}^n(1+2x_j+\gra_j),
\end{equation}
cf.~\cite{Abc13}. In particular, for the first Chern class we get
\begin{equation}\label{c1}
c_1(M_n(A))=\sum_{j=1}^n(x_j+y_j).
\end{equation}

\subsection{Cohomological rigidity of Bott manifolds}

In recent years, research on Bott manifolds has centered around the {\it
  cohomological rigidity problem} which asks if the integral cohomology ring
of a toric complex manifold determines its diffeomorphism (or homeomorphism)
type~\cite{ChMaSu11,ChMa12,Choi15,MaPa08}. This problem is still open even for
Bott manifolds, but has an affirmative answer in important special cases, in
particular for Bott manifolds with topological twist $\leq 1$. This gives a
lot of information about the topology of these Bott manifolds.

\subsubsection*{Bott manifolds with topological twist $0$}

A Bott manifold $M_n$ has topological twist $0$ if and only if it is
diffeomorphic to $(\bbc\bbp^1)^n$. In this case there is the following
characterization.
\begin{theorem}[Masuda--Panov~\cite{MaPa08}]\label{mapathm}
A stage $n$ Bott manifold $M_n$ is diffeomorphic to $(\bbc\bbp^1)^n$ if and
only if $H^*(M_n,\bbz)\cong H^*((\bbc\bbp^1)^n,\bbz)$ as graded rings.
Furthermore, $M_n(A)$ is diffeomorphic to $(\bbc\bbp^1)^n$ if and only if the
matrix $A$ takes the form\footnote{Comparing notations we note that our $A$ is
  $-A^t$ in \cite{MaPa08} with $a_{ij}=-A^i_j$; this transposition also reverses
the order of the $C_k$'s.}
\begin{equation}\label{0teqn}
A=2C_n\cdots C_1-\BOne_n
\end{equation}
where $C_k$ is a lower triangular unipotent integer-valued matrix with at most
one non-zero element $C^{i_k}_k$ below the diagonal and that lies in the $k$th
row.
\end{theorem}

Note that $C_1$ is the identity matrix. It follows immediately from
\eqref{0teqn} that all off-diagonal elements of $A$ are multiples of $2$
(in accordance with Proposition~\ref{top-to-hol}).

\begin{example}
As an example we consider stage $3$ Bott manifolds with topological twist
zero. It follows from~\eqref{0teqn} (and also Proposition~\ref{c0} below) that
they can be represented by the Bott towers $M_3(2a,2b,0)$ and
$M_3(2a,2ac,2c)$. The former has cotwist $\leq 1$, hence is a
$\bbc\bbp^1\times \bbc\bbp^1$ bundle over $\bbc\bbp^1$, whereas the latter has
twist and cotwist $2$. However, the $A$ matrices for $M_3(2a,2b,2c)$ with
$c(b-ac)\neq 0$ do not satisfy \eqref{0teqn}; they have non-vanishing first
Pontrjagin class.
\end{example}

\subsubsection*{$\bbq$-trivial Bott manifolds}

The cohomological rigidity of Bott manifolds with topological twist $0$
generalizes to Bott manifolds $M_n$ which are \emph{$\bbq$-trivial}, i.e.,
with $H^*(M_n,\bbq)\cong H^*((\bbc\bbp^1)^n,\bbq)$ as graded
rings~\cite{ChMa12}. A key ingredient in establishing this is the following
observation.

\begin{lemma}[\cite{ChMa12}]\label{primsq0} Let $\lambda x_j + u$ be a
primitive element of $H^2(M_n,\bbz)$ with $\lambda\neq 0$ and $u$ in the span
of $x_i:i<j$. Then $(\lambda x_j + u)^2=0$ if and only if $\gra_j^2=0$ and
$2u=\lambda\gra_j$.
\end{lemma}
\begin{proof} By the assumptions on $\lambda$ and $u$,
$(\lambda x_j + u)^2=\lambda^2 x_j^2 + 2\lambda u x_j+u^2= \lambda(2u-\lambda
  \gra_j)x_j + u^2=0$ if and only if $2u-\lambda\gra_j=0$ and $u^2=0$.
\end{proof}
Thus the primitive square-zero elements of $H^2(M_n,\bbz)$ have the form $2x_j
+\gra_j$ or $x_j+\gra_j/2$ up to sign. It follows easily~\cite{ChMa12}
that $M_n$ is $\bbq$-trivial if and only if $\gra_k^2=0$ for all
$k\in\{1,\ldots, n\}$. Note that $\gra_1=0$ and $\gra_2^2=0$, so stage $2$
Bott manifolds, i.e., Hirzebruch surfaces, are always $\bbq$-trivial, and it
is well known that there are precisely two diffeomorphism types, distinguished
by the parity of $A^1_2$.

More generally, Choi and Masuda~\cite{ChMa12} show that $\bbq$-trivial Bott
manifolds are distinguished by their integral cohomology rings with $\bbz$
coefficients, and in this case there are exactly $P(n)$ diffeomorphism types
where $P(n)$ is the number of partitions of $n$. Further, for any
$\bbq$-trivial Bott manifold $M_n$, there is a partition
$\lambda_1\geq\lambda_2\geq\cdots \geq \lambda_m\geq 1$ of $n$ such that
$M_n\cong M_{(\lambda_1)}\times\cdots \times M_{(\lambda_m)}$ where
$M_{(\lambda)}$ is the stage $\lambda$ Bott manifold associated to the matrix
$A$ with $A^1_k=1$ for all $k$ and $A^j_k=\delta^j_k$ for $j\geq 2$.

Thus, by the famous formula of Hardy and Ramanujan the number of
diffeomorphism types of stage $n$ $\bbq$-trivial Bott manifolds grows
asymptotically like
\[
\frac{\exp(\pi(2/3)^{1/2}\sqrt{n})}{4n\sqrt{3}} \qquad
\text{as $n\to\infty$}.
\]
However, for $n\geq 3$, it follows from the formula~\eqref{totPont} for the
total Pontrjagin class that there are infinitely many diffeomorphism types of
stage $n$ Bott manifolds that are not $\bbq$-trivial. Nevertheless, the
cohomological rigidity problem has an affirmative answer for stage $n$ Bott
manifolds with $n\leq 4$ \cite{ChMaSu10,Choi15}.

\subsubsection*{Bott manifolds with topological twist $1$}

Any stage $N+1$ Bott manifold $M_{N+1}$ with topological twist $\leq 1$ is
diffeomorphic to $M_{N+1}(\bfk)$ for some $\bfk=(k_1,\ldots,k_N)\in\bbz^N$ as
in Example~\ref{twist1-examples}. These manifolds have each have one
nonvanishing Pontrjagin class, viz.
\begin{equation}\label{1tPont}
p_1(M_{N+1}(\bfk))=2\sum_{i<j}k_ik_jx_ix_j.
\end{equation}
Their first Chern class is given by
\[
c_1(M_{N+1}(\bfk))=\sum_{i=1}^{N}(2 +k_i)x_i +2x_{N+1},
\]
so the second Stiefel--Whitney class is 
\[
w_2(M_{N+1}(\bfk))=\sum_{i=1}^{N}k_ix_i \mod 2.
\]
In this case the diffeomorphism type is determined by the graded cohomology
ring $H^*(\bbp(\BOne\oplus \calo(k_1,\ldots,k_N)),\bbz)$.

\begin{theorem}[\cite{ChSu11}]\label{ChSuthm}
Let $M_{N+1}(\bfk)$ and $M_{N+1}(\bfk')$ be two Bott towers whose $A$
matrix has the form \eqref{genkexample}. Then $M_{N+1}(\bfk)$ and
$M_{N+1}(\bfk')$ are diffeomorphic if and only if there is a permutation
$\grs$ of $\{1,\ldots,N\}$ such that $k'_{\grs(i)}\equiv k_i\mod 2$ for all
$i$ and $k'_{\grs(i)}k'_{\grs(j)} =\pm k_ik_j$ for all $i\neq j$.
\end{theorem}

This theorem characterizes the case that $M_{N+1}(\bfk)$ has topological twist
$0$, i.e., is diffeomorphic to $(\bbc\bbp^1)^{N+1}$: this happens precisely
when $\bfk$ has at most one non-vanishing component and it is even; thus
$M_{N+1}(\bfk)$ is the product of an even Hirzebruch surface with
$(\bbc\bbp^1)^{N-1}$.

When $N=2$ we have a stage $3$ Bott manifold with twist $\leq 1$ which will be
treated in Section~\ref{stage3}. In this case the number of Bott manifolds in
a given diffeomorphism type is determined by the prime decomposition of
$k_1k_2$. For $N>2$, Theorem~\ref{ChSuthm} has the following refinement.

\begin{theorem}\label{1twistcor} If $N>2$ and all $k_j$ are nonvanishing,
then $M_{N+1}(\bfk)$ is diffeomorphic to $M_{N+1}(\bfk')$ if and only if there
is a permutation $\grs$ of $\{1,\ldots,N\}$ such that $k'_{\grs(i)}=\pm k_i$
for all $i=1,\ldots,N$.  Moreover, there are generically $2^{N-1}$
inequivalent such Bott manifolds in each diffeomorphism class.
\end{theorem}
\begin{proof} By the theorem $M_{N+1}(\bfk)$ and $M_{N+1}(\bfk')$ are
diffeomorphic if and only if there is a permutation $\grs$ such that $\mu_i
\mu_j = \pm 1$ for all $i\neq j$, where $\mu_i=k_{\grs (i)}'/k_i$; hence $\mu_i
=\pm 1/\mu_j$ for $i\neq j$. For $N>2$, we have also $\mu_j=\pm 1/\mu_k$ and
$\mu_k=\pm 1/\mu_i$ for $k\neq i,j$, which implies that $\mu_i = \pm 1$, i.e.,
$k'_{\grs(i)}=\pm k_i$ for all $i$. This proves the first part of the theorem.

For the second part we note that the only fiber inversion that yields a
nontrivial equivalence is $\grt_{N+1}$, which interchanges $\hat
D_{N+1}^\infty$ and $\hat D_{N+1}^0$, and is induced by an equivalence between
$M_{N+1}(\bfk)$ and $M_{N+1}(-\bfk)$. Now in the diffeomorphism class of
$M_{N+1}(\bfk)$, there $2^N$ choices of sign for the components of $\bfk$, but
the equivalence of $M_{N+1}(\bfk)$ and $M_{N+1}(-\bfk)$ makes half of them
equivalent. There are no further equivalences unless $k_i=\pm k_j$ for some
$i\neq j$ (in which case some sign choices are identified by transposing the
$i$th and $j$th factors in the base). Thus there are generically $2^{N-1}$
inequivalent Bott manifolds in each diffeomorphism class.
\end{proof}

Now we let $\cala(\bfk)$ be the set of lower triangular unipotent matrices
over $\bbz$ such that the Bott tower $M_{N+1}(A_\bfk)$ for
$A_\bfk\in\cala(\bfk)$ is diffeomorphic to $M_{N+1}(\bfk)$.  In the degenerate
case when some of the $k_j$s vanish we can without loss of generality assume
that $k_j\neq 0$ for $j=1,\ldots,m$, but $k_j=0$ for $j=m+1,\ldots,N$. In this
case we have a biholomorphism for $m=2,\ldots,N+1$
\begin{equation}\label{1tdeg}
M_{N+1}(A_\bfk)\cong M_{m}(A_{\tilde{\bfk}})\times
\overbrace{\bbc\bbp^1\times\cdots\times\bbc\bbp^1}^{N+1-m}
\end{equation}
where none of the components $k_1,\ldots,k_{m-1}$ of $\tilde{\bfk}$ vanish.

\subsubsection*{Bott manifolds with topological twist $2$}

A Bott manifold with topological twist $\leq 2$ is diffeomorphic to a Bott
tower of the form $M_{N+1}(\bfl,\bfk)$ as in Example~\ref{twist2-examples}.
Less is known about the topology in this case; their only non-vanishing
$\gra_j$ are for $j=N,N+1$ with
\[
\gra_N^2=2\sum_{i<j}^{N-1}l_il_jx_ix_j, \quad \gra_{N+1}^2=2\sum_{i<j}^{N-1}k_ik_jx_ix_j +k_N\sum_{i=1}^{N-1}(2k_i -k_Nl_i)x_ix_N.
\]
The total Pontrjagin class $p(M_{N+1}(\bfl,\bfk))$ is a diffeomorphism invariant and from Lemma \ref{Pontt} there are at most only 2 non-vanishing classes: 
\begin{equation}\label{2tdiffinv}
p_1(M_{N+1}(\bfl,\bfk))=\gra_N^2+\gra_{N+1}^2,\qquad p_2(M_{N+1}(\bfl,\bfk))=\gra_N^2\gra_{N+1}^2.
\end{equation}
Notice also by Lemma \ref{Pontt} that for $N=3$ we have $p_2(M_{4}(\bfl,\bfk))=0$.

We also have 
\[
c_1(M_{N+1}(\bfl,\bfk))=\sum_{i=1}^{N-1}(2+l_i +k_i)x_i +(2+k_N)x_N+2x_{N+1}
\]
from which we obtain the second Stiefel--Whitney class
\[
w_2(M_{N+1}(\bfl,\bfk))=\sum_{i=1}^{N-1}(l_i +k_i)x_i +k_Nx_N \mod 2.
\]
This implies that $w_2=0$ if and only if $k_N$ is even and $l_i,k_i$ have the same parity for all $i=1,\ldots,N-1$.

\subsection{Topological classification of stage $3$ Bott
manifolds}\label{stage3}

Let us consider in detail the topology of height $3$ Bott towers
$M_3(a,b,c)$. From \eqref{totPont} we have
\begin{equation}\label{p1}
p_1(M_3)=\gra_3^2 = c(2b-ac)x_1x_2,
\end{equation}
while from \eqref{c1} we have
\begin{equation}\label{B3c11}
c_1(M_3)=(2+a+b)x_1+(2+c)x_2+2x_3.
 \end{equation}
The mod 2 reduction of $c_1$ is the second Stiefel--Whitney class $w_2$ so we
see that
\begin{equation}\label{w2}
w_2(M_3)\equiv (a+b)x_1+cx_2 \mod 2.
\end{equation}
\begin{lemma}\label{diff}
The Bott tower $M_3(a,b,c)$ is $\bbq$-trivial if and only if $p_1(M_3)=0$
if and only if $c(2b-ac)=0$. Furthermore $w_2(M_3)=0$ if and only if $c$ and
$a+b$ are even.
\end{lemma}

By Lemma~\ref{primsq0}, the primitive square-zero elements of $H^2(M_3,\bbz)$
are, up to sign,
\begin{equation}\label{betas}
\beta_1=x_1, \quad\beta_2=\begin{cases} x_2 + \frac12 a x_1& a \text{ even}\\
2x_2 + a x_1 & a \text{ odd,}\end{cases} \quad
\beta_3= \begin{cases} x_3 + \frac12 b x_1 + \frac12 cx_2 & b,c\text{ even}\\
2x_3 + b x_1 + c x_2 &\text{otherwise,}\end{cases}
\end{equation}
with $\beta_3$ only occuring in the $\bbq$-trivial case. The topological twist
$0$ case is characterized as follows.

\begin{proposition}\label{c0} A Bott tower $M_3(a,b,c)$ is diffeomorphic
to $(\bbc\bbp^1)^3$ if and only if it is $\bbq$-trivial \textup(i.e.,
$c(2b-ac)=0$\textup) and $a,b,c$ are all even.
\end{proposition}

\begin{proof} Clearly if the integral cohomology ring of $M_3$ is isomorphic
to that of $(\bbc\bbp^1)^3$ then $M_3$ is $\bbq$-trivial, and since
$\beta_1\beta_2\beta_3$ is primitive, $a,b$ and $c$ are all even.  Conversely,
these conditions imply that $\beta_1,\beta_2,\beta_3$ have square zero and span
$H^2(M,\bbz)$, so that they generate $H^*(M,\bbz)$. Let $\xi_1,\xi_2,\xi_3$ be
the lifts of $\beta_1,\beta_2,\beta_3$ to $\bbz[x_1,x_2,x_3]$ using the
explicit expressions~\eqref{betas}. It is straightforward to check that
$\xi_1^2,\xi_2^2$ and $\xi_3^2$ generate the ideal $\cali$. Hence
$H^*(M,\bbz)\cong\bbz[\xi_1,\xi_2,\xi_3]/\bigl(\xi_1^2,\xi_2^2,\xi_3^2)$ and
the result follows by Theorem~\ref{mapathm}.
\end{proof}

\begin{remark}
Note that Proposition \ref{c0} in particular implies that the total space of
$\bbp(\BOne\oplus \calo(k_1,k_2))\to\bbc\bbp^1\times \bbc\bbp^1$ with $k_1k_2\neq 0$ is never diffeomorphic to $(\bbc\bbp^1)^3$.
\end{remark}

From \cite{ChMa12} we know that for $\bbq$-trivial stage $3$ Bott manifolds
there are precisely three diffeomorphism types. The above proposition shows
that $M_3(a,b,c)\cong M_{(1)}^3=(\bbc\bbp^1)^3$ if and only if $a,b,c$ are all
even.  The other two diffeomorphism types can be distinguished by the second
Stiefel--Whitney class $w_2(M_3)$. If $a,b$ are odd and $c$ is even then
$w_2(M_3)=0$ and $M_3$ is diffeomorphic to $M_{(3)}$. Otherwise, $w_2(M_3)\neq
0$, $M_3$ is diffeomorphic to $M_{(1)}\times M_{(2)}$, and the equation
$c(2b-ac)=0$ implies that either $c$ is even and $a+b$ is odd, with
$w_2(M_3)\equiv \beta_2\equiv x_1$, or $c$ is odd and $a$ is even, with
\[
w_2(M_3)\equiv\beta_3\equiv\begin{cases} x_2 & b \text{ even}\\
x_1+x_2 & b \text{ odd.} 
\end{cases}
\]
Thus the isomorphism of cohomology rings between $\bbq$-trivial height $3$ Bott
towers $M_3(a,b,c)$ and $M_3(a',b',c')$ with $w_2\neq 0$ maps $\beta_j\to
\beta_j'$ for all $j$ if $a$ and $a'$ have the same parity; otherwise it maps
$\beta_2$ to $\beta_3'$ and $\beta_3$ to $\beta_2'$.

\begin{proposition}\label{stage3diff} 
Let $M_3(a,b,c)$ and $M_3(a',b',c')$ be Bott towers
which are not $\bbq$-trivial. Then $M_3(a,b,c)$ and $M_3(a',b',c')$ are
diffeomorphic if and only if $c(2b-ac)=\pm c'(2b'-a'c')$, $(a,a')$ have the
same parity, and\textup:
\begin{itemize}
\item if $(a,a')$ are both even, then $((1+b)(1+c),(1+b')(1+c'))$ have the
same parity\textup;
\item if $(c,c')$ are both even, then $(b,b')$ have the same parity.
\end{itemize}
\end{proposition}
\begin{proof} If $M_3(a,b,c)$ and $M_3(a',b',c')$ are diffeomorphic, there is
an isomorphism $\psi$ of their integral cohomology rings intertwining their
Pontrjagin and Stieffel--Whitney classes. Since $x_1x_2$ and $x_1'x_2'$ are
primitive, this implies $c(2b-ac)=\pm c'(2b'-a'c')$. Since the two Bott
towers are not $\bbq$-trivial, $\psi$ induces a bijection between
$\{\pm\beta_1,\pm\beta_2\}$ and $\{\pm\beta_1',\pm\beta_2'\}$ and hence
$\beta_1\beta_2$ is mapped to $\pm\beta_1'\beta_2'$. Thus $(a,a')$ have the same
parity, and the last two conditions follow by considering whether $w_2$
vanishes or not.

Conversely, given the assumptions on $a,b,c$ and $a',b',c'$, it suffices
by~\cite{ChMaSu10} to show that $M_3(a,b,c)$ and $M_3(a',b',c')$ have
isomorphic integral cohomology rings. Replacing $x_1$ by $-x_1$, we see that
$M_3(a,b,c)$ and $M(-a,-b,c)$ have isomorphic cohomology, so we may assume
that $c(2b-ac)=c'(2b'-a'c')$. Replacing $x_2$ by $x_2+\lambda x_1$ for
$\lambda\in\bbz$, we see that $M_3(a,b,c)$ and $M_3(a+2\lambda, b+ c\lambda,c)$
have isomorphic cohomology, so we may assume $a=a'$. Thus
$2bc-2b'c'=a(c+c')(c'-c)$. If $a$ is odd then $c$ and $c'$ have the same
parity and hence $bc\equiv b'c'$ mod $2$, so that $b$ and $b'$ have the same
parity (by assumption if $c$ and $c'$ are even). If $a=2\tilde a$ then
$bc-b'c'=\tilde a(c+c')(c'-c)$ hence $(1+b)(1+c)-(1+b')(1+c') +(b'-b) =
(\tilde a(c+c')-1) (c'-c)$. Since the first two terms have the same parity,
$(b,b')$ have the same parity if and only if $(c,c')$ do. In the opposite
parity case, $\tilde a$ must be even, and it follows that $(b,c')$ and
$(c,b')$ have the same parity. Now replacing $x_1$ by $x_2+(a/2) x_1$ and
$x_2$ by $(1-a^2/4) x_1 - (a/2) x_2$, we see that $M_3(a,b,c)$ and
$M_3(a,c+(a/2) b -(a^2/4) c, b-(a/2) c)$ have isomorphic cohomology.
Hence we may assume that $(b,b')$ have the same parity and $(c,c')$
have the same parity (as in the case $a$ odd).

Finally, replacing $x_3$ by $x_3 - \frac12(b-b') x_1 - \frac12(c-c') x_2$,
we conclude that $M_3(a,b,c)$ and $M_3(a',b',c')$ have isomorphic integral
cohomology.
\end{proof}

The cotwist $\leq 1$ case, when $M_3$ is a $\bbc\bbp^1\times\bbc\bbp^1$ bundle
over $\bbc\bbp^1$, is $c=0$, and the ($\bbq$-trivial) cohomology ring reduces
to
\[
H^*(M_3,\bbz)=\bbz[x_1,x_2,x_3]/\bigl(x_1^2, x_2(ax_1+x_2), x_3(bx_1+x_3)\bigr),
\]
whereas the twist $\leq 1$ case, when $M_3$ is a $\bbc\bbp^1$ bundle over the
product $\bbc\bbp^1\times \bbc\bbp^1$, is $a=0$, and the cohomology ring
reduces to
\begin{equation}\label{B3cohring2}
H^*(M_3(0,b,c),\bbz)
=\bbz[x_1,x_2,x_3]/\bigl(x_1^2, x_2^2, x_3(bx_1+cx_2+x_3)\bigr).
\end{equation}
In the twist $\leq 1$ case, the parameters $b,c$ are the bidegree of the line
bundle $\calo(b,c)$ over $\bbc\bbp^1\times \bbc\bbp^1$, so $M_3(0,b,c)$ can be
realized as the projectivization $\bbp(\BOne\oplus \calo(b,c))$, which fits
into the general admissible construction of \cite{ACGT08}. This case includes
the K\"ahler--Einstein examples of Koiso and Sakane~\cite{KoSa86} with
$(b,c)=(1,-1)$ as well as other extremal and CSC metrics in \cite{Hwa94,Gua95}
as briefly discussed above.

More precisely, $p_1(M_3)$ is the multiple $p=2bc$ of the primitive integral
class $x_1x_2=\frac{1}{2}c_1(\calo(1,1))^2$ (pulled back to $M_3$) and
$w_2(M_3)\equiv bx_1+cx_2$ mod $2$. Hence the number of Bott manifolds with
twist $\leq 1$ in a fixed diffeomorphism type is the number of factorizations
$bc$ of $p$ with fixed parity of $(1+b)(1+c)$ (which is odd if and only if
$b,c$ are both even).  This may be determined easily from the prime
decomposition of $p$.
\begin{example}\label{type2diffeo}
For example consider $bc=\pm 24$. Proposition \ref{stage3diff} gives two
distinct diffeomorphism types with diffeomorphisms 
$M_3(0,24,1)\cong M_3(0,24,-1)\cong M_3(0,8,3)\cong M_3(0,8,-3)$ and
$M_3(0,12,2)\cong M_3(0,12,-2)\cong M_3(0,6,4)\cong M_3(0,6,4)$. The
first set has $w_2\neq 0$ while the second has vanishing $w_2$, so the two
sets are distinct even as homotopy types.  It is interesting to note that when
$bc$ is negative, we have, as mentioned previously, CSC K\"ahler metrics
\cite{Hwa94}.
\end{example}

\section{The complex viewpoint}\label{complex}

\subsection{The automorphism group}\label{aut-group}

The isotropy subgroup $\Aut(M_n(A))$ of $\BT_1$ of the Bott tower $M_n(A)\in
\BT_0$ is by definition the automorphism group $\Aut(M_n)$ of the underlying
Bott manifold $M_n$. We let $\Aut_0(M_n(A))$ denote its identity component.

There are many Bott towers $M_n(A)$ (e.g., a Hirzebruch surface $\calh_a$ with
$a\neq 0$) such that $\Aut_0(M_n(A))$ is not reductive. So by the
Matushima--Lichnerowicz criterion~\cite{Lic58}, many Bott manifolds do not
admit K\"ahler metrics of constant scalar curvature. As pointed out by Mabuchi
\cite{Mab87}, it follows from the work of Demazure \cite{Dem70} that the
reductivity of $\Aut_0(M_\grS)$ for a toric complex manifold $M_\grS$ is
equivalent to a balancing condition on a certain set $R(\grS)$ of \emph{roots}
associated to its fan $\grS$. A (Demazure) root of $M_\grS$ is an element
$\chi\in \gt^*$ which is dual to some normal $u_\grr\in \gt$ (for
$\grr\in\grS_1$) in the sense that $\chi(u_\grr)=1$ and $\chi(u_{\grr'})\leq
0$ for all $\grr'\in \grS_1\setminus\{\grr\}$. We then have the following
result (see also the Demazure Structure Theorem in \cite[p.~140]{Oda88}).

\begin{proposition}[Demazure]\label{Demprop}
Let $M_\grS$ be a smooth complete toric complex manifold with fan $\grS$. 
\begin{enumerate}
\item The set $R(\grS)\subset \gt^*$ of roots of $M_\grS$ is the root system
  of the algebraic group $\Aut_0(M_\grS)$ with respect to the maximal torus
  $\bbt^c$\textup;
\item $\Aut_0(M_\grS)$ is reductive if and only if and only if $R(\grS)=-R(\grS)$. 
\end{enumerate}
\end{proposition}

To apply this proposition to $M_n(A)$, let $\eps_1,\ldots, \eps_n$ be the
basis of $\gt^*$ dual to $v_1,\ldots, v_n$.
\begin{lemma}\label{ichar} We have $\eps_i\in R(M_n(A))$ if and only if all
entries of the $i$th row of $A$ are non-negative. In particular, $\eps_1 \in
R(M_n(A))$. Similarly, $-\eps_i\in R(M_n(A))$ if and only if all entries of
the $i$th row of $A$ \textup(except $A^i_i=1$\textup) are non-positive, and
$-\eps_1\in R(M_n(A))$.
\end{lemma}

\begin{proof}
By definition $\eps_i(v_j)=\grd_{ij}$, and
\[
\eps_i(u_j)=-\eps_i\biggl(\,\sum_{k=j+1}^nA^j_iv_k\biggr)=-A^j_i
\]
which is $\leq 0$ for all $j$ if and only if $A^j_i\geq 0$ for all $j$.
Similarly $-\eps_i(u_i)=A^i_i=1$ and $-\eps_i$ is $\leq 0$ on all other
normals if $A^j_i\leq 0$ for all $j\neq i$.
\end{proof}

In particular, $\eps_2$ is a root if and only if $A^1_2\geq 0$ and $-\eps_2$
is a root if and only if $A^1_2\leq 0$.
\begin{corollary}\label{A12}
If for any row all the entries below the diagonal have the same sign and are
not all zero, then $\Aut_0(M_n(A))$ is not reductive. In particular, if
$A^1_2\neq 0$ then $\Aut_0(M_n(A))$ is not reductive.
\end{corollary}

This has a strong implication for Bott towers with topological twist zero.
\begin{proposition}\label{0tred}
Let $M_n(A)$ be a Bott tower with topological twist $0$ and $\Aut_0(M_n(A))$
reductive. Then $A=\BOne_n$ and the Bott manifold is the product
$(\bbc\bbp^1)^n$.
\end{proposition}

\begin{proof}
Suppose to the contrary that $\Aut_0(M_n(A))$ is reductive and $A\neq
\BOne_n$.  By Theorem~\ref{mapathm}, for such a Bott tower the $A$ matrix
takes the form $A=2C_n\cdots C_1-\BOne_n$ where $C_k$ is a lower triangular
unipotent matrix with at most one non-zero element $C^{i_k}_k$ below the
diagonal and that lies in the $k$th row.  Since $A\neq \BOne_n$ there exists
$k\in\{2,\ldots,n\}$ such that $C_k\neq \BOne_n$ but $C_i=\BOne_n$ for
$i<k$. So the $k$th row of $C_k$ has the all $0$ entries below the diagonal
except in column $j<k$ whose entry is $C^j_k\neq 0$. But then from the form of
$A$ we see that the first row of $A$ with non-zero entries below the diagonal
has all zeroes except for $A^j_k=2C^j_k\neq 0$. But this contradicts Corollary
\ref{A12}.
\end{proof}

\begin{corollary}\label{ttred}
Let $M_n(A)$ be a Bott tower with twist $t$ and
\begin{equation}\label{Attmatrix}
A=\begin{pmatrix}
&\\ \tilde A & \BZero_t \\ &\\
\begin{matrix}
A^1_{n-t+1} &\cdots & A^{n-t}_{n-t+1}\\
\vdots &\ddots & \vdots \\
A^1_n   &\cdots & A^{n-t}_n
\end{matrix} &
\begin{matrix}
1 & \cdots & 0 \\
\vdots &\ddots & \vdots \\
A^{n-t+1}_n & \cdots & 1
\end{matrix}\end{pmatrix}, \qquad A^i_j\in\bbz,
\end{equation}
where $\tilde{A}\neq \BOne_{n-t}$ has topological twist $0$. Then
$\Aut_0(M_n(A))$ is not reductive.
\end{corollary}

\begin{example}\label{root3}
We can obtain complete results in the case of a stage $3$ Bott tower. Recall
that the set $\grS_1$ of rays or $1$-dimensional cones in $\grS$ is
\[
\grS_1=\{\bbr_{\geq 0}v_1,\bbr_{\geq 0}v_2,\bbr_{\geq 0}v_3,\bbr_{\geq 0}(-v_3),\bbr_{\geq 0}(-v_2-cv_3), \bbr_{\geq 0}(-v_1-av_2-bv_3)\}.
\]
Applying Lemma \ref{ichar} and Corollary \ref{A12}, the reductive cases are
as follows.
\begin{corollary}\label{3red}
The connected component $\Aut_0(M_3(a,b,c))$ of the automorphism group of a
stage $3$ Bott tower $M_3(a,b,c)$ is reductive if and only if $a=0$ and
$bc<0$, or $a=b=c=0$.
\end{corollary}

Thus, we see that the Bott tower $M_3(a,b,c)$ can admit a CSC K\"ahler metric
only if $a=0$ and either $bc<0$ or $b=c=0$.

\end{example}

\subsection{Divisors, support functions and primitive collections}

If $M$ is a complex toric manifold then any Weil divisor is
Cartier;\footnote{This also holds for rational divisors on orbifolds, and some
  of the general facts below have extensions to the orbifold case.}
furthermore, any (integral, rational or real) divisor is linearly equivalent
to a $\bbt^c$-invariant divisor, which in turn is a linear combination
$\sum_{\grr\in\cals} s_\grr D_{u_\grr}$ (with each $s_\grr$ in $\bbz$, $\bbq$ or
$\bbr$ respectively) indexed by the set $\cals$ of rays in the fan as in
Section~\ref{s:toric}. Thus we may identify both the Picard group
$\mathrm{Pic}(M)$ and the Chow group $A_{n-1}(M)$ with the quotient of
$\bbz^\cals\cong\mathrm{Hom}_\bbz(\bbz_\cals,\bbz)$ by the inclusion
$\bfu^\top\colon\mathrm{Hom}_\bbz(\grL,\bbz)\to\mathrm{Hom}_\bbz(\bbz_\cals,\bbz)$
sending $\grl$ to the principal ($\bbz$-)divisor $\sum_{\grr\in\cals}
\grl(u_\grr) D_{u_\grr}$.  Similarly, we have an exact sequence,
\begin{equation}\label{eq:h2}
0 \to \gt^* \to \bbr_\cals^* \to H^2(M,\bbr)\to 0,
\end{equation}
where we use $\gt^*=\mathrm{Hom}_\bbr(\gt,\bbr)$ and
$H^2(M,\bbr)\cong\mathrm{Pic}(M)\otimes_\bbz \bbr$. This last isomorphism
identifies the K\"ahler cone $\calk(M)$ in $H^2(M,\bbr)$ with the ample cone
$\cala(M)$ in $\mathrm{Pic}(M)\otimes_\bbz \bbr\cong
A_{n-1}(M)\otimes_\bbz\bbr$. We note also that the Picard number $\grr(M_n)=n$
for any Bott manifold $M_n$.

Let $PL(\grS)$ be the set of continuous real valued functions on $\gt$
which are linear on each cone in $\grS$. Since $\grS$ is simplicial, the map
$PL(\grS)\to \bbr_\cals^*\cong \bbr_\cals$ sending $\psi$ to
$(\psi(u_\grr))_{\grr\in\cals}$ is a bijection. Hence any invariant divisor
$D$ has a unique \emph{support function} $\psi_D\in PL(\grS)$ such that
$D=\sum_{\grr\in\cals} \psi_D(u_\grr) D_{u_\grr}$, and $D$ is principal iff $\psi_D$
is linear. Hence $H^2(M,\bbr)\cong PL(\grS)/\gt^*$. A function $\psi\in
PL(\grS)$ is {\it convex}\footnote{Following \cite{CoRe09} the convention used
  here is opposite of the convention used in \cite{CoLiSc11}.} if for all
$x,y\in \gt$,
\[
\psi(x)+\psi(y)\geq \psi(x+y);
\]
then $\psi$ is \emph{strictly convex} if it is a different linear function on
each maximal cone $\grs\in\grS_n$. It is a classical result~\cite{Dan78} that
$[D]$ is ample if and only if $\psi_D$ is strictly convex. Then
\[
P_D:=\{\xi\in\gt^*\,|\, \xi\leq \psi_D\}
\]
is a convex polytope dual to $\grS$, i.e., its vertices correspond to maximal
cones in $\grS$ by assigning $\sigma\in\grS_n$ the linear form
$\psi_D\big|_\sigma\in\gt^*$; $P_D$ is often called a \emph{Delzant polytope}.

As observed in~\cite{AsGrMo93,Bat93,Cox97}, the following notion of
Batyrev~\cite{Bat91} (see also \cite{CoRe09,CoLiSc11}) leads to a more
explicit description of the K\"ahler cone $\calk(M)$: we say
$R\subseteq\cals$, or the corresponding collection $\{u_\grr\,|\,\grr\in R\}$
of normals, is a \emph{primitive collection} for a simplicial fan $\grS$, if
$R$ is a minimal element of $P(\cals)\setminus\Phi$, i.e.,
$\{u_\grr\,|\,\grr\in R\}$ does not span a cone of $\grS$, but every proper
subset does. Note that the primitive collections determine $\Phi$, and that
the open subset $\bbc_{\cals,\Phi}$ is the complement of the coordinate planes
(closures of $\bbc_{\cals,R}$) corresponding to primitive collections $R$.

\begin{theorem}[\cite{Bat91,CoRe09}] \label{batthm}
A function $\psi\in PL(\grS)$ is convex if and only if for any primitive
collection $R\subseteq\cals$,
\[
\sum_{\grr\in R}\psi(u_\grr)\geq \psi\Bigl(\,\sum_{\grr\in R} u_{\grr}\Bigr).
\]
It is strictly convex if this inequality is strict for all such $R$.
\end{theorem}

Since $M_n(A)$ is a quotient of $(\bbc^2_*)^n$, we see that the primitive
collections in this case are the pairs $\{u_j,v_j\}$ for $j\in\{1,\ldots,n\}$.
Hence the cones of the fan are the convex hulls of any subset of normals which
does not contain any pair $\{u_j,v_j\}$. There are thus $2^n$ maximal cones,
each containing precisely one normal from each pair $\{u_j,v_j\}$,
$j\in\{1,\ldots,n\}$. Any such family of normals is a $\bbz$-basis of $\grL$.

\subsection{The ample cone and K\"ahler cone}\label{ampleKahlercone}

By applying Theorem~\ref{batthm} to a Bott tower $M_n(A)$, we obtain a
description of its ample cone $\cala(M_n(A))$ and hence its K\"ahler cone
$\calk(M_n(A))$. It will be of convenience later to consider polarized Bott
towers $(M_n(A),[D])$ where $D$ is an ample (integral) $\bbt^c$-invariant Weil
divisor. Since we are interested in the K\"ahler geometry of $M_n(A)$ we
describe such polarized Bott towers in terms of the corresponding (Poincar\'e
dual) K\"ahler class $c_1([D])$.

\begin{corollary}\label{Batcor}
A $\bbt^c$-invariant $\bbr$-divisor $D$ defines an ample
class $[D]$ on $M_n(A)$ if and only if
\begin{equation}\label{divineq}
\psi_D(u_j)+\psi_D(v_j)> \psi_D(u_j+v_j)=\psi_D\biggl(-\sum_{i=j+1}^nA^j_iv_i\biggr)
\end{equation}
for all $j\in\{1,\ldots,n\}$. 
\end{corollary}

If $\eps_1,\ldots,\eps_n$ is the dual basis to $v_1,\ldots,v_n$, then
$\eps_i(u_j)=-A_i^j$ and so the invariant principal divisors are spanned by
\begin{equation}\label{chardiv}
\sum_{\grr\in\cals}\eps_i(u_\grr) D_{u_\grr} = D_{v_i}-\sum_{j=1}^n A_i^j D_{u_j}.
\end{equation}
Note that the divisors $D_{v_i}$ and $D_{u_i}$ are the zero and infinity
sections of $\pi_i\colon \bbp(\BOne\oplus \call_i)\to M_{i-1}$
respectively. In particular, $D_{u_i}:i\in\{1,\ldots, n\}$ is a basis for the
Chow group $A_{n-1}(M_n)$ of Weil divisors on $M_n$, and the Poincar\'e dual of
$D_{u_i}$ can be identified with $x_i$. Hence the $x_i$'s are algebraic classes.

We now write an arbitrary invariant divisor as
\begin{equation}\label{invdiv}
D=\sum_{\grr\in\cals} s_\grr D_{u_\grr} = \sum_i s_i D_{u_i}+\sum_it_i D_{v_i},
\end{equation}
where we denote the coefficients $s_{i'}$ by $t_i$ (just as we denote $u_{i'}$
by $v_i$). From~\eqref{chardiv} we have the relations
\begin{equation}\label{invprindiv2}
D_{v_i}\sim D_{u_i}+\sum_{j=1}^{i-1}A^j_iD_{u_j}
\end{equation}
so that, with $\psi_D(u_i)=s_i$ and $\psi_D(v_i)=t_i$, the ampleness
condition \eqref{divineq} becomes
\begin{equation}\label{ampleineq}
s_j+t_j> \psi_D\biggl(\sum_{i=j+1}^{n}-A^j_iv_i\biggr).
\end{equation}
If $A^j_i\leq 0$ for all $j<i$, then the summands on the right hand side all
belong to the cone spanned by the $v_i$, on which $\psi_D$ is linear, and
so~\eqref{ampleineq} becomes
\begin{equation}\label{ampleineq2}
s_j+t_j + \sum_{i=j+1}^{n}A^j_it_i> 0.
\end{equation} 
Now using \eqref{invprindiv2} we have
\begin{equation}\label{arbdivu}
D\sim \sum_is_iD_{u_i}+\sum_it_i\biggl(D_{u_i}+\sum_{j=1}^{i-1}A^j_iD_{u_j}\biggr)=
\sum_{j=1}^n\biggl(s_j+t_j+\sum_{i=j+1}^nt_iA^j_i\biggr)D_{u_j}.
\end{equation}
Thus, if we define $r_j=s_j+t_j+\sum_{i=j+1}^nt_jA^j_i$, we arrive at

\begin{proposition}\label{Kahcone}
If $A^j_i\leq 0$ for $j<i$ then the ample cone $\cala(M_n(A))$ is the entire
first orthant of divisor classes $[D]=\sum_{j=1}^n r_j [D_{u_j}]$ defined by
$r_j>0$ for all $j=1,\ldots,n$.  Hence $\calk(M_n(A))$ is the entire first
orthant in the basis $x_i:i\in\{1,\ldots, n\}$.
\end{proposition}

We do not have to restrict to the basis $[D_{u_i}]:i\in\{1,\ldots, n\}$: from
\eqref{invprindiv2} we see that there are $2^{n-1}$ $\bbt^c$-invariant bases
$[D_1],\ldots, [D_n]$ of $A_{n-1}(M_n(A))$, with $D_1=D_{u_1}\sim D_{v_1}$ and
either $D_i=D_{u_i}$ or $D_i=D_{v_i}$.

\begin{lemma}\label{l:everypos} A $\bbt^c$-invariant divisor $D$ is ample if
and only if in every $\bbt^c$-invariant basis $[D_1],\ldots, [D_n]$ of
$A_{n-1}(M_n(A))$ with $D_1=D_{u_1}$ and either $D_i=D_{u_i}$ or
$D_i=D_{v_i}$, the coefficients $r_i$ of $[D]=\sum_{i=1}^n r_i [D_i]$ are all
positive.
\end{lemma}
\begin{proof} By the Toric Nakai Criterion~\cite[Theorem 2.18]{Oda88},
a $\bbt^c$-invariant divisor $D$ is ample if and only if $D\cdot C>0$ for all
$\bbt^c$-invariant curves $C$ corresponding to some cone in $\grS(n-1)$.
Equivalently, in $H^2(M_n(A),\bbr)$, $c_1([D])$ is in $\calk(M_n(A))$ if and
only if $c_1([D])z_1\cdots \hat z_j\cdots z_n>0$, where $j\in\{1,\ldots, n\}$
and for all $z_i\in H^2(M_n(A),\bbr)$ with $z_i\in\{x_i,y_i\}$.  However, if
we express $[D]$ in the basis $[D_1],\ldots,[D_n]$ with $D_1=D_{u_1}$ and
$D_i=D_{v_i}$ (i.e., $c_1([D_i])=y_i$) if and only if $z_i=x_i$ (and
$D_i=D_{u_i}$ otherwise), then using $x_i y_i=0$, this cup product is the
coefficient of $D_j$ in this basis.
\end{proof}

Hence in every such basis, the ample cone is contained in the first orthant.
Proposition~\ref{Kahcone} shows that in some cases, it can be equal to the
entire first orthant in some such basis. However, as we shall see, this need
not always be the case. We now explore this through a series of examples
covering the special cases we are interested in.

\begin{example}\label{Hirex} We first consider the well-understood Hirzebruch
surfaces $\calh_a$.  A Hirzebruch surface $\calh_a=M_2(a)$, as a Bott tower
with matrix $A$, $A^1_2=a$, has normals $u_1=-v_1-av_2,u_2=-v_2,v_1,v_2$.  A
diagram of the fan is given in~\cite[Example 4.1.8]{CoLiSc11}.
\footnote{Note that our convention is the opposite to the convention used in
  \cite{CoLiSc11}: our $\calh_a$ is their $\calh_{-a}$.}
There are two primitive collections, $\{u_1,v_1\}$ and
$\{u_2,v_2\}$, and by~\eqref{chardiv}, the principal divisors are spanned by
$D_{v_1}-D_{u_1}$ and $D_{v_2}-(a D_{u_1}+D_{u_2})$. Then an invariant divisor
takes the form
\begin{equation*}
D=\sum_{\grr\in\cals} s_\grr D_{u_\grr} = r_1 D_{u_1}+r_2 D_{2}
\end{equation*}
where we can take $D_2$ to be either $D_{u_2}$ or $D_{v_2}$.  By
\eqref{ampleineq} $D$ defines an ample class if $r_2>0$ and $r_1>
\psi_D(-av_2)$. So if $a\leq 0$ we take $D_2=D_{u_2}$ which gives $r_1>0$ and
shows that in this case the ample cone $\cala(\calh_a)$, and hence the
K\"ahler cone $\calk(\calh_a)$, is the entire first quadrant $r_1>0,r_2>0$.

If on the other hand $a>0$ we take $D_2=D_{v_2}$ giving
$r_1>\psi_D(-av_2)=\psi_D(au_2)=0$ which again gives $\cala(\calh_a)$ as the
entire first quadrant.
\end{example}

Next we consider the case of stage $3$ Bott manifolds.
\begin{example}\label{stage3ex} 
We compute the ample cone for $n=3$ in which case we have
$a=A^1_2,b=A^1_3,c=A^2_3$.  For a height $3$ Bott tower $M_3(a,b,c)$, the
fan $\grS_{a,b,c}$ has normals
\begin{equation}\label{norveceq}
u_1=-v_1-av_2-bv_3,\quad u_2 = -v_2-cv_3,\quad u_3=-v_3,\quad
v_1,\quad v_2,\quad v_3,
\end{equation}
and by~\eqref{chardiv} the invariant principal divisors are spanned by
\begin{equation}\label{3stagedivrel}
D_{v_1} - D_{u_1},\quad
D_{v_2} - (a D_{u_1} + D_{u_2}),\quad 
D_{v_3} - (b D_{u_1} + c D_{u_2} + D_{u_3}).
\end{equation}
Writing $\psi_D(u_i)=s_i$ and $\psi_D(v_i)=t_i$ (for an invariant divisor
$D$), we see by~\eqref{divineq} that the ample cone (or equivalently the
K\"ahler cone) is determined by the inequalities
\begin{equation}\label{stage3ineq}
s_1+t_1>\psi_D(-av_2-bv_3),\quad s_2+t_2>\psi_D(-cv_3),\quad s_3+t_3>0.
\end{equation}
If $a,b,c\leq 0$, the first two relations reduce to $s_1 + t_1 + at_2+bt_3 >0$
and $s_2+t_2 + c t_3 >0$. Using~\eqref{3stagedivrel} we
may write
\begin{equation}\label{arbdiv3}
D=\sum_{\grr\in\cals}s_\grr D_{u_\grr}
=\bigl(s_1+t_1+a t_2+bt_3\bigr)D_{u_1}+(s_2+t_2 +ct_3)D_{u_2} +(s_3+t_3)D_{u_3}
\end{equation}
and so if we parametrize $H^2(M_3,\bbr)$ by $r_1 PD([D_{u_1}])+ r_2
PD([D_{u_2}])+ r_3 PD([D_{u_3}]$, the K\"ahler cone is the first octant,
$r_1>0,r_2>0,r_3>0$. Equivalently, $\{[D_{u_1}],[D_{u_2}],[D_{u_3}]\}$ is a
set of generators for the Chow group $A_{2}(M_3)$ of Weil divisors on $M_3$.

Now consider the general case where $a,b,c$ have any signs. We consider the
possible primitive collections whose elements are generators of
$A_{2}(M_3)$. We note from the first of Equations \eqref{3stagedivrel} that
$[D_{u_1}]$ and $[D_{v_1}]$ are equal, so there are only four such sets of
generators, namely
\begin{gather}\
\label{bases1} \{[D_{u_1}],[D_{u_2}],[D_{u_3}]\}, \\
\{[D_{u_1}],[D_{u_2}],[D_{v_3}]\}, \\
\{[D_{u_1}],[D_{v_2}],[D_{u_3}]\}, \\
\label{bases4} \{[D_{u_1}],[D_{v_2}],[D_{v_3}]\}. 
\end{gather}
\begin{proposition}\label{3stamcone}
For a height $3$ Bott tower $M_3(a,b,c)$ there is a choice of invariant divisors \eqref{bases1}--\eqref{bases4} such that the ample cone is the full first
octant if and only if one of the following two conditions hold\textup:
\begin{itemize}
\item $a\leq 0$ and $bc\geq 0$\textup; or
\item $a\geq 0$ and $(b-ac)c\geq 0$.
\end{itemize}
\end{proposition}

\begin{proof}
It is enough to check this on the four sets of generators given above, that is
we write $D=r_1D_1+r_2D_2+r_3D_3$ where $D_1=D_{u_1}$, $D_2$ is either
$D_{u_2}$ or $D_{v_2}$, and $D_3$ is either $D_{u_3}$ or $D_{v_3}$. Then
\eqref{stage3ineq} becomes
\begin{equation}\label{stage3ineq2}
r_1>\psi_D(-av_2-bv_3),\quad r_2>\psi_D(-cv_3),\quad r_3>0.
\end{equation}
First suppose that $a\leq 0$; then in order that $r_1$ should give the full
range $0<r_1<\infty$ the first of \eqref{stage3ineq2} implies that we should
choose $D_2=D_{u_2}$. In addition for both $r_1$ and $r_2$ to give the full
range \eqref{stage3ineq2} implies that $b$ and $c$ cannot have opposite
signs. If instead $b,c\leq 0$, then choosing
$D=r_1D_{u_1}+r_2D_{u_2}+r_3D_{u_3}$ shows that the ample cone is the entire
first octant; whereas, if $b\geq 0,~c>0$ or $b>0,~c\geq 0$ we can choose
$D=r_1D_{u_1}+r_2D_{u_2}+r_3D_{v_3}$ to get the ample cone to be the entire
first octant. This gives the first two conditions of the proposition. The case
$a\geq 0$ is analogous, but also follows immediately from the equivalence
$M_3(a,b,c)\sim M_3(-a,b-ac,c)$.
\end{proof}

Both of the conditions in Proposition \ref{3stamcone} imply that $(2b-ac)c
\geq 0$, i.e., $p_1(M)$ is a nonnegative multiple of $x_1 x_2 = (1/2)
c_1(O(1,1))^2$.

\begin{remark}\label{nooct}
Here we enumerate the complimentary inequalities to those of Proposition
\ref{3stamcone}, that is, those cases where the ample cone is not the full
first octant. If $a<0$ we must take $D=r_1D_{u_1}+r_2D_{u_2}+r_3D_3$, and from
the proposition, we can obtain the entire first octant unless $b<0$ and $c>0$,
or $b>0$ and $c<0$. In the first case ($b<0$, $c>0$), if we choose
$D_3=D_{u_3}$, we get $r_2>\psi_D(c u_3)=c r_3$, while if we choose
$D_3=D_{v_3}$, we get $r_1> -b r_3$. Hence neither choice yields the entire
first octant. In the second case ($b>0$, $c<0$), if we choose $D_3=D_{v_3}$,
we get $r_2>\psi_D(-cv_3)=-cr_3$, while if we choose $D_3=D_{u_3}$ we get
$r_1>br_3$, and again neither choice yields the entire first octant.

A similar analysis holds for $a>0$. In this case we must take
$D=r_1D_{u_1}+r_2D_{v_2}+r_3D_3$, and we obtain the entire first octant unless
$c<0$ and $b-ac>0$, or $c>0$ and $b-ac<0$. For $c<0$ and $b-ac>0$, if we
choose $D_3=D_{u_3}$ we get
\[
r_1>\psi_D(-av_2-bv_3)=\psi_D(a(u_2+cv_3)-bv_3)=\psi_D(au_2-(b-ac)v_3)
=\psi_D(au_2+(b-ac)u_3)
\]
which implies $r_1>(b-ac)r_3, r_2>0$ when $b>ac$; whereas, if we choose
$D_3=D_{v_3}$ this and the second of Equations \eqref{stage3ineq} gives
$r_1>0,r_2>-cr_3$. Thus neither choice yields the entire first octant.
For $c>0$ and $b-ac<0$, we use instead $r_1>\psi_D(au_2-(b-ac)v_3)$
and $r_2>\psi_D(c u_3)$ to obtain the same conclusion.
\end{remark}
\end{example}

\begin{example}\label{1twist2}
Here we consider again the Bott towers $M_{N+1}(\bfk)$ of
Example~\ref{twist1-examples} whose $A$ matrix takes the form
\eqref{genkexample}. For these Bott manifolds we write
$D=\sum_{i=1}^Nr_iD_{u_i}+r_{N+1}D_{N+1}$. Then from \eqref{ampleineq} the
ample cone is determined by inequalities
\[
r_j>\psi_D(-k_jv_{N+1})=\psi_D(k_ju_{N+1})~~\text{and}~r_{N+1}>0.
\]
Hence, if $k_j\leq 0$ for all $j=1,\ldots,N$ we choose $D_{N+1}=D_{u_{N+1}}$
giving $\cala(M_{N+1}(A))$ to be the entire first orthant; whereas, if
$k_j\geq 0$ for all $j=1,\ldots,N$ we choose $D_{N+1}=D_{v_{N+1}}$ again
giving $\cala(M_{N+1}(\bfk))$ to be the entire first orthant. However, when
not all $k_j$ have the same sign, we do not get the entire first
orthant. Without loss of generality we assume that $k_j\leq 0$ for
$j=1,\ldots,m$ and $k_j>0$ for $j=m+1,\ldots,N$. If $m\leq \bigl\lfloor
\frac{N}{2}\bigr\rfloor$ we choose $D_{N+1}=D_{v_{N+1}}$, and if $m>\bigl\lfloor
\frac{N}{2}\bigr\rfloor$ we choose $D_{N+1}=D_{u_{N+1}}$. In the former case the
ample cone is given by $r_j>-k_jr_{N+1}$ for $j=1,\dots,m$ and $r_j>0$ for
$j=m+1,\dots,N+1$; whereas, in the latter case it is determined by $r_j>0$ for
$j=1,\ldots,m,N+1$ and $r_j>k_jr_{N+1}$ for $j=m+1,\ldots,N$.
\end{example}

\begin{example} 
For the Bott towers $M_{N+1}(\bfl,\bfk)$ of Example~\ref{twist2-examples}, the
primitive collections are $\{u_j,v_j\}:j\in\{1,\ldots, N+1\}$ where
$u_j=-v_j-l_jv_N-k_jv_{N+1}$ for $j=1,\ldots,N-1$, $u_N=-v_N-k_Nv_{N+1}$ and
$u_{N+1}=-v_{N+1}$.  So an invariant divisor can be written as
\[
D=\sum_{i=1}^{N-1}r_iD_{u_i}+r_{N}D_{N}+r_{N+1}D_{N+1}
\]
where  $D_N$ is $D_{v_N}$ or $D_{u_N}$ and $D_{N+1}$  is $D_{v_{N+1}}$ or $D_{u_{N+1}}$.
So from \eqref{ampleineq} we have $r_{N+1}>0$, and
$r_{N}>\psi_D(-k_Nv_{N+1})=\psi_D(k_Nu_{N+1})$ and for $j=1,\ldots,N-1$,
\begin{equation}\label{2tample}
 r_j> \psi_D(-l_jv_{N}-k_jv_{N+1})=\psi_D(-l_jv_{N}+k_ju_{N+1})
=\psi_D(l_j u_N + (k_j-l_j k_N) u_{N+1})
\end{equation}
since $v_{N+1}=-u_{N+1}$ and $v_N=-u_N+k_Nu_{N+1}$. One can proceed with an
analysis similar to Example \ref{stage3ex} to determine for which
$M_{N+1}(\bfl,\bfk)$ the K\"ahler cone is the entire first orthant, and for
which it is not; however, here we just consider a special case of interest in
Proposition \ref{2textremal} below.

\begin{subexample}
We take $N=2m+1$ with the matrix
\begin{equation}\label{2tspecial}
A=\begin{pmatrix}
\BOne_m&\BZero_m& {\bf 0}_m & {\bf 0}_m \\
\BZero_m & \BOne_m & {\bf 0}_m & {\bf 0}_m \\
q{\bf 1}_m& -q{\bf 1}_m & 1 & 0 \\
q(k+1){\bf 1}_m& q(k-1){\bf 1}_m &2&1\\
\end{pmatrix}
\end{equation}
with $q,k-1\in\bbz^+$, and where $\BOne_m$ is the $m$ by $m$ identity matrix,
$\BZero_m$ the $m$ by $m$ zero matrix, ${\bf 0}_m$ is the zero column vector
in $\bbr^m$, and ${\bf 1}_m =(1\;1\cdots1)$ is a row vector in $\bbr^m$.

\begin{proposition}\label{2tkcone}
The ample cone, hence the K\"ahler cone, of the Bott tower $M_{2m+2}(A)$ with $A$ given by \eqref{2tspecial} is represented by
\begin{gather*}
r_j>0\quad \text{for $j\in\{1,\ldots,m\}$};
\qquad r_j>qr_{2m+1}\quad\text{for $j\in\{m+1,\ldots,2m\}$};\\
r_{2m+1}>0; \qquad r_{2m+2}>0.
\end{gather*}
\end{proposition}

\begin{proof}
To determine the ample cone, we note that $r_{2m+2}>0$, $r_{2m+1}>\psi_D(-2v_{2m+2})=\psi_D(2u_{2m+2})=0$ if $D_{2m+2}=D_{v_{2m+2}}$. For $j\in\{1,\ldots,m\}$ we get
from \eqref{2tample} that
\[
r_j > \psi_D(-qv_{2m+1}-q(k+1)v_{2m+2})=\psi_D(qu_{2m+1}+q(k-1)u_{2m+2}),
\]
which is zero if we also choose $D_{2m+1}=D_{v_{2m+1}}$. But then for
$j\in\{m+1,\ldots,2m\}$ we get
$r_j>\psi_D(qv_{2m+1}+q(k-1)u_{2m+2})=qr_{2m+1}$ as stated.
\end{proof}

Since the last row of the matrix \eqref{2tspecial} has all positive entries,
the connected component of the automorphism group of $M_{2m+2}(A)$ is not
reductive by Corollary \ref{A12}. It follows that $M_{2m+2}(A)$ does not
admit a CSC K\"ahler metric.

\end{subexample}
\end{example}

\subsection{Fano Bott manifolds}

It follows immediately from the fact that Bott manifolds are compact toric or
from Equations \eqref{Amatrix}, \eqref{xyeq}, and \eqref{c1} that the first
Chern class of a Bott manifold is either positive semidefinite or
indefinite,\footnote{If $M_n(A)$ admits a K\"ahler metric with negative
  semidefinite Ricci form, then every holomorphic vector field is parallel,
  contradicting the presence of nontrivial holomorphic vector fields with
  zeros.} and for most it is indefinite. Recall that a complex manifold
$(M,J)$ is {\it Fano} if the first Chern class $c_1(M)$ is positive definite,
that is, $c_1(M)$ lies in the K\"ahler cone, in which case, the K\"ahler
classes $[\omega]$ in the span of $c_1(M)$ (and any K\"ahler metrics $\omega$
in such a class) are said to be \emph{monotone}.  Equivalently, the
anti-canonical divisor lies in the ample cone.  There is a lot known about
toric Fano manifolds with $n\leq 4$ (see for example
\cite{Bat81,WaWa82,Sak86,Mab87,Bat91,Nak93,Nak94,Baty99,BatSe99}).

\subsubsection*{Fano Bott manifolds with topological twist $0$}
\begin{proposition}\label{otFprop}
The only Fano Bott tower $M_n(A)$ with topological twist $0$ is the product
$(\bbc\bbp^1)^n$ represented by $A=\BOne_n$.
\end{proposition}
\begin{proof}
If $M_n(A)$ is a Bott tower with topological twist $0$, then by
Theorem~\ref{mapathm}, $A$ must satisfy $A=2C_n\cdots C_1-\BOne_n$ where $C_1$
is the identity matrix and for each $k$ the unipotent integer-valued matrix
$C_k$ can have at most one non-zero element below the diagonal and it is in
the $k$th row.  Hence any off-diagonal element of $A$ must be a multiple of
$2$, so the non-zero off-diagonal elements have absolute value at least $2$.

Now the Fano condition is that the first Chern class
\[
c_1(M_n)=\sum_{i=1}^n(x_i+y_i)
\]
is ample, hence by Lemma~\ref{l:everypos}, it has positive coefficients with
respect to each of the $2^{n-1}$ invariant bases of $H^2(M_n,\bbz)$. With
respect to the basis $x_i:i\in\{1,\ldots, n\}$ we have
\begin{equation}\label{c1xi}
c_1(M_n)=(2+A^1_2+\cdots +A^1_n)x_1+\cdots +(2+A^{n-1}_n)x_{n-1}+2x_n.
\end{equation}
To obtain $c_1(M_n)$ in the other bases we use the general formulae
\begin{equation}\label{ykxk}
y_k=x_k+A^1_kx_1+\cdots +A^{k-1}_kx_{k-1}. 
\end{equation}
We begin by making the substitution
$$x_n=y_n -(A^1_nx_1+\cdots +A^{n-1}_nx_{n-1})$$
in \eqref{c1xi} which gives
\begin{equation}\label{c1xyn}
c_1(M_n)=(2+A^1_2+\cdots -A^1_n)x_1+\cdots +(2-A^{n-1}_n)x_{n-1}+2y_n.
\end{equation}
Then the positivity of $c_1(M_n)$ in both \eqref{c1xi} and \eqref{c1xyn}
together with the fact that the off diagonal elements of $A$ are even implies
$A^{n-1}_n=0$. We can now make the substitution for $x_{n-1}$ in the Equations
\eqref{c1xi} and \eqref{c1xyn} giving the 4 equations for $c_1(M_n)$ ending
with a choice of the invariant pairs $\{x_{n-1},y_{n-1}\}$ and
$\{x_n,y_n\}$. These four equations then imply $A^{n-2}_k=0$ for
$k=n-1,n$. Now assume by induction that $A^{j'}_k=0$ for all $j'>j$ and all
$k>j'$. Then in the $x_i$ basis $c_1(M_n)$ takes the form
$$c_1(M_n)=(2+A^1_2+\cdots +A^1_n)x_1+\cdots +(2+A^{j}_{j+1}+\cdots +A^j_n)x_{j}+2x_{j+1}+\cdots +2x_n.$$
The $2^{n-j}$ substitutions using \eqref{ykxk} gives the coefficient of $x_j$ with all combinations of signs in front of $A^j_k$ and these must all be positive. This implies that $A^j_k=0$ for all $k>j$ which completes the induction.
\end{proof}

\begin{remark}
Notice that in the proof, the 0 topological twist condition has only been used
to show that the off diagonal coefficients of $A$ are all even. So for the
general Fano case with higher topological twist, there must be an off-diagonal
coefficient equal to $\pm1$. In fact one can give a characterization of all
Fano Bott manifolds in this way, as has recently been done for generalized
Bott manifolds \cite{Suy18}.
\end{remark}

\subsubsection*{Fano Bott manifolds with topological twist $1$}
For the Bott towers $M_{N+1}(\bfk)$, the {\it monotone} case when
$c_1(M_{N+1}(\bfk))$ is a K\"ahler class which has been well studied
\cite{KoSa86,Koi90}.

\begin{proposition}\label{1tFano}
A Fano Bott tower of the type $M_{N+1}(\bfk)$ can be put in the form
\[
M_{N+1}(\bfk)\cong M_{m}((\overbrace{\pm 1,\ldots,\pm 1}^{m-1}))\times
\overbrace{\bbc\bbp^1\times\cdots\times\bbc\bbp^1}^{N+1-m}
\]
for some $m=2,\ldots,N+1$. Moreover, $M_{N+1}(\bfk)$ with the monotone K\"ahler class admits a K\"ahler--Ricci soliton which is K\"ahler--Einstein if and only if the number of $-1$ equals the number of $+1$ in $\bfk$.
\end{proposition}

\begin{proof}
For the Bott towers $M_{N+1}(\bfk)$ we have
\begin{equation*}
c_1(M_{N+1}(\bfk))=\sum_{i=1}^{N}(2 +k_i)x_i +2x_{N+1}
=\sum_{i=1}^{N}(2 -k_i)x_i +2y_{N+1}.
\end{equation*}
So by Lemma~\ref{l:everypos}, $M_{N+1}(\bfk)$ is Fano if and only if $k_i=0,\pm 1$. By an element of $BC_n$ we can put $\bfk$ in the form $(1,\ldots,1,-1,\ldots,-1,0\ldots,0)$ where the number of $0$'s is $N+1-m$, the number of $+1$'s is $r$ and the number of $-1$'s is $s$ with $r+s=m-1$ which implies the given form. 

It follows from \cite[Section 3]{ACGT08b} together with \cite{Koi90,WaZh04} that $M_{N+1}(\bfk)$ admits a K\"ahler--Ricci soliton which is K\"ahler--Einstein if and only if the number of $-1$'s in $\bfk$ equals the number of $+1$'s in $\bfk$.
\end{proof}

\subsubsection*{Fano Bott manifolds with topological twist $2$}

For the Bott towers $M_{N+1}(\bfl,\bfk)$ we have
\begin{align}\label{2tChern}
c_1(M_{N+1}(\bfl,\bfk))&=\sum_{i=1}^{N-1}(2+l_i +k_i)x_i +(2+k_N)x_N+2x_{N+1} \\
            \notag&=\sum_{i=1}^{N-1}(2+l_i -k_i)x_i +(2-k_N)x_N+2y_{N+1} \\
           \notag&=\sum_{i=1}^{N-1}(2-l_i +k_i-k_Nl_i)x_i +(2+k_N)y_N+2x_{N+1} \\
           \notag&=\sum_{i=1}^{N-1}(2-l_i -k_i+k_Nl_i)x_i +(2-k_N)y_N+2y_{N+1} 
\end{align}

\begin{lemma}\label{2tineq}
If $M_{N+1}(\bfl,\bfk)$ is Fano then $l_i=0,\pm 1$ and $k_i=0,\pm 1$. 
\end{lemma}

\begin{proof}
That $k_N=0,\pm 1$ is obvious, and that $l_i=0,\pm 1$ follows from adding the
first two and the last two equations. Then one checks that $k_i=0,\pm 1$ as
well.
\end{proof}

However, not all possibilities in Lemma \ref{2tineq} can occur. For example,
from the first two of Equations \eqref{2tChern} we see that for
$i=1,\ldots,N-1$ if $l_i=-1$ then we must have $k_i=0$; whereas, if $l_i=1$
then we must have $k_i=k_N=-1$ or $k_i=k_N=1$. Referring to the equivalence
relations of Example \ref{2tequiv} we see that this last case implies that
$k'_i=k_i-k_N=0$. So up to fiber inversion equivalence we can assume that
$l_i=-1$ or $0$. If $l_i=0$ then all possibilities for $k_i,k_N$ can
occur. Since we have already considered Bott towers of twist $\leq 1$, we can
assume that not all $l_i$ nor all $k_i$ vanish. Up to equivalence we can
assume that $l_i=-1$ and $k_i=0$ for $i=1,\ldots,m$, and $l_i=0$ for
$i=m+1,\ldots,N-1$ with $k_i=0,\pm 1$ for $i=m+1,\ldots,N$ and not all $k_i$
vanish. So we have arrived at
\begin{proposition}\label{2tmonoprop}
The Bott tower $M_{N+1}(\bfl,\bfk)$ is Fano if and only if, up to equivalence, the matrix $A$ takes the form
\begin{equation}\label{mono2example}
A =\begin{pmatrix}
\BOne_r   & \BZero_{r\times s} &0 &\cdots & 0 &0 \\
\BZero_{s\times r}   & \BOne_s &0 &\cdots & 0 &0 \\
\vdots &\vdots&\vdots&\ddots&\vdots&\vdots \\
-{\bf 1}_r  &{\bf 1}_s & 0&\cdots &1 & 0 \\
{\bf 0}_r & k_N{\bf 1}_s& k_m& \cdots& k_N & 1
\end{pmatrix},
\end{equation}
where $\BOne_r$ is the $r$ by $r$ identity matrix, $r+s=m-1$, ${\bf 1}_r=(\overbrace{1,\ldots,1}^r)$, and for $i=m,\ldots,N$, $k_i=0,\pm 1$ and not all $k_i$ vanish.
\end{proposition}

These Bott manifolds admit K\"ahler--Ricci solitons by \cite{WaZh04} and a
K\"ahler--Einstein metric when the Futaki invariant vanishes. It would be
interesting to see if there are any new K\"ahler--Einstein metrics in this
class of Bott manifolds. It follows from the Nakagawa's classification
\cite{Nak93,Nak94} of K\"ahler--Einstein Fano toric 4-folds that there are no
such K\"ahler--Einstein metrics when $N=3$. In fact, from the classification
we have
\begin{corollary}\label{Nakcor}
The only stage $4$ Bott manifolds that admit K\"ahler--Einstein metrics are
the standard product $(\bbc\bbp^1)^4$ and $\bbp(\BOne\oplus \calo(1,-1))\times
\bbc\bbp^1$.
\end{corollary}

We mention also that the Bott towers given by the matrix $A$ in
\eqref{2tspecial} are not Fano.

\begin{example}\label{3stageBottmono}
We specialize these results to the height $3$ Bott towers $M_3(a,b,c)$.
Here, the first Chern class in the distinguished bases is
\begin{equation}\label{B3c1}\begin{split}
c_1(M_3)&=(2+a+b)x_1+(2+c)x_2+2x_3, \\
  \notag   &=(2+a-b)x_1+(2-c)x_2+2y_3, \\
   \notag  &=(2-a+b-ac)x_1+(2+c)y_2+2x_3, \\
    \notag &=(2-a-b+ac)x_1+(2-c)y_2+2y_3.
\end{split}\end{equation}

\begin{proposition}\label{monoprop}
A height $3$ Bott tower $M_3(a,b,c)$ is Fano if and only if $(a,b,c)$ is one of the following:
\begin{itemize}
\item $(1,0,0),((1,1,1),(1,-1,-1)$,
\item $(0,b,c)$ with $b=0,\pm 1$ and $c=0,\pm 1$.
\item $(-1,0,0),(-1,0,1),(-1,0,-1)$.
\end{itemize}
\end{proposition}

The Bott manifolds are the equivalences classes described in Example
\ref{st3equiv}.  We give a representative in each equivalence
class. Explicitly,
\begin{proposition}\label{mono3stage}
Up to equivalence a Fano stage $3$ Bott manifold is one of the following
\[
M_3,(0,0,0),\; M_3(0,0,1),\;M_3(0,1,1),\;M_3(0,1,-1),\;M_3(-1,0,1).
\]
\end{proposition}

The first Pontrjagin class $p_1(M_3(a,b,c))=c(2b-ac)x_1x_2$ is a diffeomorphism invariant. We have
\begin{align*}
p_1(M_3(0,0,0))&=p_1(M_3(0,0,1))=0, &p_1(M_3(0,1,1))&=2x_1x_2,, \\      
p_1(M_3(0,1,-1))&=-2x_1x_2, &p_1(M_3(-1,0,1))&=x_1x_2.
\end{align*}
The first of these manifolds is the $(\bbc\bbp^1)^3$, and the second is
$\bbc\bbp^1\times \calh_1$ which has topological twist $1$. The third and
fourth are $\bbc\bbp^1$ bundles over $\bbc\bbp^1\times \bbc\bbp^1$ which also
have topological twist $1$, while the last is a Bott manifold with topological
twist $2$. Note that the Picard number of a stage $3$ Bott manifold is $3$, so
it must be of type III in the list of the main theorem in \cite{WaWa82}. By
\cite{WaZh04} the monotone K\"ahler class admits a K\"ahler--Ricci soliton;
however, from~\cite[Remark 2.5]{Mab87} we have
\begin{corollary}\label{3KE}
The only stage $3$ Bott manifolds that admit a K\"ahler--Einstein metric are
the standard product $M_3(0,0,0)=(\bbc\bbp^1)^3$ and the Koiso--Sakane
projective bundle $M_3(0,1,-1)=\bbp(\BOne\oplus \calo(1,-1))$.
\end{corollary}
\end{example}

\section{The symplectic viewpoint}\label{symp}

\subsection{Toric symplectic manifolds and finiteness}

Any K\"ahler manifold $(M,g,J,\gro)$ has an underlying symplectic manifold
$(M,\gro)$. The symplectic viewpoint studies conversely the space of complex
structures $J$ on a given symplectic manifold $(M,\gro)$ for which $\gro$ is
the K\"ahler form of a K\"ahler metric $g$ on $(M,J)$.

\begin{definition}\label{compsymp}
A complex $n$-manifold $M$ is {\it compatible} with a symplectic $2n$-manifold
$(X,\gro)$ if there is a diffeomorphism $f\colon M \to X$ such that $f^*\gro$
is a K\"ahler form on $M$.  We also say that the symplectic manifold
$(X,\gro)$ is {\it of K\"ahler type} (with respect to $M$).
\end{definition}

If $f\colon M\to X$ is such a diffeomorphism and $M$ is toric with respect to
a complex $n$-torus $\bbt^c$, then $T^n=f_*(\bbt^c)\cap\Symp(X,\gro)$ is a
(real) hamiltonian $n$-torus in $\Symp(X,\gro)$ with momentum map $\mu\colon
X\to \gt^*$ (where $\gt$ is the Lie algebra of $T^n$). Then $(X,\gro,T^n)$ is
a \emph{toric symplectic manifold} and (for $X$ compact) the image of $\mu$ is
a compact convex polytope $P$ in $\gt^*$, called the \emph{Delzant polytope}
of $(X,\omega,T^n)$, such that the collection of cones in $\gt$ dual to the
faces of $P$ is the fan of $M$.

Hamiltonian $n$-tori $T^n$ and $\tilde{T}^n$ in $\Symp(X,\gro)$ define
\emph{equivalent} toric symplectic manifolds if there is a symplectomorphism
that intertwines them, i.e., they are conjugate as subgroups of
$\Symp(X,\gro)$. Thus the set $\calc_n(X,\gro)$ of conjugacy classes of
hamiltonian $n$-tori in $\Symp(X,\gro)$ parametrizes equivalence classes of
toric symplectic structures on $(X,\gro)$.  In~\cite[Prop.~3.1]{McD11}, McDuff
shows that the set $\calc_n(X,\gro)$ is finite.

\begin{proposition}\label{finite} Let $(X,\gro)$ be a symplectic $2n$-manifold.
Then there are finitely many biholomorphism classes of toric complex manifolds
compatible with $(X,\gro)$, and this set is naturally a quotient of
$\calc_n(X,\gro)$.
\end{proposition}
\begin{proof} A hamiltonian $n$-torus $T^n$ in $\Symp(X,\gro)$ determines
a Delzant polytope $P$ in $\gt^*$ and hence a fan in $\gt$. Let $M$ be the
toric complex manifold constructed from the fan as a quotient of $\bbc_\cals$
where $\cals$ is the set of rays in the fan. Then Delzant's
Theorem~\cite{Del88} says that $(X,\gro)$ is $T^n$-equivariantly
symplectomorphic to a symplectic quotient of $\bbc_\cals$ canonically
diffeomorphic to $M$, in such a way that $M$ is compatible with $(X,\omega)$
and $T^n$ is the induced torus. Furthermore, if $T^n$ and $\tilde{T}^n$ are
conjugate hamiltonian $n$-tori, then by~\cite{Del88}, their Delzant polytopes
$P$ and $\tilde P$ are equivalent by an affine transformation
$\gt\to\tilde\gt$ whose linear part is integral, hence the associated fans are
equivalent, and so the corresponding toric complex manifolds are equivariantly
biholomorphic---see for example \cite[Theorem 1.13]{Oda88}.

Thus there is a well-defined map from $\calc_n(X,\gro)$ to biholomorphism
classes of toric complex manifolds compatible with $(X,\gro)$, and any such
biholomorphism class arises in this way, so the map is a surjection. The
biholomorphism classes thus form a quotient of $\calc_n(X,\gro)$, which is
finite by~\cite[Prop.~3.1]{McD11}.
\end{proof}

To apply this result to Bott towers, we let $\BT(X,\gro)$ denote the set of
Bott towers that are compatible with $(X,\gro)$ and $\BM(X,\gro)$ the set of
their biholomorphism classes (elements of $\BM=\BT_0/\BT_1$). When this set is
nonempty, we say $(X,\gro)$ has \emph{Bott type}. If $M_n(A)$ is compatible
with $(X,\gro)$, so are all elements of its biholomorphism class.

\begin{theorem}\label{conjclassthm}
Let $(X,\gro)$ be a symplectic $2n$-manifold of Bott type. Then there is a
canonical surjection $\calc_n(X,\gro) \to \BM(X,\gro)$ sending the conjugacy
class of a hamiltonian $n$-torus $T^n$ to the class of Bott towers compatible
with $(X,\gro)$ and $T^n$. In particular, $\BM(X,\gro)$ is finite.
\end{theorem}
\begin{proof} Since $(X,\gro)$ is of Bott type, its cohomology ring is a Bott
quadratic algebra of rank $n$ by \cite[Theorem 4.2]{ChSu11}. Hence if $M$ is a
toric complex manifold compatible with $(X,\gro)$, the Delzant polytope of the
induced toric symplectic structure is combinatorially equivalent to an
$n$-cube. But then \cite[Theorem 3.4]{MaPa08} implies that $M$ is
equivariantly biholomorphic to a Bott tower $M_n(A)$. The result now follows
from Proposition~\ref{finite}.
\end{proof}

It is natural to ask whether the map $\calc_n(X,\gro) \to \BM(X,\gro)$ is a
bijection. For this, suppose $T^n$ and $\tilde T^n$ have the same image. Then
there is a diffeomorphism $f$ of $(X,\gro)$ intertwining $T^n$ and $\tilde
T^n$. Hence $(X,\gro,T^n)$ and $(X,f^*\gro,T^n)$ are toric symplectic manifolds
which are symplectomorphic (by $f$).

\subsection{Symplectic products of spheres}

We are interested in finding the compatible Bott structures to a given
symplectic manifold for certain cases.  We need to fix the diffeomorphism type
and then the symplectic form.

The product of $2$-spheres $(S^2)^n=(\bbc\bbp^1)^n$ admits symplectic forms
$\gro_\bfk=\sum_{i=1}^nk_i\gro_i$ where $k_i\in\bbr^+$ and $\gro_i$ is the
standard volume form on the $i$th factor of $(S^2)^n$. In this case we know from
Theorem \ref{mapathm} that this diffeomorphism type is determined entirely by
its integral graded cohomology ring.

\begin{lemma}\label{3inj} Let $\gro_\bfk=\sum_{i=1}^nk_i\gro_i$
and $\gro_{\bfk'}=\sum_{i=1}^nk_i'\gro_i$ be symplectic forms on $(S^2)^n$
with $k_i,k_i'\in \bbr^+$, and suppose $((S^2)^n,\gro_\bfk)$ and
$((S^2)^n,\gro_{\bfk'})$ are symplectomorphic. Then they are equivalent as
toric symplectic manifolds.
\end{lemma}
\begin{proof} Let $f\colon (S^2)^n\to (S^2)^n$ be a diffeomorphism such
that $f^*\gro_\bfk=\gro_{\bfk'}$. This induces a linear map $f_*$ on
$H^2((S^2)^n,\bbz)\cong\bbz^n$ via the basis $x_1,\ldots x_n$ of standard area
forms on the $S^2$ factors. Now $f_*$ must preserve the cup product relations
$x_i\cup x_i=0$. This implies that the matrix $B$ of $f_*$ with respect to
$x_1,\ldots x_n$ satisfies the relations $B^j_iB^k_i=0$ for all $i$ and $j<k$.
Since $\det B=\pm 1$, $B$ is a monomial matrix whose nonzero entries are $\pm
1$. Hence $k_i=\pm k_{\sigma(i)}'$ for some permutation $\sigma\in S_n$, and
the signs must be positive since $k_i,k_i'\in \bbr^+$. It follows that
$((S^2)^n,\gro_\bfk)$ and $((S^2)^n,\gro_{\bfk'})$ are equivalent by permuting
the factors of $(S^2)^n$.
\end{proof}

\begin{proposition}\label{symkahclass}
Let the Bott tower $M_n(A)$ be diffeomorphic to $(S^2)^n$ with a split
symplectic form $\gro_\bfk=\sum_{i=1}^nk_i\gro_i$ where $k_i\in\bbr^+$ and $\gro_i$
is the standard volume form on the $i$th factor of $(S^2)^n$. Suppose further
that $A^j_i\leq 0$ for all $j<i$. Then with an appropriate choice of order
$[\gro]$ is a K\"ahler class if and only if
$\sum_{i=1}^nk_i(\grd^j_i-m^j_i)>0$ and $\gra_j^2=0$ for each $j=1,\ldots,n$
where $m^j_i=-\frac{1}{2}A^j_i$ if $j<i$ and $m^j_i=0$ if $j\geq i$.
\end{proposition}

\begin{proof}
Up to order and sign $[\gro_1],\ldots, [\gro_n]$ is the unique basis of
square-zero generators of the cohomology ring $H^*(M_n(A),\bbz)$. Thus, if
$M_n(A)$ is diffeomorphic to $(S^2)^n$, by Lemma \ref{primsq0}, we can
identify $[\gro_i]$ with
\begin{equation}\label{symident}
x_i+\frac{\gra_i}{2}=x_i+\frac{1}{2}\sum_{j=1}^{i-1}A^j_ix_j.
\end{equation}
In particular, the entries $A^j_i$ with $j<i$ must be even. Let us consider
the matrix whose entries satisfy $m^j_i=-\frac{1}{2}A^j_i$ for all $j<i$ and
zero elsewhere. So the square-zero classes are $x_i-\sum_{j=1}^{i-1}m^j_ix_j$
with $m^j_i\geq 0$. Thus, using Equation \eqref{symident} we have
\[
[\gro]=\sum_ik_i[\gro_i]=\sum_ik_i(x_i-\sum_{j=1}^{n}m^j_ix_j)
=\sum_{i=1}^n\sum_{j=1}^{n}k_i(\grd^j_i-m^j_i)x_j.
\]
Then since $x_i$ is the Poincar\'e dual $PD(D_{u_i})$, Proposition
\ref{Kahcone} implies that $[\gro]$ is a K\"ahler class if and only if
$\sum_{i=1}^nk_i(\grd^j_i-m^j_i)>0$ for each $j$.

Furthermore, since $M_n$ is diffeomorphic to $(S^2)^n$ we know that the total
Pontrjagin class is trivial, which by Equation \eqref{totPont} is equivalent
to $\gra_j^2=0$ for all $j$.
\end{proof}

The compatibility problem is complicated for arbitrary $n$, so we restrict
ourselves to $n=2,3$. First we take $n=2$ with diffeomorphism type $S^2\times
S^2$, and take the split symplectic structure as
\[
\gro_{k_1,k_2}=k_1\gro_1+k_2\gro_2
\]
where $\gro_i$ is the pullback of the volume form by the projection map ${\rm
  pr}_i\colon S^2\times S^2\to S^2$ onto the $i$th factor and $k_1,k_2\in
\bbr^+$. Now we know that it is precisely the even Hirzebruch surfaces that
are diffeomorphic to $S^2\times S^2$. Let us take $a\leq 0$ and even, so we
set $a=-2m$ with $m\in\bbz_{\geq 0}$. Since $x_i$ is the Poincar\'e
dual of $[D_{u_i}]$, from \eqref{symident} we can write
\[
[\gro_{k_1,k_2}]=k_1x_1+k_2(x_2-mx_1)=
(k_1-mk_2)x_1+k_2x_2=(k_1-mk_2){\rm PD}(D_{u_1})+ k_2{\rm PD}(D_{u_2}).
\]
Thus, the symplectic manifold $(S^2\times S^2,\gro_{k_1,k_2})$ is of Bott type
with respect to $\calh_a=\calh_{-2m}$ if and only if $m<\frac{k_1}{k_2}$. So
the number of compatible complex structures is
$\bigl\lceil\frac{k_1}{k_2}\bigr\rceil$.  Karshon proves in \cite{Kar03} that
$\bigl\lceil\frac{k_1}{k_2}\bigr\rceil$ is the number of conjugacy classes of
maximal tori in the symplectomorphism group.

In the case that $a$ is positive, we put $a=2m$ and use $PD(D_{u_1})$ and
$PD(D_{v_2})$ as the algebraic generators stipulated above. Using
${\rm PD}(D_{v_2})=x_2+2mx_1$, we have that
\[
[\gro_{k_1,k_2}]=k_1x_1+k_2(x_2+mx_1)
=(k_1-mk_2){\rm PD}(D_{u_1})+k_2{\rm PD}(D_{v_2}).
\]
So for $m$ positive we have $k_2>0,k_1>mk_2$. So the symplectic manifold
$(S^2\times S^2,\gro_{k_1,k_2})$ is of Bott type with respect to $\calh_{2m}$
for any $m\in\bbz$ if and only if $k_2>0,k_1>|m|k_2$. It is known that the
cases $m$ and $-m$ are equivalent as toric complex manifolds, but not as
Bott towers.

Next we consider the case $n=3$. For the symplectic viewpoint we first fix a
diffeomorphism type. From Proposition \ref{c0} we know that $M_3(a,b,c)$ is
diffeomorphic to $(S^2)^3$ if and only if $c(2b-ac)=0$, and $a,b$ and $c$ are
all even. We now assume the latter by replacing $(a,b,c)$ by $(2a,2b,2c)$ and
consider the Bott tower $M_3(2a,2b,2c)$ which is diffeomorphic to $(S^2)^3$
with the symplectic form
\begin{equation}\label{3omega}
\gro_{k_1,k_2,k_3}=k_1\gro_1+k_2\gro_2+k_3\gro_3>0
\end{equation}
where $\gro_i$ is the standard area form on the $i$th factor $S^2$ and
$k_i\in\bbr^+$, and as before the $k_i$ are ordered $0<k_3\leq k_2\leq k_1$.

\begin{theorem}\label{splitsymcase}
The symplectic manifold $((S^2)^3,\gro_{k_1,k_2,k_3})$ is of Bott type with
respect to $M_3(2a,2b,2c)$ if and only if one of the following hold\textup:
\begin{enumerate}
\item $c=0$ with $k_1-|a|k_2-|b|k_3>0,~k_2>0,~k_3>0$.
\item $c\neq 0$ and $b=ac$ with $k_1-|a|(k_2-|c|k_3)>0,~k_2-|c|k_3>0,~k_3>0$.
\end{enumerate}
\end{theorem}

\begin{proof}
First applied to $M_3(2a,2b,2c)$ the $p_1(M)=0$ constraint now takes the form
$c(b-ac)=0$. So we have the two cases enunciated in the theorem which reduce
to the Bott towers $M_3(2a,2b,0)$ and $M_3(2a,2ac,2c)$, respectively.  Now
using Equations \eqref{symident} and \eqref{xyeq} we can rewrite the
symplectic class \eqref{3omega} in terms of the 4 sets of preferred bases in
$\grS_3$ giving
\begin{gather}
\label{grobases1}[\gro_{k_1,k_2,k_3}]=(k_1+ak_2+bk_3)x_1+(k_2+ck_3)x_2+k_3x_3, \\
[\gro_{k_1,k_2,k_3}]=(k_1+ak_2-bk_3)x_1+(k_2-ck_3)x_2+k_3y_3, \\
[\gro_{k_1,k_2,k_3}]=(k_1-ak_2+(b-2ac)k_3)x_1+(k_2+ck_3)y_2+k_3x_3, \\
\label{grobases4} [\gro_{k_1,k_2,k_3}]=(k_1-ak_2-(b-2ac)k_3)x_1+(k_2-ck_3)y_2+k_3y_3.
\end{gather}
To assure that $[\gro_{k_1,k_2,k_3}]$ is a K\"ahler class the coefficients
must be positive in each basis.  When $c=0$ the positivity of the coefficients
in these equations is equivalent to
$k_1-|a|k_2-|b|k_3>0,~k_2>0,~k_3>0$. However, when $c\neq 0$ and $b=ac$ the
equations reduce to
\begin{gather*}
[\gro_{k_1,k_2,k_3}]=(k_1+ak_2+ack_3)x_1+(k_2+ck_3)x_2+k_3x_3, \\
[\gro_{k_1,k_2,k_3}]=(k_1+ak_2-ack_3)x_1+(k_2-ck_3)x_2+k_3y_3, \\
[\gro_{k_1,k_2,k_3}]=(k_1-ak_2-ack_3)x_1+(k_2+ck_3)y_2+k_3x_3, \\
[\gro_{k_1,k_2,k_3}]=(k_1-ak_2+ack_3)x_1+(k_2-ck_3)y_2+k_3y_3.
\end{gather*}
The positivity of the coefficients in these equations is equivalent to
$k_1-|a|(k_2-|c|k_3)>0,~k_2-|c|k_3>0,~k_3>0.$ This completes the proof.
\end{proof}

\begin{example}
Let us consider a specific case. Here $(a,b,c)$ are one-half the $(a,b,c)$ in
the itemized list above. We take $(k_1,k_2,k_3)=(5,2,1)$ and ask which Bott
manifolds belong to the symplectic manifold
$\bigl((S^2)^3,\gro_{5,2,1})\bigr)$.  Applying Theorem \ref{splitsymcase} we
see that if $c=0$ we have one constraint, namely $5-2|a|-|b|>0$. Then using
Example \ref{st3equiv} we can take both $a$ and $b$ to be non-negative. So we
have the possibilities $a=2,b=0$, $a=1,~b=2,1,0$, and
$a=0,~b=4,3,2,1,0$. These give $M_3(4,0,0)$, $M_3(2,4,0)$, $M_3(2,2,0)$,
$M_3(2,0,0)$, $M_3(0,8,0)$, $M_3(0,6,0)$, $M_3(0,4,0)$, $M_3(0,2,0)$,
$M_3(0,0,0)$.

Now if $c\neq 0$ the constraint $2-|c| >0$ implies $c=\pm 1$ and we up to
equivalence we can take $c=1$. Then Theorem \ref{splitsymcase} gives the
constraint $0<5-|a|(2-|c|)= 5-|a|$. Again using the equivalences of Example
\ref{st3equiv} we have $a=0,1,2,3,4$ which gives the Bott towers $M_3(0,0,2)$,
$M_3(2,2,2)$, $M_3(4,4,2)$, $M_3(6,6,2)$, $M_3(8,8,2)$. So there are 14
distinct equivalence classes of Bott towers that are compatible with the
symplectic manifold $\bigl((S^2)^3,\gro_{5,2,1})\bigr)$, namely
\begin{gather*}
M_3(0,0,0), M_3(0,2,0), M_3(0,4,0), M_3(0,6,0),
M_3(0,8,0), M_3(2,0,0), M_3(4,0,0),\\
M_3(2,4,0), M_3(2,2,0), M_3(0,0,2), M_3(2,2,2),
M_3(4,4,2), M_3(6,6,2), M_3(8,8,2).
\end{gather*}
Notice that there are nine Bott manifolds with cotwist $\leq 1$, six with
twist $\leq 1$, five of which also have cotwist $\leq 1$. Finally, there are
four with twist and cotwist $2$. By Theorem \ref{conjclassthm} these
correspond to distinct conjugacy classes of maximal tori in the
symplectomorphism group.
\end{example}

Using Theorem \ref{splitsymcase} we obtain a formula for the number
$N_B(k_1,k_2,k_3)$ of Bott manifolds $M_3(2a,2b,2c)$ compatible with a given
symplectic structure $\bigl((S^2)^3,\gro_{k_1,k_2,k_3}\bigr)$ as well as a
growth estimate. For the growth estimates we put $k_3=1$ since clearly the
growth slows for larger $k_3$.

\begin{proposition}\label{numberBott}
  The number of Bott manifolds compatible with the symplectic manifold $\bigl((S^2)^3,\gro_{k_1,k_2,k_3}\bigr)$ is
  $N_B(k_1,k_2,k_3)=N_{B,0}(k_1,k_2,k_3)+N_{B,\neq 0}(k_1,k_2,k_3)$, where
\begin{align*}
N_{B,0}(k_1,k_2,k_3)&=\sum_{j=0}^{b_{max}}\left\lceil\frac{k_1-jk_3}{k_2}\right\rceil
&N_{B,\neq 0}(k_1,k_2,k_3),
&= \sum_{j=1}^{c_{max}}\left\lceil\frac{k_1}{k_2-jk_3}\right\rceil.
\end{align*}
Furthermore, when $k_1,k_2,k_3$ are positive integers we have the growth estimates
\[
\frac{k_1(k_1-1)}{2k_2}\leq N_{B,0}(k_1,k_2,1)\leq \frac{k_1(k_1-1)}{2k_2} +\frac{(k_1-1)(k_2+1)}{k_2},
\]
and when $c_{max}>0$ 
\begin{multline*}
k_1(\ln(k_2-1)+\grg +\gre_{k_2-1})\\
\leq N_{B,\neq 0}(k_1,k_2,1)\leq k_1(\ln(k_2-1)+\grg +\gre_{k_2-1})
+ k_2-1 +\ln(k_2-1)+\grg +\gre_{k_2-1}
\end{multline*}
where $\grg\approx .577$ is Euler's constant, $\gre_{1}=1-\grg$ and $\gre_{k}$ goes to $0$ as $1/2k$.
\end{proposition}

\begin{proof}
According to Theorem \ref{splitsymcase} there are two cases to consider, $c=0$ and $c\neq 0$. If $c=0$ we can count the number of nonnegative integers $a$ in item (1) of Theorem \ref{splitsymcase} that contribute to each $b=0,\ldots,b_{max}$ where $b_{max}$ is the largest nonnegative integer $b$ such that $k_1-bk_3>0$. This gives $N_B(k_1,k_2,k_3)$ when $c=0$,
\[
N_{B,0}(k_1,k_2,k_3)=\sum_{j=0}^{b_{max}}\left\lceil\frac{k_1-jk_3}{k_2}\right\rceil
\]
where $\lceil\cdot\rceil$ is the ceiling function. Similarly, in the case $c>0$ we count the number of nonnegative integers $a$ in item (2) of Theorem \ref{splitsymcase} that contribute to each $c=0,\ldots,c_{max}$ where $c_{max}$ is the largest nonnegative integer $c$ such that $k_2-ck_3>0$. This gives the number of compatible Bott manifolds for each $c\neq 0$, namely
\[
N_{B,c}(k_1,k_2,k_3)=\left\lceil\frac{k_1}{k_2-ck_3}\right\rceil
\]
which implies the formula. To prove the first estimate we put $k_3=1$ and note that 
\[
N_{B,0}(k_1,k_2,1)=\sum_{j=0}^{b_{max}}\left\lceil\frac{k_1-j}{k_2}\right\rceil\geq \sum_{j=0}^{k_1-1}\frac{k_1-j}{k_2} =\frac{1}{k_2}\sum_{l=1}^{k_1-1}l=\frac{k_1(k_1-1)}{2k_2}.
\]
On the other hand using the well known relation between floor and ceiling functions we find that $N_{B,0}(k_1,k_2,1)$ equals
\begin{align*}
\sum_{j=0}^{b_{max}}\left\lceil\frac{k_1-j}{k_2}\right\rceil&\leq
\sum_{j=0}^{k_1-1}\left(\left\lfloor\frac{k_1-j+1}{k_2}\right\rfloor+1\right)\\
&\leq \sum_{l=1}^{k_1-1}\frac{l+1}{k_2} +k_1-1= \frac{k_1(k_1-1)}{2k_2} +\frac{(k_1-1)(k_2+1)}{k_2}.
\end{align*}

For the second estimate we have
\[
N_{B,\neq 0}(k_1,k_2,1)=\sum_{j=1}^{k_2-1}\left\lceil\frac{k_1}{k_2-j}\right\rceil
\geq \sum_{j=1}^{k_2-1}\frac{k_1}{k_2-j}= k_1\sum_{m=1}^{k_2-1}\frac{1}{m}
=k_1(\ln(k_2-1)+\grg +\gre_{k_2-1}).
\]
Again using the relation between floor and ceiling functions we find
\begin{align*}
N_{B,\neq 0}(k_1,k_2,1)&=\sum_{j=1}^{k_2-1}\left\lceil\frac{k_1}{j}\right\rceil
=\sum_{j=1}^{k_2-1}\left(\left\lfloor\frac{k_1+1}{j}\right\rfloor +1\right)
\leq (k_2-1)+(k_1+1)\sum_{j=1}^{k_2-1}\frac{1}{j} \\ 
&=(k_2-1)+(k_1+1)(\ln(k_2-1)+\grg +\gre_{k_2-1})
\end{align*}
which implies the estimate. 
\end{proof}

An easy corollary of Theorem \ref{conjclassthm} and Proposition
\ref{numberBott} is
\begin{corollary}\label{contgroup}
The number of conjugacy classes of maximal tori of dimension $n$ in the symplectomorphism group ${\rm Symp}((S^2)^3,\gro_{k_1,k_2,k_3})$ is at least
\[
\sum_{j=0}^{b_{max}}\left\lceil\frac{k_1-jk_3}{k_2}\right\rceil
+ \sum_{j=1}^{c_{max}}\left\lceil\frac{k_1}{k_2-jk_3}\right\rceil.
\]
\end{corollary}

Note that the inequality $k_1+ack_3-ak_2>0$ implies that the number $N_{B,\neq
  0}(k_1,k_2,k_3)$ grows as the product $ac$. In fact, when $c_{max}>0$ the
largest values of $a,b,c$ occur when $k_3=1$ in which case we find the Bott
tower $M_3\bigl(2(k_1-1),2(k_1-1)(k_2-1),2(k_2-1)\bigr)$.

When $c_{max}=k_2-1>0$ the number $N_{B,\neq 0}(k_1,k_2,1)$ grows like
$k_1\ln(k_2-1)$. Notice that the error term $k_2-1 +\ln(k_2-1)+\grg
+\gre_{k_2-1}$ is independent of $k_1$ and grows linearly with $k_2$.

\begin{example} Consider the symplectic
manifold $\bigl((S^2)^3,\gro_{11,6,1}\bigr)$. Here $N_{B,0}(11,6,1)=16$ and
$N_{B,\neq 0}(11,6,1)=27$, giving $N_{B}(11,6,1)=43$. It is also a
straightforward process to delineate the Bott manifolds. We see in the case of
$c=0$ we have the $17$ Bott manifolds $M_3(0,2b,0):b\in\{0,\ldots,10\}$ and
$M_3(2,2b,0):b\in\{0,\ldots, 5\}$; whereas, for $c\neq 0$ we have the $27$ Bott
manifolds
\begin{gather*}
M_3(2a,10a,10):a\in\{0,\ldots,10\},
M_3(2a,8a,8):a\in\{0,\ldots,5\},\\
M_3(2a,6a,6):a\in\{0,1,2,3\},
M_3(2a,4a,4):a\in\{0,1,2\},
M_3(2a,2a,2):a\in\{0,1,2\}.
\end{gather*}
The largest values of $a,b,c$ occur for the Bott manifold $M_3(20,100,10)$. 
\end{example}

\subsection{Further symplectic equivalences}

\begin{example}
Let us consider the Bott towers $M_{N+1}(\bfk)$ of twist $\leq 1$. To begin
with we fix a symplectic (actually K\"ahler) structure by fixing $k_i<0$. In
this case we know from Example \ref{1twist2} that the K\"ahler cone is the
entire first orthant with respect to the basis $x_1,\ldots,x_{N+1}$. Thus,
there is a symplectic (K\"ahler) form $\gro_\bfr$ whose class is
$[\gro_\bfr]=r_1x_1+\cdots +r_{N+1}x_{N+1}$. We denote this symplectic
manifold by $(M_{N+1}(\bfk),\gro_\bfr)$, and for simplicity we fix an order of
the $k_i$ and consider the case that they all are distinct. Then for $N>2$ by
Theorem~\ref{1twistcor} there are at most $2^{N-1}$ distinct Bott manifolds,
including $(M_{N+1}(\bfk),\gro_\bfr)$ itself that are compatible with
$(M_{N+1}(\bfk),\gro_\bfr)$ as a symplectic manifold which are obtained by
reversing the sign of $k_i$ for each $i=1,\ldots,N$ taking into account that
$M_{N+1}(-\bfk)$ is equivalent to $M_{N+1}(\bfk)$. We now choose a
diffeomorphic Bott manifold $M_{N+1}(\bfk')$ by reversing the signs of some of
the $k_i$. As in Example \ref{1twist2} by a choice of order we can assume that
$k_i'=k_i<0$ for $i=1,\ldots,m$ and $k_{i+1}'=-k_{i+1}>0$ for
$i=m+1,\ldots,N$. Then choosing $D_{N+1}=D_{v_{N+1}}$ the K\"ahler cone is
given by $r_i>-k_ir_{N+1}$ for $i=1,\ldots,m$, and $r_i>0$ for
$j=m+1,\ldots,N$. We have arrived at

\begin{proposition}\label{1twistprop}
The Bott manifold $M_{N+1}(\bfk')$ with $k_i'=k_i<0$ for $i=1,\ldots,m$ and
$k_{i+1}'=-k_{i+1}>0$ for $i=m+1,\ldots,N$ is compatible with the symplectic
manifold $(M_{N+1}(\bfk),\gro_\bfr)$ if and only if $r_i>-k_ir_{N+1}$ for
$i=1,\ldots,m$, and $r_i>0$ for $j=m+1,\ldots,N$. Thus, if the set $\{r_i\}$
satisfies these inequalities and if $N>2$ generically we have
$|\BM(M_{N+1}(\bfk),\gro_\bfr)|=2^{N-1}$.
\end{proposition}

One can check that the equivalent Bott tower $M_{N+1}(-\bfk')$ which swaps
$D_{v_{N+1}}$ and $D_{u_{N+1}}$ gives the same data. Proposition
\ref{1twistprop} shows that when the components of $\bfk$ have different signs
generally there are constraints for a Bott manifold to be compatible with a
given symplectic structure. For example, a symplectic structure $\gro$ whose
class is $x_1+\cdots +x_{N+1}$ is only compatible with the Bott towers
$M_{N+1}(\bfk)$ with either $k_i<0$ for all $i$, or $k_i>0$ for all $i$.
\end{example}

\section{The generalized Calabi and admissible constructions}\label{calabiconstruction}

The purpose of this section is to obtain existence results for extremal
K\"ahler metrics on Bott manifolds. We use the generalized Calabi construction
to prove the general result of Theorem \ref{existence} (2), and the admissible
construction to prove existence of extremal K\"ahler metrics for some of our
running examples. We begin with a brief outline of these two constructions.

\subsection{The generalized Calabi construction}
The {\em generalized Calabi construction} was introduced in \cite[Section
  2.5]{ACGT04} and further discussed in \cite[Section 2.3]{ACGT11}.  This
construction can be used to obtain existence results for extremal K\"ahler
metrics and in particular it will be used to prove Theorem \ref{existence}
(2). In order to make this paper self-contained we review the generalized
Calabi construction for the case of ``no blowdowns'' and where the resulting
manifold is a bundle over a compact Riemann surface $\Sigma$ with fiber equal
to a toric K\"ahler manifold.

\begin{definition}\label{calabidata}
{\em Generalized Calabi data} of dimension $\ell+1$ and rank ${\ell}$ consist of the following.
\begin{enumerate}
\item A compact Riemann surface $\Sigma$ with K\"ahler structure
  $(\omega_\Sigma,g_\Sigma)$. For simplicity we assume that $g_\Sigma$ has
  constant scalar curvature $\Scal_\Sigma$ and $\omega_\Sigma/{2\pi}$ is
  primitive. Hence $\Scal_\Sigma = 4(1-\gg)$, where $\gg$ is the genus of
  $\Sigma$. If $\rho_\Sigma$ denotes the Ricci form then we have $\rho_\Sigma
  = 2(1-\gg)\omega_\Sigma$.
\item A compact toric $\ell$-dim K\"ahler manifold $(V,g_V,\omega_V)$ with Delzant polytope $\Delta \subseteq \gt^*$ and momentum map $z\colon V \rightarrow \Delta$.
\item A principal $T^\ell$ bundle, $P \rightarrow \Sigma$, with a principal
  connection of curvature $\omega_\Sigma \otimes p \in \Omega^{1,1}(\Sigma,
  \gt)$, where $T^\ell$ is the $\ell$-torus acting on $V$ and $p\in \gt$.
\item A constant $p_0\in \bbr$ such that the $(1,1)$-form $p_0\, \omega_{\Sigma} + \langle v,\omega_{\Sigma} \otimes p\rangle$ is positive for $v\in \Delta$.
\end{enumerate}
\end{definition}
From this data we may define the manifold
\[
M=P\times_{T^{\ell}}V = \mathring{M}\times_{(T^{\ell})^\bbc}V \rightarrow \Sigma,
\]
where $ \mathring{M} = P \times_{T^{\ell}} z^{-1}(\mathring{\Delta})$ and
$\mathring{\Delta}$ is the interior of $\Delta$. Since the curvature $2$-form
of $P$ has type $(1,1)$, $\mathring{M}$ is a holomorphic principal
$(T^{\ell})^\bbc$ bundle with connection $\theta \in
\Omega^1(\mathring{M},\gt)$ and $M$ is a complex manifold.

On $\mathring{M}$ we define K\"ahler structures of the form
\begin{equation}\label{compatiblemetric}
\begin{aligned}
g &= \left(p_0+\langle p,z \rangle \right) g_{\Sigma} + \langle dz , {\bf G}, dz\rangle + \langle \theta, {\bf H}, \theta \rangle\\
\omega &=  \left(p_0+\langle p,z \rangle \right) \omega_{\Sigma} + \langle dz \wedge \theta\rangle\\
d\theta &= \omega_{\Sigma}  \otimes p,
\end{aligned}
\end{equation}
where ${\bf G} = {\rm Hess} (U) = {\bf H}^{-1}$, $U$ is the symplectic
potential~\cite{Gui94b} of the chosen toric K\"ahler structure $g_V$ on $V$,
and $\langle \cdot,\cdot,\cdot \rangle$ denotes the pointwise contraction
$\gt^* \times S^2\gt \times \gt^* \rightarrow \bbr$ or the dual contraction.

The {\em generalized Calabi construction} arises from seeing
\eqref{compatiblemetric} as a blueprint for the construction of various
K\"ahler metrics on $M$ by choosing various $S^2\gt^*$-valued functions ${\bf
  H}$ on $\mathring{\Delta}$ satisfying that
\begin{itemize}
\item \textup{[integrability]} ${\bf H}={\rm Hess}(U)^{-1}$ for some smooth
  function $U$ on $\mathring{\Delta}$.

\item \textup{[smoothness]} ${\bf H}$ is the restriction to $\mathring{\Delta}$
  of a smooth $S^2{\mathfrak t}^*$-valued function on $\Delta$\textup;

\item \textup{[boundary values]} for any point $z$ on the codimension one face
  $F_i \subseteq \Delta$ with inward normal $u_i$, we have
\begin{equation}\label{eq:toricboundary}
{\bf H}_{z}(u_i, \cdot) =0\qquad {\rm and}\qquad (d{\bf H})_{z}(u_i,u_i) = 2
u_i,
\end{equation}
where the differential $d{\bf H}$ is viewed as a smooth $S^2{\mathfrak
t}^*\otimes {\mathfrak t}$-valued function on $\Delta$\textup;

\item \textup{[positivity]} for any point $z$ in the interior of a face
$F \subseteq \Delta$, ${\bf H}_{z}(\cdot, \cdot)$ is positive definite
when viewed as a smooth function with values in $S^2({\mathfrak
t}/{\mathfrak t}_F)^*$.
\end{itemize}
Then \eqref{compatiblemetric} extends to a K\"ahler metric on $M$.  Metrics
constructed this way are called {\em compatible K\"ahler metrics} and their
K\"ahler classes are called {\em compatible K\"ahler classes}.

\begin{remark} 
The toric bundle $M$ equipped with a compatible K\"ahler metric as above is a
so-called {\em rigid, semisimple} toric bundle over ${\bbc\bbp^1}$
\cite{ACGT11}.  We also use this terminology for any toric bundle
$M=P\times_{T^{\ell}}V $ where $P$ and $V$ are as defined above.
\end{remark}
\begin{remark}\label{toric}
From now on we assume that $\Sigma = \bbc\bbp^1$, so $\Scal_\Sigma = \Scal_{
  \bbc\bbp^1} =4$.  Note that, since $(g_{\bbc\bbp^1},\omega_{\bbc\bbp^1})$ is
the toric Fubini--Study K\"ahler structure, the K\"ahler metrics arising this
way are themselves toric and generalize the notion of Calabi toric metrics
from~\cite{Leg09}.
\end{remark}

It follows from \cite[Theorem 3 \& Remark 6]{ACGT11} that if $V$ admits a
toric extremal K\"ahler metric $g_V$, then $M$ admits compatible extremal
K\"ahler metrics (at least in some K\"ahler classes). We now use this result
inductively to prove Theorem~\ref{existence} (2).

\begin{proof}[Proof of Theorem~\textup{\ref{existence} (2)}]
First we note that for stage $2$ Bott manifolds, the result holds by Calabi's
original construction of extremal metrics on Hirzebruch
surfaces~\cite{Cal82}. Now consider the Bott tower
\[
M_{k+1} \xrightarrow{\pi_{k+1}} M_k \xrightarrow{\pi_k}
\cdots M_2 \xrightarrow{\pi_2} M_1=\bbc\bbp^1
\]
described by the matrix $A_{k+1}$. By construction $M_{k+1}$ is a rigid
semisimple toric bundle over $\bbc\bbp^1$. Now assume by induction that for a
fixed $k$ and for each $k\times k$ matrix $A$ of the form \eqref{Amatrix}
(with $n$ replaced by $k$) the corresponding Bott manifold $M_k$ has a
compatible extremal K\"ahler metric.  We have seen that the fiber of the map
$M_{k+1}\xrightarrow{\pi_{k+1}\circ\cdots\circ\pi_2} \bbc\bbp^1$ is the Bott
manifold $\tilde{M}_k$ obtained by removing the first row and column of the
matrix defining $M_{k+1}$. We can now apply~\cite[Theorem 3 \& Remark
  6]{ACGT11} to conclude that $M_{k+1}$ admits a compatible toric extremal
K\"ahler metric in some K\"ahler class.
\end{proof}

\subsection{The admissible construction}

When $V$ is $\bbc\bbp^1$ in Definition~\ref{calabidata}, we have Calabi's
construction of extremal K\"ahler metrics on Hirzebruch surfaces \cite{Cal82}.
This construction is a special case of the so-called {\em admissible
  construction} \cite{ACGT08}.  Even though the admissible construction as
summarized below (again for ``no blowdowns'') is not universally useful for
Bott manifolds, there are still some interesting explicit examples occurring
in the literature \cite{KoSa86,Hwa94,Gua95} which we exhibit below in the
context of Bott manifolds. These provide some further existence results for
extremal K\"ahler metrics on Bott manifolds with twist $\leq 2$.

Let $S$ be a compact complex manifold admitting a local K\"ahler product
metric, whose components are K\"ahler metrics denoted $(\pm g_{a}, \pm
\omega_{a})$, and indexed by $a \in {\cala} \subseteq \bbz^{+}$. Here $(\pm
g_{a},\pm \omega_{a})$ is the K\"ahler structure. In this notation we allow
for the tensors $g_{a}$ to possibly be negatively definite---a parametrization
given later justifies this convention.  Note that in all our applications,
each $\pm g_a$ will be CSC.  The real dimension of each component is denoted
$2 d_{a}$, while the scalar curvature of $\pm g_{a}$ is given as $\pm 2 d_{a}
s_{a}$.  Next, let $\call$ be a hermitian holomorphic line bundle over $S$,
such that $c_{1}(\call) = \sum_{a \in\cala} [\omega_{a}/2\pi]$, and let
$\BOne$ denote the trivial line bundle over $S$.  Then, following
\cite{ACGT08}, the total space of the projectivization $M=P(\BOne \oplus
\call) \rightarrow S$ is called an {\it admissible manifold}.

Let $M$ be an admissible manifold of complex dimension $d$. An {\it admissible
  K\"ahler metric} on $M$ is a K\"ahler metric constructed as
follows. Consider the circle action on $M$ induced by the standard circle
action on $\call$. It extends to a
holomorphic $\bbc^*$ action. The open and dense set $M_0$ of stable points
with respect to the latter action has the structure of a principal $\bbc^*$-bundle over the stable quotient.  
The hermitian norm on the fibers induces via
a Legendre transform a function $z\colon M_0\rightarrow (-1,1)$ whose
extension to $M$ has critical submanifolds $z^{-1}(1)=P(\BOne \oplus 0)$ and
$z^{-1}(-1)=P(0 \oplus \call)$.  Letting $\theta$ be a connection one form for
the Hermitian metric on $M_0$, with curvature $d\theta =
\sum_{a\in\cala}\omega_a$, an admissible K\"ahler metric and form are given up
to scale by the respective formulas
\begin{equation}\label{g}
g=\sum_{a\in\cala}\frac{1+\x_az}{\x_a}g_a+\frac {dz^2}
{\Theta (z)}+\Theta (z)\theta^2,\quad
\omega = \sum_{a\in\cala}\frac{1+\x_az}{\x_a}\omega_{a} + dz \wedge
\theta,
\end{equation}
valid on $M_0$. Here $\Theta$ is a smooth function with domain containing
$(-1,1)$ and $\x_{a}$, $a \in \cala$ are real numbers of the same sign as
$g_{a}$ and satisfying $0 < |\x_a| < 1$. The complex structure yielding this
K\"ahler structure is given by the pullback of the base complex structure
along with the requirement $Jdz = \Theta \theta$. The function $z$ is a
hamiltonian for the Killing vector field $K= J\grad z$, hence a momentum map
for the circle action, so $M$ decomposes into the free orbits $M_{0} =
z^{-1}((-1,1))$ and the special orbits $z^{-1}(\pm 1)$. Finally, $\theta$
satisfies $\theta(K)=1$.

For an admissible metric~\eqref{g}, the $2$-form
\[
\phi = \sum_{a\in\cala}-\frac{1+\x_az}{\x_a^2}\omega_a + z dz \wedge \theta
\]
is a \emph{hamiltonian $2$-form} in the sense that
\[
\nabla_{X} \phi = \frac{1}{2}(d\tr_\omega\phi \wedge JX - d^c\tr_\omega\phi  
\wedge X)
\]
for any vector field $X$, where $\tr_\omega\phi=\langle \phi,\omega\rangle$ is
the trace with respect to $\omega$. The theory of~\cite{ApCaGa06} implies that
the metrics~\eqref{g} are (up to scale) the general form of K\"ahler metrics
admitting a hamiltonian $2$-form of order $1$.

For such $g$ to define a metric on $M$, $\Theta$ must satisfy
\begin{align}
\label{positivity}
\text{(i) }\Theta(z) > 0, \quad -1 < z <1,\quad
\text{(ii) }\Theta(\pm 1) = 0,\quad
\text{(iii) }\Theta'(\pm 1) = \mp 2.
\end{align}
where (ii) and (iii) are necessary and sufficient for $g$ to extend to $M$,
while (i) ensures positive definiteness.

The K\"ahler class $\Omega_{\x} = [\omega]$ of an admissible metric is also
called {\it admissible} and is uniquely determined by the parameters $\x_{a}:a
\in \cala$, once the data associated with $M$ (i.e., $d_a$, $s_a$, $g_a$ etc.)
is fixed. The $\x_{a}:a \in \cala$, together with the data associated with $M$
is called {\it admissible data}---see \cite[Section 1]{ACGT08} for
further background on this set-up. Note that different choices of $\Theta$
(with all else being fixed) define different compatible complex structures
with respect to the same symplectic form $\omega$. However, as is discussed in
\cite{ACGT08}, up to a fiber preserving $S^1$-equivariant diffeomorphism, we
may consider the complex structure fixed; then functions $\Theta$ satisfying
\eqref{positivity} determine K\"ahler metrics within the same K\"ahler class.

It is useful to define a function $F(z)$ by the formula $\Theta(z)=
F(z)/p_c(z)$, where $p_c(z) = \prod_{a \in \cala} (1 + \x_{a} z)^{d_{a}}$.
With this notation $g$ has scalar curvature
\begin{equation}\label{scalarcurvature}
\Scal_g = \sum_{a\in\cala}\frac{2 d_a s_a\x_a}{1+\x_az} - \frac{F''(z)}{p_c(z)}
\end{equation}
One may now check \cite{ApCaGa06,ACGT08} that with fixed admissible data,
there is precisely one function, the \emph{extremal polynomial} $F(z)$, so
that the right hand side of \eqref{scalarcurvature} is an affine function of
$z$ and the corresponding $\Theta(z)$ satisfies (ii) and (iii) of
\eqref{positivity}.  This means that as long as also (i) of \eqref{positivity}
is satisfied, i.e., the extremal polynomial is positive for $z\in (-1,1)$, we
would have an admissible extremal metric (and CSC if the affine function is
constant).  However, there is in general no guarantee that the extremal
polynomial is positive for $z\in (-1,1)$.  In the special case where
$\Scal_{\pm g_a}$ is non-negative, results of Hwang~\cite{Hwa94} and
Guan~\cite{Gua95} show that positivity is satisfied and thus every admissible
K\"ahler class has an admissible extremal metric determined by its {\it
  extremal polynomial} $F(z)$.

\subsection{Applications of the admissible construction to Bott manifolds}\label{extradmBott}

We begin by considering the Bott towers $M_{N+1}(\bfk)$ of twist $\leq 1$ first
described in Example \ref{twist1-examples}, which are $\bbc\bbp^1$ bundles
over the product complex manifold $(\bbc\bbp^1)^N$, realized as the
projectavization $\bbp(\BOne \oplus \calo(k_1,k_2,\ldots,k_N)) \rightarrow
(\bbc\bbp^1)^N$. Making contact with the admissible construction we note that
$\cala=\{1,2,\ldots,N\}$, i.e., for $a\in\{1,2,\ldots,N\}$, $\pm g_{a}$ is the
Fubini--Study metric on $\bbc\bbp^1$ with positive constant scalar curvature
$\pm 2 s_{a} = \pm 4/k_a$ for some non-zero integer $k_a$.  By the discussion
above, every admissible K\"ahler class has an admissible extremal metric and
if there is at least one pair $i,j\in\{1,2,\ldots,N\}$ such that $k_i\,k_j<0$
then, by \cite[Corollary 1.2 \& Theorem 2]{Hwa94}, some of these classes
admit admissible CSC metrics. Moreover, since in this case every K\"ahler class
is admissible (see \cite[Remark 2]{ACGT08}), we may conclude as follows.

\begin{proposition}\label{1textprop}
Any stage $N+1$ Bott manifold $M_{N+1}(\bfk)$, with matrix of the form
\eqref{genkexample} such that $\displaystyle \prod_{a=1}^N k_a \neq 0$, admits
an admissible extremal K\"ahler metric in every K\"ahler class. Moreover, when
the $k_1,\ldots,k_N$ do not all have the same sign, some of these metrics are
CSC.
\end{proposition}

\begin{remark}\label{CSCargument}
More specifically, by \cite[Proposition 6]{ACGT08}, the K\"ahler classes of admissible CSC K\"ahler metrics alluded to in Proposition \ref{1textprop} are the ones satisfying that $\alpha_0\beta_1-\alpha_1\beta_0 =0$ with here
\begin{equation}\label{betas2}
\begin{split}
\alpha_p  &:= \int_{-1}^{1} 
p_c(t)t^p dt, \\
\beta_p &:= p_c(1)+(-1)^p p_c(-1)
+\smash[t]{\int_{-1}^1 \biggl(\sum_{a} \frac{2 r_a}{k_a(1+r_a t)}\biggr)}
\vphantom{\Big|}p_c(t) t^p dt.
\end{split}
\end{equation}
and $p_c(t) = \prod_{a=1}^{N} (1+r_at)$.
\end{remark}

Now, if we let $x_a$ denote the pullback from the $a^{th}$ factor of
$\bbc\bbp^1 \times \cdots \times \bbc\bbp^1$ of the generator of
$H^2(\bbc\bbp^1,\bbz)$, i.e., $x_a = \frac{[\omega_a]}{2\pi k_a}$, then the
discussion in \cite[Section 1.4]{ACGT08} allows us to write the (admissible)
K\"ahler classes as
\[
\Omega_\x = 2\pi \left[2 x_{N+1} + \sum_{a=1}^N k_a(1+1/\x_a) x_a  \right],
\]
where $x_{N+1}$ is defined to be $PD(D_{N+1}^\infty)$ as usual. 

By the famous result of Koiso and Sakane \cite{KoSa86}, when $N=2m$, 
\[
k_a = \left\{ \begin{array}{ccl}
1&\text{if}& 1\leq a \leq m\\
-1&\text{if} & m+1 \leq a \leq 2m
\end{array} \right.
\]
and $\x_a=k_a/2$, the admissible extremal metric in $\Omega_\x$ is
K\"ahler--Einstein. 

More generally, when $N=2m$,
\[
k_a = \left\{ \begin{array}{ccl}
q&\text{if}& 1\leq a \leq m\\
-q&\text{if} & m+1 \leq a \leq 2m
\end{array} \right.
\]
for some $q\in \bbz^+$, and
\[
\x_a = \left\{ \begin{array}{ccl}
\x&\text{if}& 1\leq a \leq m\\
-\x&\text{if} & m+1 \leq a \leq 2m
\end{array} \right.
\]
for any $\x\in(0,1)$, we have that \eqref{betas2} in Remark \ref{CSCargument} simplifies to
\[
\begin{split}
\alpha_p  &:= \int_{-1}^{1} 
(1-r^2t^2)^mt^p dt, \\
\beta_p &:= (1+(-1)^p)(1-r^2)^m
+\smash[t]{\int_{-1}^1 \biggl( \frac{4 mr}{q}\biggr)}
\vphantom{\Big|}(1-r^2t^2)^{m-1} t^p dt.
\end{split}
\]
and thus $\alpha_1=\beta_1=0$ giving us $\alpha_0\beta_1-\alpha_1\beta_0 =0$.
Thus the admissible extremal metric in $\Omega_\x$ is CSC with positive scalar curvature. 

As a special case of the above, suppose for example that $\x=1/k$ for $k\in \bbz^+\setminus\{1\}$ in the above. Then 
\[
\Omega_\x/(2\pi) =  \left[2 x_{2m+1} +q(k+1) \sum_{a=1}^{m} x_a  +  q(k-1) \sum_{a=m+1}^{2m} x_a\right]
\]
is an integer K\"ahler class and hence defines a line bundle.
If we take the admissible CSC representative of $\Omega_\x$  on the admissible stage $2m+1$ Bott manifold
\[
M_{2m+1}= \bbp(\BOne \oplus \calo(\overbrace{q,\ldots,q}^m,\overbrace{-q,\ldots,-q}^m) \rightarrow 
\overbrace{\bbc\bbp^1\times \cdots \times \bbc\bbp^1}^{2m}
\]
and feed it into another admissible construction, with $M_{2m+1}$ now playing
the role of the base $S$, we get explicit extremal admissible metrics on the
resulting stage $2m+2$ Bott manifold $M_{2m+2} = \bbp(\BOne \oplus
\call_{2m+2}) \rightarrow M_{2m+1}$ in all K\"ahler classes spanned by
\[
\left(1+\frac{1}{\x_{2m+1}}\right) \left[2 x_{2m+1} +q(k+1) \sum_{a=1}^{m} x_a  +  q(k-1) \sum_{a=m+1}^{2m} x_a\right] +  2 x_{2m+2},
\]
where $\x_{2m+1}\in(0,1)$ and $x_{2m+2} = PD(D_{2m+2}^\infty)$. Thus we may conclude
as follows.

\begin{proposition}\label{2textremal}
Assume that $q \in \bbz^+$ and $k\in \bbz^+\setminus\{1\}$ and consider the
height $2m+2$ Bott tower $M_n(A)$ with
\[
A=\begin{pmatrix}
\BOne_m&0_m&\cdots & 0 \\
0_m & \BOne_m & \vdots & \vdots \\
q{\bf 1}_m& -q{\bf 1}_m & 1 & 0 \\
q(k+1){\bf 1}_m& q(k-1){\bf 1}_m &2&1\\
\end{pmatrix}
\]
Then there is a $2$-dimensional subcone of the K\"ahler cone of $M_n(A)$ whose
elements admit admissible extremal metrics.  Furthermore, the stage $2m+1$
Bott manifold obtained by removing the last column and the last row of $A$
admits admissible extremal K\"ahler metrics in every K\"ahler class of the
entire K\"ahler cone and admissible CSC K\"ahler metrics in some K\"ahler
classes. If $q=1$ this stage $2m+1$ Bott manifold also admits a
K\"ahler--Einstein metric.

\textup(Here admissibility is with respect the $\bbc\bbp^1$-fibration of
$M_n(A)$ over the previous stage $M_{n-1}(A)$ in the Bott tower.\textup)
\end{proposition} 

\begin{remark}
A similar result to Proposition \ref{1textprop} holds in the degenerate case when some of the components of $\bfk$ vanish. These Bott towers as in \eqref{1tdeg} have the form 
\begin{equation}\label{1tdeg2}
M_{N+1}(\bfk)\cong M_{m}({\tilde{\bfk}})\times
\overbrace{\bbc\bbp^1\times\cdots\times\bbc\bbp^1}^{N+1-m}
\end{equation}
where none of the components $k_1,\ldots,k_{m-1}$ of $\tilde{\bfk}$ vanish. 
\end{remark}

\subsection{Relation with c-projective geometry}\label{cprojconnection}

We now present a curious observation in the case of the stage $3$ Bott
manifold of the Bott tower with matrix
\[
A=\begin{pmatrix} 1&0&0\\ 0&1&0\\ 1&-1&1 \end{pmatrix}.
\]
That is, we take a closer look at admissible CSC K\"ahler metrics on
$M_3(0,1,-1)= \bbp(\BOne \oplus \calo(-1,1)) \rightarrow \bbc\bbp^1\times
\bbc\bbp^1 $.

It turns out that on this special case of a stage $3$ Bott manifold, there
exist infinitely many pairs of c-projectively equivalent CSC K\"ahler metrics
which are not affinely equivalent. For background on c-projective geometry we
refer the reader to \cite{CEMN16}.  Here we focus on the c-projective
equivalent admissible metrics.

Let $(g,\omega, J)$ be a K\"ahler metric as in \eqref{g} and for $ \alpha,
\beta \in\bbr$ consider the hamiltonian $2$-form
\[
\phi_{\alpha,\beta} := (\alpha \phi + \beta \omega) =  \sum_{a\in\cala}\frac{(\beta \x_a - \alpha)(1+\x_az)}{\x_a^2}\omega_a+(\alpha z + \beta) dz \wedge
\theta.
\]
According to the general theory \cite{CEMN16}, we can form a symmetric and
$J$-invariant tensor $g_{\alpha,\beta}^{-1}\colon T^*M \rightarrow TM$ by
\[
g_{\alpha,\beta}^{-1} = (\det C)^{1/2} C \circ g^{-1},
\]
where $C\colon TM \rightarrow TM$ is the $g$-selfadjoint tensor defined by
\[
\phi_{\alpha,\beta}(\cdot, J \cdot)= g (C \cdot, \cdot).
\]
If $g_{\alpha,\beta}^{-1}$ is invertible, then $g_{\alpha,\beta}$ defines a
K\"ahler metric $(g_{\alpha,\beta}, J, \hat{\omega})$ such that $g$ and
$\hat{g}$ are c-projectively equivalent. In that case, $(g_{\alpha,\beta}, J,
\hat{\omega})$ also admits a hamiltonian $2$-form (of same order) and so
should be admissible up to scale and in appropriate coordinates.

Now we see that
\begin{equation*}
(\alpha \phi + \beta \omega)(\cdot, J \cdot) =
\sum_{a\in\cala}\frac{(\beta \x_a - \alpha)(1+\x_az)}{\x_a^2}g_a
+ (\alpha z + \beta)\Bigl(\frac {dz^2}{\Theta (z)}+\Theta (z)\theta^2\Bigr).
\end{equation*}
If we (locally) extend $dz, \theta$ to a basis of $1$-forms by pullbacks from
the base, then we have corresponding dual vector fields $Z, K$ and may write
\[
g^{-1} = \sum_{a\in\cala}\frac{\x_a}{1+\x_az}g_a^{-1} +\Theta (z) Z^2 + \frac {1}
{\Theta (z)} K^2.
\]
Now $C\,g^{-1}$  must be equal to
\[
(\alpha \phi + \beta \omega)(\cdot, J \cdot)^\sharp=
\sum_{a\in\cala}\frac{(\beta \x_a - \alpha)}{(1+\x_az)}g_a^{-1}+(\alpha z + \beta)(\Theta (z) Z^2 + \frac {1}
{\Theta (z)} K^2).
\]
Therefore
\[
\det(C)^{\frac{1}{2}} = 
\frac{(\det(\alpha \phi + \beta \omega)(\cdot, J \cdot)^\sharp)^{\frac{1}{2}}}
{\det(g^{-1})^{\frac{1}{2}}}
= \frac{(\alpha z + \beta)\prod_{a\in\cala}(\beta \x_a - \alpha)^{d_a}}
{\prod_{a\in\cala}\x_a^{d_a}}
\]
and so (since $g_{\alpha,\beta}^{-1}=\det(C)^{\frac{1}{2}} C\,g^{-1}$) we have
(for appropriate choices of $\alpha$ and $\beta$) that the metric
$g_{\alpha,\beta}$ given up to scale by
\begin{equation}\label{admissiblenewg}
\frac{1}{(\alpha z + \beta)}\left(
 \sum_{a\in\cala}\frac{(1+\x_az)}{(\beta \x_a - \alpha)}g_a
 +(\alpha z + \beta)^{-1}\biggl(\frac {dz^2}
{\Theta (z)}+\Theta (z)\theta^2\biggr)\right)
\end{equation}

Now consider the change of coordinates
\[
\tilde{z} = \frac{\alpha + \beta z}{\alpha z + \beta},\quad
z = \frac{\alpha - \beta \tilde{z}}{\alpha \tilde{z} - \beta},
\quad
dz =  \frac{\beta^2-\alpha^2 }{(\alpha \tilde{z} - \beta)^2}\, d\tilde{z}.
\]
Then $\tilde{z} \in [-1,1]$ if and only if $z \in [-1,1]$ and moreover
\[
1+\x_a z = \frac{(\alpha - \beta \x_a)\tilde{z} + \alpha \x_a - \beta}{\alpha \tilde{z} - \beta}.
\]
We now rewrite \eqref{admissiblenewg} according to the change of coordinates and get
\[
 \frac{1}{( \beta^2-\alpha^2 )}\left( \sum_{a\in\cala} 
\frac{((\beta \x_a-\alpha)\tilde{z} + \beta- \alpha \x_a)}{(\beta \x_a - \alpha)}g_a
+ \frac {(\beta^2-\alpha^2)d\tilde{z}^2}
{(\alpha \tilde{z} - \beta)^2\Theta (z)}+\frac{(\alpha \tilde{z} - \beta)^2\Theta (z)}{( \beta^2-\alpha^2 )}\theta^2 \right).
\]
Now, with
\begin{equation*}
\Theta_{\alpha,\beta}(\tilde{z}) = \frac{(\alpha \tilde{z} - \beta)^2\Theta (z)}{( \beta^2-\alpha^2 )} = \frac{d\tilde{z}}{dz} \Theta(z)
\quad\text{and} \quad 
(\x_a)_{\alpha,\beta} = \frac{\beta \x_a - \alpha}{\beta - \alpha \x_a},
\end{equation*}
we recognize this as an admissible metric.

It is easy to check that the end point conditions~\eqref{positivity} are
satisfied for $\Theta_{\alpha,\beta}(\tilde{z})$ at $\tilde{z} = \pm 1$ and
the complex structure remains compatible with the admissible set-up:
\[
J d\tilde{z} = \frac{d\tilde{z}}{dz}J dz =  \frac{d\tilde{z}}{dz}\Theta(z) \theta = \Theta_{\alpha,\beta}(\tilde{z})\theta.
\]
Instead of looking at 
\[
\Theta_{\alpha,\beta}(\tilde{z}) = \frac{(\alpha \tilde{z} - \beta)^2\Theta (z)}{( \beta^2-\alpha^2 )}
\]
we consider the corresponding change in $F(z)$ as defined in Section
\ref{extradmBott}. From
\[
F_{\alpha,\beta}(\tilde z) =\Theta_{\alpha,\beta}(\tilde{z})\prod_{a \in \cala} (1 + (\x_a)_{\alpha,\beta} \tilde{z})^{d_{a}}
\quad\text{and}\quad
\Theta(z)=\frac{F(z)}{\prod_{a \in \cala} (1 + \x_{a} z)^{d_{a}}}
\]
it easily follows that
\begin{equation}\label{admissiblechangeofF}
F_{\alpha,\beta}(\tilde z)  = \frac{(\beta - \alpha \tilde{z})^{d+1} F( \frac{\alpha - \beta \tilde{z}}{\alpha \tilde{z} - \beta})}{(\beta^2-\alpha^2)\left(\prod_{a \in \cala} (\beta -\alpha \x_{a})^{d_{a}}\right)}.
\end{equation}

\begin{remark} Note that $\langle \phi,  \omega \rangle$ is not a constant,
so as long as $\alpha \neq 0$; then $g$ and $g_{\alpha,\beta}$ are not
affinely equivalent.
\end{remark}

We now consider admissible metrics on $P({\mathcal O} \oplus {\mathcal
  O}(1,-1)) \to \bbc\bbp^1 \times \bbc\bbp^1$ by choosing $\cala=\{ 1,2\}$
with $g_1$ and $-g_2$ being K\"ahler--Einstein metrics, $s_1 = 2$, $s_2 = -2$,
and $d_i = 1$.  As discussed in Section \ref{extradmBott}, each (admissible)
K\"ahler class, determined by $0<\x_1<1$ and $-1<\x_2<0$, admits an admissible
extremal K\"ahler metric with extremal polynomial $F(z)$.  Setting $q=1$ in
\cite[Theorem 9]{ACGT08}, $F(z)$ satisfies the CSC conditions if and only if
$\x_2=-\x_1$ or $\x_2 = -1 + \x_1$. In each of these families the
K\"ahler--Einstein condition on $F(z)$ is satisfied precisely when
$\x_2=-\x_1=-1/2$ which corresponds to the Koiso--Sakane K\"ahler--Einstein
metric \cite{KoSa86}.

\begin{remark}\label{KSsecondfamilyremark}
Note that we  already discussed the $\x_2=-\x_1$ solutions in Section \ref{extradmBott} and saw that this family of solutions have higher dimensional analogues on
\[
M_{2m+1}= \bbp(\BOne \oplus \calo(\overbrace{1,\ldots,1}^m,\overbrace{-1,\ldots,-1}^m) \rightarrow 
\overbrace{\bbc\bbp^1\times \cdots \times \bbc\bbp^1}^{2m}.
\]
The second family of solutions, namely $\x_2 = -1 + \x_1$, also seems to have
an analogue in higher dimensions. In the notation of Section \ref{extradmBott}
assume that $\Omega_r$ is the (admissible) K\"ahler class on $M_{2m+1}$ given
by
\[
\x_a = \left\{ \begin{array}{ccl}
\x_+&\text{if}& 1\leq a \leq m\\
\x_-&\text{if} & m+1 \leq a \leq 2m,
\end{array} \right.
\]
where $0<\x_+<1$ and $-1<\x_-<0$. Then, using~\cite{ACGT08} or Remark \ref{CSCargument}, we observe that $\Omega_r$ has an admissible CSC representative if and only if
$\alpha_0 \beta_1 -\alpha_1 \beta_0 =0$, where
\begin{align*}
\alpha_0 & = \int_{-1}^{1} (1+\x_+ z)^m(1+\x_- z)^m\,dz\\
\alpha_1 & = \int_{-1}^{1} z(1+\x_+ z)^m(1+\x_- z)^m\,dz\\
\beta_0 & = (1+\x_+)^m(1+\x_- )^m +  (1-\x_+)^m(1- \x_- )^m\\
&\qquad+ 2m(\x_+-\x_-) \int_{-1}^{1} (1+\x_+ z)^{m-1}(1+\x_- z)^{m-1}\,dz\\
\beta_1 & = (1+\x_+)^m(1+\x_- )^m -  (1-\x_+)^m(1- \x_- )^m\\
&\qquad+ 2m(\x_+-\x_-) \int_{-1}^{1} z(1+\x_+ z)^{m-1}(1+\x_- z)^{m-1}\,dz.
\end{align*}
For $m=1$ this recovers the two families of solutions above. In general, for any positive integer $m$, it is easy to confirm the first family of solutions $\x_- = -\x_+$.
Computer aided calculations for $m=2, m=3$, and $m=4$ show a second family of solutions for these cases as well and we conjecture that this happens for any positive integer 
$m$. We have not been able to confirm this directly. Figure~\ref{fig1} shows the behavior of the second family for $m=1,2,3$, and $m=4$.

\begin{figure}
\begin{center}
\includegraphics[width=0.5\textwidth]{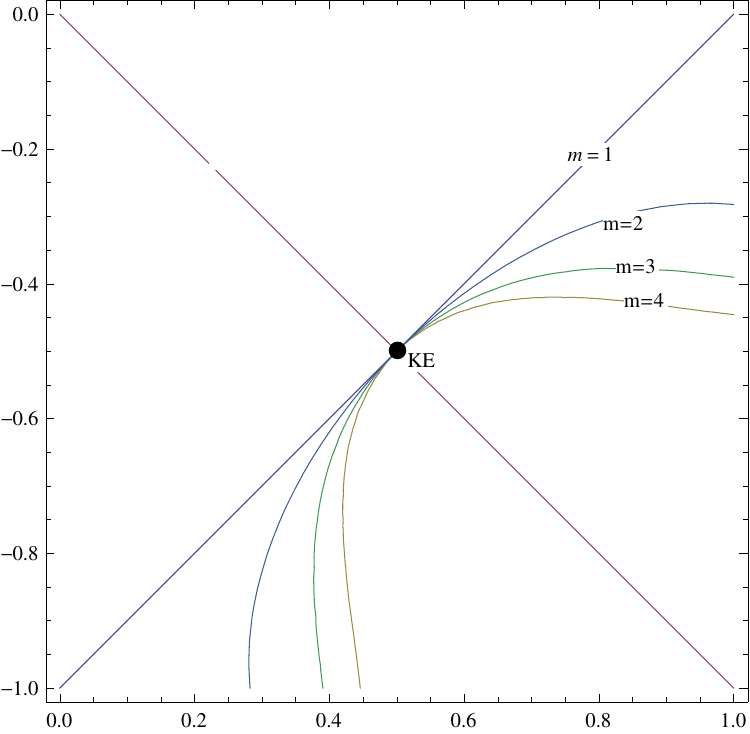}
\end{center}
\caption{Behaviour of the second family for $1\leq m\leq 4$.}\label{fig1}
\end{figure}
\end{remark}

Returning to the $m=1$ case, let us look at the second family of CSC K\"ahler
metrics described above. We parametrize it by $0<\x<1$ by setting $\x_1=\x$
and $\x_2= -1+\x$. Let us denote this family by ${\mathcal KS}$. For a given
choice of $0<\x< 1$, the CSC K\"ahler metric corresponds to the extremal
polynomial
\begin{equation}\label{kspol}
\begin{split}
F_\x(z)&  =  (1-z^2)((1+\x_1 z)(1+\x_2 z) + \tfrac12 \x_1 \x_2 (1-z^2)) \\
&= \tfrac12 (1-z^2)(2-\x+\x^2+(4\x-2)z+\x(\x-1)z^2).
\end{split}
\end{equation}

Now, we start with a CSC admissible metric $g \in {\mathcal KS}$ as above,
determined by a parameter $0<\x<1$. If we choose
\[
\alpha = \frac{2\x-1}{1-\x+\x^2}
\quad\text{and}\quad
\beta = 1
\]
then it is not hard to check that the c-projectively equivalent metric
$g_{\alpha,\beta}$ is defined.  Further, since $\x_1=\x$ and $\x_2=-1+\x$, we
calculate that, with this choice of $\alpha$ and $\beta$,
\[
(\x_1)_{\alpha,\beta} =  \frac{\beta \x - \alpha}{\beta - \alpha \x} = 1-\x
\quad\text{and}\quad
(\x_2)_{\alpha,\beta} =  \frac{\beta (-1+\x) - \alpha}{\beta-\alpha (-1+\x)}= -\x.
\]
Thus the new metric $g_{\alpha,\beta}$ corresponds (up to scale) to a K\"ahler
class of one of the metrics belonging to ${\mathcal KS}$ determined by the
parameter $\x_{\alpha,\beta} = 1-\x$.

Figure~\ref{fig2} illustrates the K\"ahler cone (up to scaling) as
parametrized by $0<\x_1<1$ and $-1<\x_2<0$. The two intersecting line segments
correspond to the classes with CSC K\"ahler metrics and their intersection
point is the K\"ahler--Einstein class. The line segment going from $(-1,0)$ to
$(1,0)$ (not including the endpoints of course) correspond to the K\"ahler
classes of the metrics in ${\mathcal KS}$. The trajectories represent how the
K\"ahler classes change as we move in a c-projective class. The bold dots
illustrate the above observation about c-projectively equivalent CSC metrics.

\begin{figure}
\begin{center}
\includegraphics[width=0.5\textwidth]{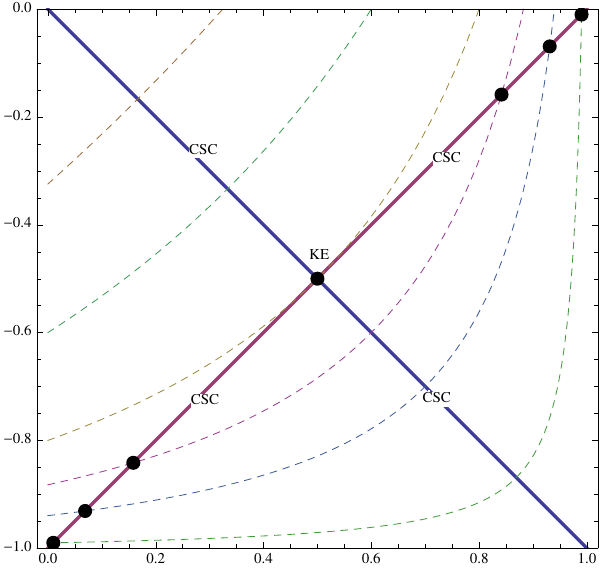}
\end{center}
\caption{CSC K\"ahler classes.}\label{fig2}
\end{figure}

Now adapting \eqref{admissiblechangeofF} to our case we have
\begin{equation*}
\begin{split}
F_{\alpha,\beta}(\tilde z) &= 
\frac{(\beta - \alpha \tilde{z})^{4} F_\x( \frac{\alpha - \beta \tilde{z}}{\alpha \tilde{z} - \beta})}{(\beta^2-\alpha^2)\left( (\beta-\alpha \x_{1})(\beta-\alpha \x_{2})\right)} 
=  \frac{(\beta - \alpha \tilde{z})^{4} F_\x( \frac{\alpha - \beta \tilde{z}}{\alpha \tilde{z} - \beta})}{(\beta^2-\alpha^2)\left( (\beta-\alpha \x)(\beta-\alpha (-1+\x))\right)}\\
&= \tfrac12 (1-\tilde{z}^2)(2-\x+\x^2+(2-4\x){\tilde{z}+\x(\x-1)}{\tilde z}^2)
 = F_{1-\x}(\tilde{z}).
\end{split}
\end{equation*}
Therefore $g_{\alpha,\beta} \in {\mathcal KS}$ as well. Since $\alpha \neq 0$,
we then have an example of a c-projective class of K\"ahler metrics with more
affinely inequivalent CSC metrics.

\begin{proposition}
On the stage $3$ Bott manifold $M_3(0,1,-1)$ of twist $1$ there exist an
infinite number of pairs of c-projectively equivalent CSC K\"ahler metrics
which are not affinely equivalent.
\end{proposition}

\begin{remark}
At the moment it seems that the example above is rather special. From Remark
\ref{KSsecondfamilyremark} it is natural to speculate whether the same
phenomenon occurs for the presumed second family of CSC admissible metrics on
\[
M_{2m+1}=
\bbp(\BOne \oplus \calo(\overbrace{1,\ldots,1}^m,\overbrace{-1,\ldots,-1}^m)
\rightarrow\overbrace{\bbc\bbp^1\times \cdots \times \bbc\bbp^1}^{2m}
\]
for $m>1$.  Due to the non-explicit nature of the second family for $m>1$ it seems rather intractable to check in general, but a computer aided calculation suggests that the phenomenon does not occur for $m=2, m=3,$ and $m=4$.
\end{remark}

\subsection{Extremal almost K\"ahler metrics on certain stage $3$ Bott manifolds}

Extremal metrics on almost K\"ahler manifolds have been discussed
in~\cite{ApDr03,Don02,ACGT11} and studied in detail by Lejmi \cite{Lej10}.  We
now consider constructions of smooth almost K\"ahler extremal metrics on
certain stage $3$ Bott manifolds. Essentially we use the generalized Calabi
construction set-up from the beginning of Section \ref{calabiconstruction},
but we do not attempt to satisfy the integrability condition on $\mathbf H$. We
refer to \cite[Appendix A \& Section 4.3]{ACGT11} for definitions and
justification of the adaption of the generalized Calabi construction.

Consider the case when the fiber is the product $\bbc\bbp^1 \times
\bbc\bbp^1$.  In this case $V$ of Definition \ref{calabidata} is simply
$\bbc\bbp^1 \times \bbc\bbp^1$ and with the identification of $\gt$ and
$\gt^*$ with $\bbr^2$, the polytope is a rectangle. We use the notation $z =
(z_1,z_2)$ and for simplicity, we assume that the polytope is the square
$[0,1] \times [0,1]$ with normals
\[
u_1 = \langle 1,0 \rangle, \quad u_2 = \langle -1,0 \rangle, \quad u_3 = \langle 0, -1 \rangle, \quad u_4 = \langle 0, 1 \rangle
\]
on $F_1 = \{ (0,z_2)\,|\,0 \leq z_2 \leq 1\}$, $F_2 = \{ (1,z_2)\,|\,0 \leq
z_2 \leq 1\}$, $F_3 = \{ (z_1,1)\,|\,0 \leq z_1 \leq 1\}$, and $F_4 = \{
(z_1,0)\,|\,0 \leq z_1 \leq 1\}$ respectively.  An $S^2\gt^*$-valued function
$H$ on $\mathring{\Delta}$ may then be viewed as a matrix ${\bf H} =
(H_{ij})$, where each entry $H_{ij}$ is a smooth function of
$(z_1,z_2)$. Likewise $p \in \gt$ is just $p=(p_1,p_2)$, where here $p_1$ and
$p_2$ must be integers.  Note that $p_1$ and $p_2$ may be identified with $a$
and $b$ in the notation from Section \ref{stage3}.  Of course $c$ from Section
\ref{stage3} vanishes here.  Now $p_0$ from Definition \ref{calabidata} must
be a real constant such that $Q= Q(z_1,z_2) = p_0 + p_1 z_1 + p_2 z_2$ is
positive for $(z_1,z_2) \in [0,1] \times [0,1]$. In particular, we need to
assume that $p_0 \in \bbr^+$, $p_0+p_1>0$, $p_0+p_2>0$, and $p_0+p_1+p_2>0$.

The boundary conditions \eqref{eq:toricboundary} on ${\bf H}=
\left[ \begin{matrix} H_{11} & H_{12} \\ H_{12} & H_{22} \end{matrix}
  \right]$ are
\begin{align}\label{endpointHsquare1}
H_{11}(0,z_2) &= H_{12}(0,z_2)=0 &&\text{and}& H_{11,1}(0,z_2) &= 2 &&\text{for}\quad 0 \leq z_2 \leq 1,\\
\label{endpointHsquare2}
H_{11}(1,z_2) &= H_{12}(1,z_2)=0,&&\text{and}& H_{11,1}(1,z_2) &= -2 &&\text{for}\quad 0 \leq z_2 \leq 1,\\
\label{endpointHsquare3}
H_{22}(z_1,1) &= H_{12}(z_1,1)=0,&&\text{and}& H_{22,2}(z_1,1) &= -2 &&\text{for}\quad 0 \leq z_1 \leq 1,\\
\label{endpointHsquare4}
H_{22}(z_1,0) &= H_{12}(z_1,0)=0,&&\text{and}& H_{22,2}(z_1,0) &= 2 &&\text{for}\quad 0 \leq z_1 \leq 1,
\end{align}
where $H_{ij,k} := \frac{\partial}{\partial z_k} H_{ij}$.  To get a genuine
positive definite metric out of this we also need that ${\bf H}$ is positive
definite.

To complete the picture, observe that the integrability condition on $J$ is as
follows: if ${\bf H}^{-1}= (H^{ij})$, then we must have that
\begin{equation}\label{integralitysquare}
\frac{\partial H^{11}}{\partial z_2} = \frac{\partial H^{12}}{\partial z_1}\quad\text{and}\quad \frac{\partial H^{22}}{\partial z_1} = \frac{\partial H^{12}}{\partial z_2}.
\end{equation}

The extremal condition is 
\[
4 - \frac{\partial^2}{\partial z_i \partial z_j} \left( QH_{ij} \right) = (A_1 z_1+A_2 z_2 + A_3) Q
\]
for some constants $A_1,A_2, A_3$ with CSC corresponding to $A_1=A_2 = 0$.

Now assume that for any $i\leq j \in \{1,2\}$, $QH_{ij}$ is a polynomial
$P_{ij}$ of degree $4$ in $z_1, z_2$.  This gives the following equation for
extremality of the corresponding compatible metric.
\begin{equation}\label{polyextrsquare}
4-P_{11,11} - 2P_{12,12} - P_{22,22} = (A_1 z_1+A_2 z_2 + A_3) (p_0+p_1 z_1 + p_2 z_2),
\end{equation}
where $P_{ij,kl} := \frac{\partial^2P_{ij}}{\partial z_k \partial z_l}$.
Conditions \eqref{endpointHsquare1}--\eqref{endpointHsquare4} mean that
for $0 \leq z_1\leq 1$ and $0\leq z_2 \leq 1$,
\begin{align*}
P_{11}(0,z_2) &= P_{12}(0,z_2)=0, & P_{11,1}(0,z_2) &= 2 Q(0,z_2) = 2(p_0+p_2 z_2)\\
P_{11}(1,z_2) &= P_{12}(1,z_2)=0, & P_{11,1}(1,z_2) &= -2 Q(1,z_2) = -2(p_0+p_1+p_2 z_2)\\
P_{22}(z_1,1) &= P_{12}(z_1,1)=0,& P_{22,2}(z_1,1) &= -2 Q(z_1,1) = -2(p_0+p_1 z_1 + p_2) \\
P_{22}(z_1,0) &= P_{12}(z_1,0)=0,& P_{22,2}(z_1,0) &= 2 Q(z_1,0) = 2(p_0+p_1 z_1).
\end{align*}
These equations imply that
\begin{equation}\label{whatPisSquare}
\begin{split}
P_{11} &= z_1(1-z_1) (2Q + a_{11} z_1(1-z_1))\\
P_{12} & = a_{12}z_1(1-z_1)z_2(1-z_2)\\
P_{22} & =  z_2(1-z_2)(2 Q + a_{22} z_2(1-z_2)),
\end{split}
\end{equation}
for some constants $a_{ij} \in \bbr$. 

Now we substitute these functions into the left hand side of \eqref{polyextrsquare} to get
\begin{multline*}
 (4 - 4p_1-4p_2-2a_{11} - 2a_{22} - 2a_{12} + 8 p_0) + (12a_{11} + 4 a_{12} + 16 p_1) z_1\\
+ (12a_{22} + 4 a_{12} + 16 p_2) z_2 - 12 a_{11} z_1^2 - 8a_{12} z_1z_2 - 12 a_{22} z_2^2.
\end{multline*}
Expanding the right hand side of \eqref{polyextrsquare}  yields
\[
 p_0A_3 +  (p_0A_1 + p_1 A_3) z_1 + (p_0A_2 + p_2 A_3) z_2+ p_1A_1z_1^2 + (p_1A_2+p_2A_1)z_1z_2 + p_2A_2 z_2^2.
\]
It is then clear that - assuming $p_0,p_1, p_2$ fixed - we need to solve the linear system of $6$ equations with $6$ unknowns:
\[
\left[ \begin{matrix}
-2 & -2& -2& 0 & 0 & -p_0\\
12 &4&0&-p_0&0 &-p_1 \\
0& 4& 12 & 0 & -p_0 & - p_2\\
-12&0&0&-p_1&0&0\\
0&-8&0&-p_2&-p_1&0\\
0&0&-12&0&-p_2&0
\end{matrix}\right] \left[ \begin{matrix} a_{11}\\ a_{12}\\a_{22}\\A_1\\A_2\\A_3 \end{matrix} \right] = \left[ \begin{matrix}
-4+4p_1+4p_2-8p_0\\
-16 p_1\\
-16p_2\\
0\\
0\\
0 \end{matrix} \right].
\] 
This has a unique solution for $(a_{11},a_{12},a_{22},A_1,A_2,A_3)$ since $Q$
is positive over $[0,1]\times[0,1]$, so the determinant $6 p_0^2+6 p_0
p_1+p_1^2+6 p_0 p_2+3 p_1 p_2+p_2^2$ is positive.

We find (not unexpectedly) that $A_1=A_2=0$ is only possible if $p_1=p_2=0$
which yields the K\"ahler product $\bbc\bbp^1\times
\bbc\bbp^1\times\bbc\bbp^1$.  In other words, none of the solutions yield
non-trivial CSC almost K\"ahler metrics.

Since $Q$ is assumed positive over $[0,1]\times [0,1]$, checking positive definiteness of ${\bf H}$ over $(0,1)\times(0,1)$ is equivalent to ensuring that ${\bf P}$ is positive definite over $(0,1)\times(0,1)$. The latter condition amounts to checking that 
\[
\det{ \bf P} = P_{11} P_{22} - P_{12}^2 >0\quad \text{and} \quad  P_{11} >0
\]
over $(0,1)\times(0,1)$. It is easy to see that for fixed $p_1$ and $p_2$ this
is satisfied for sufficiently large and positive $p_0$ values. This is seen by
the fact that the limit, $p_0\rightarrow +\infty$ corresponds to having the
product of Fubini--Study metrics on the fiber $\bbc\bbp^1 \times
\bbc\bbp^1$. This is not unexpected in light of \cite[Theorem 3]{ACGT11}.

To be more specific, observe now that $P_{11} P_{22} - P_{12}^2$ is
\begin{equation*}
z_1(1-z_1)z_2(1-z_2)\bigl( \,(2Q + a_{11} z_1(1-z_1))\,
(2 Q  + a_{22} z_2(1-z_2)) - a_{12}^2z_1(1-z_1)z_2(1-z_2)\,\bigr)
\end{equation*}
and recall that
\[
P_{11} =  z_1(1-z_1) (2Q + a_{11} z_1(1-z_1)).
\]
Thus for positivity in a specific case we need
\begin{gather}\label{detpossquare}
(2Q + a_{11} z_1(1-z_1))(2 Q + a_{22} z_2(1-z_2))-a_{12}^2z_1(1-z_1)z_2(1-z_2)>0,\\
\label{p11possquare}
(2Q + a_{11} z_1(1-z_1))>0.
\end{gather}
Now assume also that $p_1,p_2$ are positive. In particular, each of them is at
least one.  It is not hard to check that the left hand side of
\eqref{p11possquare} may be viewed as a second order polynomial in $z_1$ which
is positive at $z_1=0$ and $z_1=1$ and has positive derivative at both these
points as well.  Thus \eqref{p11possquare} easily follows. The left hand side
of \eqref{detpossquare} is a polynomial in $p_0, p_1,p_2, z_1$, and $z_2$.
Viewing this as a polynomial of degree $6$ in $p_0$, it is easy to check that
the coefficients of $p_0^3, p_0^4, p_0^5, p_0^6$ are all positive.  Using
carefully that $p_1,p_2 \in \bbz^+$ and $(z_1,z_2) \in (0,1) \times (0,1)$, it
is also a relatively straightforward exercise to check that the coefficients
of $p_0^0$ (the constant term w.r.t. $p_0$), $p_0$, and $p_0^2$ are positive.

We conclude that for $p_1,p_2 >0$ we have an explicit almost K\"ahler extremal
metric for each $p_0 \in \bbr^+$. Using computer aided algebra one may check
that for these examples it is not possible to satisfy
\eqref{integralitysquare}. Thus, the examples are non-K\"ahler, almost
K\"ahler.

If we combine the constructions above, for $p_1$ and $p_2$ even, with the result in Proposition \ref{c0} we have the following result:

\begin{proposition}
The smooth manifold $(\bbc\bbp^1 )^3$ admits uncountably many
explicit non-CSC, non-K\"ahler, extremal almost K\"ahler metrics.
\end{proposition}

\begin{remark}
Although the generalized Calabi and admissible constructions do not seem to
lend themselves readily to producing new explicit examples of e.g. smooth stage $3$ Bott
manifolds with extremal (integrable) K\"ahler metrics, direct calculations suggest that it is possible to produce orbifold
examples when the fiber is a non-trivial Hirzebruch surface. Indeed, one can 
combine the notion of Calabi toric K\"ahler metrics from
\cite{Leg09} with the generalized Calabi construction as introduced in
\cite{ACGT04}, to construct explicit examples of admissible extremal K\"ahler metrics on
such orbibundles whose fibers are Hirzebruch orbifolds.
\end{remark}

\subsection*{Acknowledgements} We thank the anonymous referee, whose careful
reading of the manuscript and detailed comments led to numerous improvements
to the paper.

\def\cprime{$'$} \def\cprime{$'$} \def\cprime{$'$} \def\cprime{$'$}
  \def\cprime{$'$} \def\cprime{$'$} \def\cprime{$'$} \def\cprime{$'$}
  \def\cdprime{$''$} \def\cprime{$'$} \def\cprime{$'$} \def\cprime{$'$}
  \def\cprime{$'$} \def\cprime{$'$}
\providecommand{\bysame}{\leavevmode\hbox to3em{\hrulefill}\thinspace}
\providecommand{\MR}{\relax\ifhmode\unskip\space\fi MR }
\providecommand{\MRhref}[2]{%
  \href{http://www.ams.org/mathscinet-getitem?mr=#1}{#2}
}
\providecommand{\href}[2]{#2}


\begin{thebibliography}{ACGT08b}

\bibitem[Abc13]{Abc13}
Anna Abczynski, \emph{On the classification of cohomology {B}ott manifolds},
  Ph.D. thesis, {R}heinischen {F}riedrich-{W}ilhelms-{U}niversit\"at, {B}onn,
  2013.

\bibitem[ACG06]{ApCaGa06}
Vestislav Apostolov, David M.~J. Calderbank, and Paul Gauduchon,
  \emph{Hamiltonian $2$-forms in {K}\"ahler geometry. {I}. {G}eneral theory},
  J. Differential Geom. \textbf{73} (2006), no.~3, 359--412. \MR{2228318
  (2007b:53149)}

\bibitem[ACGT04]{ACGT04}
Vestislav Apostolov, David M.~J. Calderbank, Paul Gauduchon, and Christina~W.
  T{\o nnesen-Friedman}, \emph{Hamiltonian $2$-forms in {K}\"ahler geometry.
  {II}. {G}lobal classification}, J. Differential Geom. \textbf{68} (2004),
  no.~2, 277--345. \MR{2144249}

\bibitem[ACGT08a]{ACGT08}
\bysame, \emph{Hamiltonian $2$-forms in {K}\"ahler geometry. {III}. {E}xtremal
  metrics and stability}, Invent. Math. \textbf{173} (2008), no.~3, 547--601.
  \MR{MR2425136 (2009m:32043)}

\bibitem[ACGT08b]{ACGT08b}
\bysame, \emph{Hamiltonian $2$-forms in {K}\"ahler geometry. {IV}. {W}eakly
  {B}ochner-flat {K}\"ahler manifolds}, Comm. Anal. Geom. \textbf{16} (2008),
  no.~1, 91--126. \MR{2411469 (2010c:32043)}

\bibitem[ACGT11]{ACGT11}
\bysame, \emph{Extremal {K}\"ahler metrics on projective bundles over a curve},
  Adv. Math. \textbf{227} (2011), 2385--2424.

\bibitem[AD03]{ApDr03}
Vestislav Apostolov and Tedi Draghici, \emph{The curvature and the
  integrability of almost-{K}\"ahler manifolds: a survey}, Symplectic and
  contact topology: interactions and perspectives ({T}oronto, {ON}/{M}ontreal,
  {QC}, 2001), Fields Inst. Commun., vol.~35, Amer. Math. Soc., Providence, RI,
  2003, pp.~25--53. \MR{1969266}

\bibitem[AGM93]{AsGrMo93}
Paul~S. Aspinwall, Brian~R. Greene, and David~R. Morrison, \emph{The
  monomial-divisor mirror map}, Internat. Math. Res. Notices (1993), no.~12,
  319--337. \MR{1253648 (95b:14029)}

\bibitem[Bat81]{Bat81}
Victor~V. Batyrev, \emph{Toric {F}ano threefolds}, Izv. Akad. Nauk SSSR Ser.
  Mat. \textbf{45} (1981), no.~4, 704--717, 927. \MR{631434}

\bibitem[Bat91]{Bat91}
\bysame, \emph{On the classification of smooth projective toric varieties},
  Tohoku Math. J. (2) \textbf{43} (1991), no.~4, 569--585. \MR{1133869}

\bibitem[Bat93]{Bat93}
\bysame, \emph{Quantum cohomology rings of toric manifolds}, Ast\'erisque
  (1993), no.~218, 9--34, Journ{\'e}es de G{\'e}om{\'e}trie Alg{\'e}brique
  d'Orsay (Orsay, 1992). \MR{1265307}

\bibitem[Bat99]{Baty99}
\bysame, \emph{On the classification of toric {F}ano {$4$}-folds}, J. Math.
  Sci. (New York) \textbf{94} (1999), no.~1, 1021--1050, Algebraic geometry, 9.
  \MR{2000e:14088}

\bibitem[BP15]{BuPa15}
Victor~M. Buchstaber and Taras~E. Panov, \emph{Toric topology}, Mathematical
  Surveys and Monographs, vol. 204, American Mathematical Society, Providence,
  RI, 2015. \MR{3363157}

\bibitem[BS58]{BoSa58}
Raoul Bott and Hans Samelson, \emph{Applications of the theory of {M}orse to
  symmetric spaces}, Amer. J. Math. \textbf{80} (1958), 964--1029. \MR{0105694
  (21 \#4430)}

\bibitem[BS99]{BatSe99}
Victor~V. Batyrev and Elena~N. Selivanova, \emph{Einstein-{K}\"ahler metrics on
  symmetric toric {F}ano manifolds}, J. Reine Angew. Math. \textbf{512} (1999),
  225--236. \MR{2000j:32038}

\bibitem[Cal82]{Cal82}
Eugenio Calabi, \emph{Extremal {K}\"ahler metrics}, Seminar on Differential
  Geometry, Ann. of Math. Stud., vol. 102, Princeton Univ. Press, Princeton,
  N.J., 1982, pp.~259--290. \MR{83i:53088}

\bibitem[CEMN16]{CEMN16}
David M.~J. Calderbank, Michael~G. Eastwood, Vladimir~S. Matveev, and Katharina
  Neusser, \emph{C-projective geometry}, arXiv:1512.04516v2 (2016), to appear in
  Mem. Amer. Math. Soc., 117 pp.

\bibitem[Cha17a]{Cha17b}
B.~Narasimha Chary, \emph{A note on toric degenerations of a
  {B}ott-{S}amelson-{D}emazure-{H}ansen variety}, preprint, arXiv:
  math.AG/1710.06300 (2017), 1--20.

\bibitem[Cha18]{Cha17a}
\bysame, \emph{On {M}ori cone of {B}ott towers}, J. Algebra \textbf{507}
  (2018), 467--501, arXiv:math.AG/1706.02139(v3). \MR{3807057}

\bibitem[Cho15]{Choi15}
Suyoung Choi, \emph{Classification of {B}ott manifolds up to dimension 8},
  Proc. Edinb. Math. Soc. (2) \textbf{58} (2015), no.~3, 653--659. \MR{3391366}

\bibitem[Civ05]{Civ05}
Yusuf Civan, \emph{Bott towers, crosspolytopes and torus actions}, Geom.
  Dedicata \textbf{113} (2005), 55--74. \MR{2171298}

\bibitem[CLS11]{CoLiSc11}
David~A. Cox, John~B. Little, and Henry~K. Schenck, \emph{Toric varieties},
  Graduate Studies in Mathematics, vol. 124, American Mathematical Society,
  Providence, RI, 2011. \MR{2810322 (2012g:14094)}

\bibitem[CM12]{ChMa12}
Suyoung Choi and Mikiya Masuda, \emph{Classification of {$\mathbb Q$}-trivial
  {B}ott manifolds}, J. Symplectic Geom. \textbf{10} (2012), no.~3, 447--461.
  \MR{2983437}

\bibitem[CMM15]{ChMaMu15}
Suyoung Choi, Mikiya Masuda, and Satoshi Murai, \emph{Invariance of
  {P}ontrjagin classes for {B}ott manifolds}, Algebr. Geom. Topol. \textbf{15}
  (2015), no.~2, 965--986. \MR{3342682}

\bibitem[CMS10]{ChMaSu10}
Suyoung Choi, Mikiya Masuda, and Dong~Youp Suh, \emph{Topological
  classification of generalized {B}ott towers}, Trans. Amer. Math. Soc.
  \textbf{362} (2010), no.~2, 1097--1112. \MR{2551516 (2011a:57050)}

\bibitem[CMS11]{ChMaSu11}
\bysame, \emph{Rigidity problems in toric topology: a survey}, Tr. Mat. Inst.
  Steklova \textbf{275} (2011), no.~Klassicheskaya i Sovremennaya Matematika v
  Pole Deyatelnosti Borisa Nikolaevicha Delone, 188--201. \MR{2962979}

\bibitem[Cox97]{Cox97}
David~A. Cox, \emph{Recent developments in toric geometry}, Algebraic
  geometry---{S}anta {C}ruz 1995, Proc. Sympos. Pure Math., vol.~62, Amer.
  Math. Soc., Providence, RI, 1997, pp.~389--436. \MR{1492541 (99d:14054)}

\bibitem[CR05]{CiRa05}
Yusuf Civan and Nigel Ray, \emph{Homotopy decompositions and {$K$}-theory of
  {B}ott towers}, $K$-Theory \textbf{34} (2005), no.~1, 1--33. \MR{2162899
  (2006g:55019)}

\bibitem[CS11]{ChSu11}
Suyoung Choi and Dong~Youp Suh, \emph{Properties of {B}ott manifolds and
  cohomological rigidity}, Algebr. Geom. Topol. \textbf{11} (2011), no.~2,
  1053--1076. \MR{2792373}

\bibitem[CvR09]{CoRe09}
David~A. Cox and Christine von Renesse, \emph{Primitive collections and toric
  varieties}, Tohoku Math. J. (2) \textbf{61} (2009), no.~3, 309--332.
  \MR{2568257}

\bibitem[{Dan}78]{Dan78}
Vladimir~I. {Danilov}, \emph{{Geometry of toric varieties.}}, {Russ. Math.
  Surv.} \textbf{33} (1978), no.~2, 97--154.

\bibitem[Del88]{Del88}
Thomas Delzant, \emph{Hamiltoniens p\'eriodiques et images convexes de
  l'application moment}, Bull. Soc. Math. France \textbf{116} (1988), no.~3,
  315--339. \MR{984900 (90b:58069)}

\bibitem[Dem70]{Dem70}
Michel Demazure, \emph{Sous-groupes alg\'ebriques de rang maximum du groupe de
  {C}remona}, Ann. Sci. \'Ecole Norm. Sup. (4) \textbf{3} (1970), 507--588.
  \MR{0284446}

\bibitem[Don02]{Don02}
Simon~K. Donaldson, \emph{Scalar curvature and stability of toric varieties},
  J. Differential Geom. \textbf{62} (2002), no.~2, 289--349. \MR{MR1988506
  (2005c:32028)}

\bibitem[Don08]{Don08a}
\bysame, \emph{K\"ahler geometry on toric manifolds, and some other manifolds
  with large symmetry}, Handbook of geometric analysis. {N}o. 1, Adv. Lect.
  Math. (ALM), vol.~7, Int. Press, Somerville, MA, 2008, pp.~29--75.
  \MR{2483362 (2010h:32025)}

\bibitem[GK94]{GrKa94}
Michael Grossberg and Yael Karshon, \emph{Bott towers, complete integrability,
  and the extended character of representations}, Duke Math. J. \textbf{76}
  (1994), no.~1, 23--58. \MR{1301185 (96i:22030)}

\bibitem[Gro91]{Gro91}
Michael Grossberg, \emph{Complete integrability and geometrically induced
  representations}, Ph.D. thesis, {M}assachusetts {I}nstitute of {T}echnology,
  1991.

\bibitem[Gua95]{Gua95}
Daniel Guan, \emph{Existence of extremal metrics on compact almost homogeneous
  {K}\"ahler manifolds with two ends}, Trans. Amer. Math. Soc. \textbf{347}
  (1995), no.~6, 2255--2262. \MR{1285992 (96a:58059)}

\bibitem[Gui94]{Gui94b}
V.~Guillemin, \emph{Kaehler structures on toric varieties}, J. Differential
  Geom. \textbf{40} (1994), no.~2, 285--309. \MR{1293656 (95h:32029)}

\bibitem[Hwa94]{Hwa94}
Andrew~D. Hwang, \emph{On existence of {K}\"ahler metrics with constant scalar
  curvature}, Osaka J. Math. \textbf{31} (1994), no.~3, 561--595. \MR{1309403
  (96a:53061)}

\bibitem[Ish12]{Ish12}
Hiroaki Ishida, \emph{Filtered cohomological rigidity of {B}ott towers}, Osaka
  J. Math. \textbf{49} (2012), no.~2, 515--522. \MR{2945760}

\bibitem[Kar03]{Kar03}
Yael Karshon, \emph{Maximal tori in the symplectomorphism groups of
  {H}irzebruch surfaces}, Math. Res. Lett. \textbf{10} (2003), no.~1, 125--132.
  \MR{1960129 (2004f:53101)}

\bibitem[Koi90]{Koi90}
Norihito Koiso, \emph{On rotationally symmetric {H}amilton's equation for
  {K}\"ahler-{E}instein metrics}, Recent topics in differential and analytic
  geometry, Adv. Stud. Pure Math., vol.~18, Academic Press, Boston, MA, 1990,
  pp.~327--337. \MR{1145263 (93d:53057)}

\bibitem[KS86]{KoSa86}
Norihito Koiso and Yusuke Sakane, \emph{Nonhomogeneous {K}\"ahler-{E}instein
  metrics on compact complex manifolds}, Curvature and topology of {R}iemannian
  manifolds ({K}atata, 1985), Lecture Notes in Math., vol. 1201, Springer,
  Berlin, 1986, pp.~165--179. \MR{859583 (88c:53047)}

\bibitem[Leg11]{Leg09}
Eveline Legendre, \emph{Toric geometry of convex quadrilaterals}, J. Symplectic
  Geom. \textbf{9} (2011), no.~3, 343--385. \MR{2817779 (2012g:53079)}

\bibitem[Lej10]{Lej10}
Mehdi Lejmi, \emph{Extremal almost-{K}\"ahler metrics}, Internat. J. Math.
  \textbf{21} (2010), no.~12, 1639--1662. \MR{2747965}

\bibitem[Lic58]{Lic58}
Andr\'e Lichnerowicz, \emph{G\'eom\'etrie des groupes de transformations},
  Travaux et Recherches Math\'ematiques, III. Dunod, Paris, 1958. \MR{0124009}

\bibitem[LS93]{LeSi93b}
Claude LeBrun and Santiago~R. Simanca, \emph{On the {K}\"ahler classes of
  extremal metrics}, Geometry and global analysis (Sendai, 1993), Tohoku Univ.,
  Sendai, 1993, pp.~255--271. \MR{MR1361191 (96h:58037)}

\bibitem[LT97]{LeTo97}
Eugene Lerman and Susan Tolman, \emph{Hamiltonian torus actions on symplectic
  orbifolds and toric varieties}, Trans. Amer. Math. Soc. \textbf{349} (1997),
  no.~10, 4201--4230. \MR{98a:57043}

\bibitem[Mab87]{Mab87}
Toshiki Mabuchi, \emph{Einstein-{K}\"ahler forms, {F}utaki invariants and
  convex geometry on toric {F}ano varieties}, Osaka J. Math. \textbf{24}
  (1987), no.~4, 705--737. \MR{89e:53074}

\bibitem[Mab14]{Mab14}
\bysame, \emph{Relative stability and extremal metrics}, J. Math. Soc. Japan
  \textbf{66} (2014), no.~2, 535--563. \MR{3201825}

\bibitem[McD11]{McD11}
Dusa McDuff, \emph{The topology of toric symplectic manifolds}, Geom. Topol.
  \textbf{15} (2011), no.~1, 145--190. \MR{2776842 (2012e:53172)}

\bibitem[MP08]{MaPa08}
Mikiya Masuda and Taras~E. Panov, \emph{Semi-free circle actions, {B}ott
  towers, and quasitoric manifolds}, Mat. Sb. \textbf{199} (2008), no.~8,
  95--122. \MR{2452268 (2009i:57073)}

\bibitem[Nak93]{Nak93}
Yasuhiro Nakagawa, \emph{Einstein-{K}\"ahler toric {F}ano fourfolds}, Tohoku
  Math. J. (2) \textbf{45} (1993), no.~2, 297--310. \MR{94f:32059}

\bibitem[Nak94]{Nak94}
\bysame, \emph{Classification of {E}instein-{K}\"ahler toric {F}ano fourfolds},
  Tohoku Math. J. (2) \textbf{46} (1994), no.~1, 125--133. \MR{95e:32030}

\bibitem[Oda88]{Oda88}
Tadao Oda, \emph{Convex bodies and algebraic geometry}, Ergebnisse der
  Mathematik und ihrer Grenzgebiete (3) [Results in Mathematics and Related
  Areas (3)], vol.~15, Springer-Verlag, Berlin, 1988, An introduction to the
  theory of toric varieties, Translated from the Japanese. \MR{922894
  (88m:14038)}

\bibitem[Sak86]{Sak86}
Yusuke Sakane, \emph{Examples of compact {E}instein {K}\"ahler manifolds with
  positive {R}icci tensor}, Osaka J. Math. \textbf{23} (1986), no.~3, 585--616.
  \MR{866267 (87k:53111)}

\bibitem[SS11]{StSz11}
Jacopo Stoppa and G\'abor Sz\'ekelyhidi, \emph{Relative {K}-stability of
  extremal metrics}, J. Eur. Math. Soc. (JEMS) \textbf{13} (2011), no.~4,
  899--909. \MR{2800479}
  
\bibitem[Suy18]{Suy18}
Yusuke Suyama, \emph{Fano generalized {B}ott manifolds}, arXiv:1811.06209
  (2018).  

\bibitem[WW82]{WaWa82}
Keiichi Watanabe and Masayuki Watanabe, \emph{The classification of {F}ano
  {$3$}-folds with torus embeddings}, Tokyo J. Math. \textbf{5} (1982), no.~1,
  37--48. \MR{670903}

\bibitem[WZ04]{WaZh04}
Xu-Jia Wang and Xiaohua Zhu, \emph{K\"ahler-{R}icci solitons on toric manifolds
  with positive first {C}hern class}, Adv. Math. \textbf{188} (2004), no.~1,
  87--103. \MR{MR2084775 (2005d:53074)}

\bibitem[WZ11]{WaZh11}
Xu-jia Wang and Bin Zhou, \emph{On the existence and nonexistence of extremal
  metrics on toric {K}\"ahler surfaces}, Adv. Math. \textbf{226} (2011), no.~5,
  4429--4455. \MR{2770455}

\bibitem[WZ14]{WaZh14}
\bysame, \emph{{$K$}-stability and canonical metrics on toric manifolds}, Bull.
  Inst. Math. Acad. Sin. (N.S.) \textbf{9} (2014), no.~1, 85--110. \MR{3234970}

\bibitem[ZZ08]{ZhZh08a}
Bin Zhou and Xiaohua Zhu, \emph{Relative {$K$}-stability and modified
  {$K$}-energy on toric manifolds}, Adv. Math. \textbf{219} (2008), no.~4,
  1327--1362. \MR{2450612}

\end{thebibliography}
\end{document}